\documentclass[11pt]{amsart}
\usepackage[margin=1in]{geometry}  
\usepackage{amscd,latexsym,amsthm,amsfonts,amssymb,amsmath,amsxtra}
\usepackage{mathrsfs}
\usepackage[mathscr]{eucal}
\usepackage{hyperref}
\usepackage{pdfsync}
\usepackage[all,cmtip]{xy}

\numberwithin{equation}{section}
\newtheorem{thm}{Theorem}[section]
\newtheorem{cor}[thm]{Corollary}
\newtheorem{prop}[thm]{Proposition}
\newtheorem{lemma}[thm]{Lemma}

\newtheorem{rk}[thm]{Remark}

\newtheorem{defn}[thm]{Definition}

\newcommand{\BA}{{\mathbb {A}}}

\newcommand{\BC}{{\mathbb {C}}}

\newcommand{\BK}{{\mathbb {K}}}

\newcommand{\BQ}{{\mathbb {Q}}}
\newcommand{\BR}{{\mathbb {R}}}

\newcommand{\BZ}{{\mathbb {Z}}}

\newcommand{\CE}{{\mathcal {E}}}

\newcommand{\CG}{{\mathcal {G}}}

\newcommand{\CL}{{\mathcal {L}}}

\newcommand{\CO}{{\mathcal {O}}}
\newcommand{\CP}{{\mathcal {P}}}

\newcommand{\CS}{{\mathcal {S}}}

\newcommand{\CW}{{\mathcal {W}}}
\newcommand{\CX}{{\mathcal {X}}}

\newcommand{\RH}{{\mathrm {H}}}

\renewcommand{\Re}{{\mathrm {Re}}}

\newcommand{\fa}{\mathfrak{a}}
\newcommand{\fb}{\mathfrak{b}}
\newcommand{\fc}{\mathfrak{c}}

\newcommand{\fg}{\mathfrak{g}}
\newcommand{\fh}{\mathfrak{h}}
\newcommand{\fk}{\mathfrak{k}}
\newcommand{\fl}{\mathfrak{l}}
\newcommand{\fm}{\mathfrak{m}}

\newcommand{\fq}{\mathfrak{q}}

\newcommand{\fs}{\mathfrak{s}}
\newcommand{\ft}{\mathfrak{t}}
\newcommand{\fu}{\mathfrak{u}}

\newcommand{\fgl}{{\mathfrak{gl}}}

\newcommand{\GL}{{\mathrm{GL}}}

\newcommand{\quo}{\backslash}

\newcommand{\Hom}{{\mathrm{Hom}}}

\newcommand{\Ind}{{\mathrm{Ind}}}
\newcommand{\End}{{\mathrm{End}}}
\newcommand{\Aut}{{\mathrm{Aut}}}

\newcommand{\Ad}{{\mathrm{Ad}}}

\newcommand{\Mat}{{\mathrm{Mat}}}
\newcommand{\sgn}{{\mathrm{sgn}}}

\newcommand{\proten}{\hat{\otimes}}

\newcommand{\tr}{{\mathrm{tr}}}

\newcommand{\Sym}{{\mathrm{Sym}}}

\newcommand{\ct}{{\mathrm{ct}}}

\newcommand{\abs}[1]{\lvert#1\rvert}
\newcommand{\set}[2]{\left\{#1\,|\,#2\right\}}

\newcommand{\mtrtwo}[4]{\begin{bmatrix} #1 &#2 \\#3 &#4 \end{bmatrix}}

\renewcommand{\bar}{\overline}
\renewcommand{\tilde}{\widetilde}
\newcommand{\eps}{\epsilon}
\newcommand{\Rk}{\mathrm{k}}

\begin{document}


\title[Period Relations]{Period relations for standard $L$-functions of symplectic type}

\author{Dihua Jiang}
\address{School of Mathematics, University of Minnesota, Minneapolis, MN 55455, USA}
\email{dhjiang@math.umn.edu}

\author{Binyong Sun}
\address{Institute for Advanced Study in Mathematics \& New Cornerstone Science Laboratory, Zhejiang University,  Hangzhou, 310058, China}
\email{sunbinyong@zju.edu.cn}

\author{Fangyang Tian}
\address{School of Mathematical Sciences, Zhejiang University,  Hangzhou, 310058, China}
\email{tianfangyangmath@zju.edu.cn}





\thanks{
 The research of D. Jiang is supported in part by the NSF grants DMS--1901802 and DMS-2200890; that of B. Sun is supported by the National Key R \& D Program of China (Nos. 2022YFA1005300 and 2020YFA0712600) and the New Cornerstone Investigator Program; and that of F. Tian
is supported in part by  Natural Sicence Foundation of China grants 507100-P22203.}

\begin{abstract}
This article is to establish the period relations, as conjectured by D. Blasius in \cite{Bl97}, for the critical values of the standard 
$L$-functions $L(s,\Pi\otimes \chi)$ in general with $\Pi$ an irreducible regular algebraic cuspidal automorphic representation of 
$\GL_{2n}(\BA)$ of symplectic type and $\chi$ a finite order automorphic character of $\GL_1(\BA)$, where $\BA$ is the ring of adeles of a number field $\Rk$. The main results extend the classical work of G. Shimura (\cite{Sh76}) and 
Y. Manin (\cite{Mn76}) from $\GL_2$ to $\GL_{2n}$, and complete the approach that was initially taken by H. Grobner and A. Raghuram in \cite{GrobnerRaguhramShalikaModelArithmetic}  and by F. Januszewski in \cite{Jan18}, where a few technical assumptions were taken. In this article, those technical assumptions are removed by using different arguments. 
\end{abstract}

\maketitle
\tableofcontents


\section{Introduction}\label{section: Introduction}

A pioneering work of G. Shimura (\cite{Sh59}, \cite{Sh71} and \cite{Sh76}) indicates the arithmetic structure of the critical values of $L$-functions associated to elliptic modular forms and
the relation between different critical values of those $L$-functions.
These phenomena are expected, as conjectured by D. Blasius in \cite{Bl97}, for standard $L$-functions associated to algebraic cuspidal automorphic representations of $\GL_n$ over any number field, in the
sense of L. Clozel (\cite{ClozelAutomorphicFormsMotive}). On the other hand, a celebrated conjecture of P. Deligne (\cite[Conjecture 2.8]{D79}) predicts the same phenomena for critical values of $L$-series (or $L$-functions)
attached to motives. The underlying intrinsic structure that reflects the similarity of the parallel phenomena between the theory of algebraic automorphic representations and the theory of motives is the far-reaching
conjecture of R. Langlands (\cite{L79}). As further explained in \cite[Section 4.3]{ClozelAutomorphicFormsMotive}, this conjecture asserts that there exists a one-to-one correspondence between irreducible motives over a number field with coefficients
in $\bar{\BQ}$ and irreducible algebraic cuspidal automorphic representations of general linear groups, which respects $L$-functions 
(\cite[Conjecture 4.5]{ClozelAutomorphicFormsMotive}).

 The objective of this paper is to establish the global period relations for the critical values of the twisted standard $L$-functions associated to irreducible regular algebraic cuspidal automorphic representations of $\GL_{2n}$ that are of generalized symplectic type, for any integer $n\geq 1$. 
 The approach that we are undertaking in this paper is to study the generalized modular symbols in the sense of
B. Mazur (\cite{Mz}, \cite{AshNonSquareIntegrableCohomologyArithmeticGroups} and \cite{AB90}), and generalizes the classical approach of Shimura (\cite{Sh76}) and Y. Manin (\cite{Mn76}) from $\GL_2$ to $\GL_{2n}$ for any integer $n\geq 1$. The approach was initially taken by H. Grobner and A. Raghuram in \cite{GrobnerRaguhramShalikaModelArithmetic}  and by F. Januszewski in \cite{Jan18}, where a few technical assumptions were taken. In this article, those technical assumptions are removed by using different arguments. 

Let $\BA=\BA_\Rk$ be the ring of adeles of a number field $\Rk$.
We write $\BA=\BA_\infty\times\BA_f$, with
$\BA_f$ being the ring of finite adeles of $\Rk$ and $\BA_\infty := \Rk\otimes_\BQ \BR$ being the archimedean part. For a local place $v$ of $\Rk$, we write $v\mid\infty$ if $v$ is an archimedean place and
$v\nmid\infty$ if otherwise. We also write $\Rk_v$ for the local field  at a place $v$.
Let
\begin{equation}\label{decpi}
\Pi=\widehat\otimes_v' \Pi_v:=\Pi_\infty\otimes \Pi_f
\end{equation}
be an irreducible cuspidal automorphic representation of $\GL_m(\BA)$ ($m\geq 1$),
where 
$\Pi_\infty :=\widehat \otimes_{v\mid\infty} \Pi_v$ is a completed projective tensor product
and 
$\Pi_f := \otimes_{v\nmid\infty}' \Pi_v$ is the restricted tensor product. 
Here $\Pi_v$ is an irreducible  smooth representation of $\GL_{m}(\Rk_v)$ when $v\nmid\infty$ and $\Pi_v$ is an irreducible Casselman-Wallach representation of $\GL_{m}(\Rk_v)$ when $v\mid\infty$.
We refer to \cite{CasselmanCanonicalExtensionofHarishChandraModule}, \cite[Chapter 11]{WallachRealReductiveGroups2} and \cite{BernsteinKrotzSmoothGlobalization} for the definition and
some basic properties of Casselman-Wallach representations.
Then $\Pi_\infty$ and $\Pi_f$ are representations of $\GL_{m}(\BA_\infty)$ and $\GL_{m}(\BA_f)$, respectively.
When $m=1$, we write any automorphic character $\chi \colon \Rk^\times \quo \BA^\times\rightarrow \BC^\times$ as 
$\chi=\otimes_v'\chi_v=\chi_\infty\otimes \chi_f$.

 The (twisted) standard local $L$-factor $L(s, \Pi_v\otimes\chi_v)$ is defined in \cite[Theorem 3.3]{GJ72} by R. Godement and H. Jacquet. The complete twisted standard $L$-function is given by the following Eulerian product 
$$
L(s, \Pi\otimes\chi) := L(s, \Pi_f\otimes \chi_f)\cdot L(s, \Pi_\infty\otimes \chi_\infty),
$$
where the finite part and the infinite part are defined to be 
$L(s, \Pi_f\otimes \chi_f) := \prod_{v\nmid\infty} L(s, \Pi_v\otimes \chi_v)
$
and 
$L(s, \Pi_\infty\otimes \chi_\infty) := \prod_{v\mid\infty} L(s, \Pi_v\otimes \chi_v)$, respectively.
The complete $L$-function
$L(s, \Pi\otimes\chi)$ is meromorphic (entire when $m\geq 2$) as a function in $s\in\BC$ and satisfies the functional equation
$$
L(s,\Pi\otimes\chi)=\varepsilon(s,\Pi\otimes\chi)\cdot L(1-s,\Pi^\vee\otimes\chi^{-1}).
$$
Here and henceforth $\ ^\vee$ indicates the contragredient representation.

It is expected that ``generically" irreducible motives  are essentially self-dual (\cite[Section 2.3]{ClozelMotivesandAutomorphicLfun}). Hence it is natural to consider the case when $\Pi$ is essentially self-dual, in the sense that there exists an automorphic character $\eta:\Rk^\times\quo\BA^\times\rightarrow\BC^\times$ such that $\Pi^\vee\cong\Pi\otimes\eta^{-1}$.
It is a well-known theorem that the complete Rankin-Selberg $L$-function $L(s, \Pi\otimes\Pi^\vee)$ is holomorphic for $\Re(s)>1$ and has a simple pole at $s=1$ (see \cite[Page 667]{MW89},
and also \cite[Theorem 2.4]{CPS04}, \cite[Part II, Proposition 3.3, 3.6]{JacquetShalikaEulerProduct1}).
Following a non-vanishing result of F. Shahidi (\cite[Theorem 1.1]{Sh97}), with the classification of the unitary dual of general linear group over local fields of characteristic zero (\cite{V86} and \cite{T86}),
it is easy to show that
either the twisted exterior $L$-function $L(s, \Pi, \wedge^2\otimes\eta^{-1})$ or the twisted symmetric $L$-function
$L(s, \Pi, \Sym^2\otimes\eta^{-1})$ has a simple pole, but not both. Accordingly, this classifies the essentially self-dual irreducible cuspidal automorphic representations $\Pi$ of $\GL_m(\BA)$ as (generalized) symplectic type (with $m$ being even) or
(generalized) orthogonal type (with $m$ being even or odd).

In this paper we consider cuspidal automorphic representations of $\GL_m(\BA)$ with the following three conditions:
     \begin{enumerate}
       \item $\Pi$ is regular and algebraic in the sense of L. Clozel (\cite{ClozelAutomorphicFormsMotive}), which will be recalled in Section \ref{algssec-nd};
             \item $\Pi$ is of symplectic type, see Definition \ref{def: global repn of symplectic type}, which implies that $m=2n$ must be even;
              \item $\Pi$ has a balanced coefficient system, see Definition \ref{defn: balanced coeffi system}.     \end{enumerate}
By Definition \ref{def: global critical place}, for such a $\Pi$,
the critical places of $\Pi$ are of the form $\frac{1}{2}+j$ for $j\in\BZ$. As in \cite{GrobnerRaguhramShalikaModelArithmetic}, we have a family $\{\omega^{\epsilon_\infty}(\Pi)\}_{\epsilon_\infty}$ of non-zero complex numbers, to be called the {\bf 
Shalika periods} of $\Pi$, where $\epsilon_\infty$ runs over all quadratic characters of $\BA_\infty^\times$. See  \eqref{eq: def of Omega Pi}. 
The following is the main result of this paper.

\begin{thm}[Blasius Conjecture]\label{thm: global period relation}
         Let $\Pi$ be an irreducible regular algebraic cuspidal automorphic representation of $\GL_{2n}(\BA)$ $(n\geq 1$) and let $\eta: \Rk^\times\quo\BA^\times\rightarrow \BC^\times$ be a character such that the complete twisted exterior square $L$-function $L(s, \Pi, \wedge^2\otimes \eta^{-1})$ has a pole at $s=1$. Assume that the coefficient system of $\Pi$ is balanced. Then 
 \begin{equation}\label{eq: main thm complete}  
 \frac{L(\frac{1}{2}+j,\Pi\otimes \chi)}{\mathrm{i}^{ jn\cdot  [\Rk:\BQ]}\cdot \CG(\chi)^n \cdot \omega^{\sgn^j\cdot\chi_\infty}(\Pi)}\in \BQ(\Pi, \eta, \chi),\end{equation}
 for every critical place $\frac{1}{2}+j$ of $\Pi$ and
 every finite order character $\chi=\chi_\infty\otimes \chi_f: \Rk^\times\backslash \BA^\times \rightarrow \BC^\times$. 
 Here $\mathrm{i}=\sqrt{-1}$, $\CG(\chi)$ is the Gauss sum of $\chi$ defined in \eqref{eq: gauss sum}, and $\sgn$ is the quadratic character on $\BA_\infty^\times$ that has non-trivial restriction to $\Rk_v^\times$ for every real place $v$ of $\Rk$. 
\end{thm}


 The number field $\BQ(\Pi, \eta, \chi)$
is called the {\bf rationality field} of $\Pi$, $\eta$ and $\chi$, and will be defined in \eqref{eq: number field Q(pi eta chi)}.
 We have more remarks in order.

    \begin{rk}
At least when $\Rk$ contains no CM field,  the coefficient system of $\Pi$ is balanced if and only if $\Pi$ has a critical place. In general, if the coefficient system of $\Pi$ is balanced, then $\Pi$ has at least one critical place. See Remark \ref{rk217} for more comments on this issue.
\end{rk}

    \begin{rk}
The automorphic character $\eta$ is automatically algebraic. It is expected that in general, the rationality field $\BQ(\eta)$ is contained in the rationality field $\BQ(\Pi)$. When $n=1$, this is clear as $\eta$ equals the central character of $\Pi$.
\end{rk}

    \begin{rk}
It is not hard to deduce the period relation for the finite part of the twisted standard $L$-functions, which can be stated as follows.
Let the notation and the assumptions be as in Theorem \ref{thm: global period relation}. Then for every critical place $\frac{1}{2}+j$ of $\Pi$ and
every finite order automorphic character $\chi=\chi_\infty\otimes \chi_f: \Rk^\times\backslash \BA^\times \rightarrow \BC^\times$  we have that
         \begin{equation}\label{eq: main thm finite}  \frac{L(\frac{1}{2}+j,\Pi_f\otimes \chi_f)}{(2\pi\mathrm{i})^{ jn\cdot  [\Rk\, :\, \BQ]}\cdot \CG(\chi)^n \cdot \omega^{\sgn^j\cdot\chi_\infty}(\Pi) \cdot \Omega_{\Pi_\infty}} \in \BQ(\Pi, \eta, \chi),\end{equation}
         where $\Omega_{\Pi_\infty} :=\prod_{v\mid \infty}\Omega_{\Pi_{\mathrm k_v}}$, and          
         $\Omega_{\Pi_{\mathrm k_v}}$ is defined as in \eqref{omegapik}. 
See  Remark \ref{rk: finite L-fun arithmetic property proof} for details.
    \end{rk}

\begin{rk}
Grobner and Raghuram study the critical $L$-values $L(\frac{1}{2}+j,\Pi_f\otimes\chi_f)$ in \cite{GrobnerRaguhramShalikaModelArithmetic} over a totally real field. One of their main results
is \cite[Theorem 7.1.2]{GrobnerRaguhramShalikaModelArithmetic}, which asserts that the critical $L$-value $L(\frac{1}{2}+j,\Pi_f\otimes\chi_f)$ is equal to
     $
    \omega^{(-1)^{j+n-1}\epsilon_{\chi}}(\Pi_f)\CG(\chi)^n\omega(\Pi_\infty,j),
     $
up to a multiple of a number belonging to rationality field $\BQ(\Pi,\eta,\chi)$. 
This case was also considered earlier in \cite{AshGinzburgPadicLfunGL(2n)}. The explicit value of $\omega(\Pi_\infty,j)$ is not calculated in \cite{GrobnerRaguhramShalikaModelArithmetic}. 

\end{rk}

\begin{rk}
In \cite{Jan18}, Januszewski studies  the case 
over a totally real number field and asserts that
$L(\frac{1}{2}+j,\Pi_f\otimes\chi_f)$ is equal to
$
(2\pi\mathrm{i})^{jn[k:\BQ]}\CG(\chi)^n\Omega_{(-1)^{j}\sgn\chi},
$
up to a multiple of a number belonging to rationality field $\BQ(\Pi,\chi)$. 
His argument towards this global period relations
assumes in \cite[Section 5.3]{Jan18} the existence of the cohomological test vector for $\Pi_\infty$ that is independent of the critical places. As the most technical ingredient of this paper,
our Theorem \ref{thm: uniform coh test vector} establishes the existence of the cohomological test vector for $\Pi_\infty$ that is independent of the critical places for general number fields.
Meanwhile, our proof involves no rational structures of Harish-Chandra modules, which is however a major ingredient in the argument developed by Januszewski.
\end{rk}

\begin{rk}
In \cite{AshGinzburgPadicLfunGL(2n)}, \cite{GrobnerRaguhramShalikaModelArithmetic} and \cite{Jan18}, the base field is assumed to be totally real so that a certain numerical coincidence called ``degree matching" is satisfied. As far as we know, the proof of Theorem \ref{thm: global period relation} is the first case in the study of special values of $L$-functions via modular symbols without the ``degree matching" condition.
\end{rk}

\begin{rk}
 For $n=1$, all irreducible cuspidal automorphic representations of $\GL_2(\BA)$ are essentially self-dual of symplectic type. In this case, the theorem was considered by Manin in \cite{Mn76}, by Shimura in \cite[Theorem 1]{Sh76} and \cite[Theorem 4.3]{Sh78}, by G. Harder in \cite[Page 84]{Hd83}, and by H. Hida in \cite[Theorem 8.1]{Hi94}.
In \cite[Section L.9.1]{Bl97}, Blasius gives an explicit discussion on the conjecture of period relations and critical values of $L$-functions attached to cuspidal automorphic forms in general.
We also refer to the survey paper by M. Harris and J. Lin in \cite{HL17} for further explanation of the recent progress in the theory.
\end{rk}

\begin{rk}
 In the case of Rankin-Selberg convolutions for $\GL_n\times \GL_{n-1}$, some partial or conditional results for the period relations are obtained in \cite{Jan19}, \cite{HR20} and \cite{GL20}. The complete results for this case were obtained in \cite{LLS22}. 
\end{rk}

The twisted standard $L$-function $L(s,\Pi\otimes\chi)$ has been represented by three different types of global zeta integrals: one by 
Godement and Jacquet (\cite[Formula 12.4]{GJ72}), one by Jacquet and J. Shalika (\cite[Part II, Page 795]{JacquetShalikaEulerProduct1}),
and one by S. Friedberg and Jacquet (\cite[Proposition 2.3]{FriedbergJacquetLinearPeriods}). It is the global zeta integral of Friedberg-Jacquet in \cite{FriedbergJacquetLinearPeriods} that leads to the proof of Theorem \ref{thm: global period relation}. Because the
global zeta integral of Friedberg-Jacquet assumes that $\Pi$ has a nonzero Shalika integral, which leads to the assumption in Theorem \ref{thm: global period relation} that $\Pi$ is of symplectic type (Proposition \ref{prop: strong transfer}). Meanwhile, the Fiedberg-Jacquet global zeta integral also represents a linear period of $\Pi$ on $\GL_{2n}(\BA)$ over the subgroup $\GL_n(\BA)\times\GL_n(\BA)$.
When $\Pi$ is cohomological, the linear period has an interpretation in terms of (generalized) modular symbols (\cite{Mz}, \cite{AshNonSquareIntegrableCohomologyArithmeticGroups} and \cite{AB90}).
Using the modular symbols, the proof of Theorem \ref{thm: global period relation} relies  on the  global rational structure of the relevant cohomological groups of $\Pi$, and on three local results
that are responsible to the appearance of $\mathrm{i}^{ jn\cdot  [\Rk\, :\, \BQ]}$, $\CG(\chi)^n$ and $\omega^{\sgn^j\cdot\chi_\infty}(\Pi)$, respectively. The global rational structure
follows  from the assumption that $\Pi$ is regular and algebraic, due to the work of Clozel (\cite{ClozelAutomorphicFormsMotive}). The three local results are: Theorem \ref{thm: archi period relations} for relations of archimedean modular symbols at different critical places;  Theorem \ref{0lemma: Z circ f rational} for the rationality of the non-archimedean   Friedberg-Jacquet zeta integrals, and
Theorem \ref{thm: Non-vanishing hypothesis} for non-vanishing of archimedean modular symbols. Theorem \ref{thm: uniform coh test vector}, which asserts the existence of  uniform cohomological test vectors for archimedean Friedberg-Jacquet zeta integrals, is vital for the proofs of Theorems \ref{thm: Non-vanishing hypothesis} and  \ref{thm: archi period relations}. 

In Section \ref{sec-PPThm}, we recall and review basic notions and known facts that are needed for the statement and the proof of Theorem \ref{thm: global period relation}. With the preparation in Section \ref{sec-PPThm}, we are able to
state four local results in Section \ref{sectionfourt}, which forms the technical core of the paper. We prove Theorem  \ref{0lemma: Z circ f rational} in Section \ref{sectionfourt}.
The three archimedean local results (Theorem \ref{thm: uniform coh test vector}, Theorem \ref{thm: Non-vanishing hypothesis}, and Theorem \ref{thm: archi period relations}) will be proved in
Sections \ref{section: Proof of Critical place and Non-vanishing Modular Symbol} and
\ref{section: Identification}. Based on Sections \ref{sec-PPThm} and  \ref{sectionfourt}, Theorem \ref{thm: global period relation} 
is proved in Section \ref{subsection: proof of global period relation}.

The research work of the paper was partly carried out during the participation of the three authors in one month Research Program: {\sl On the Langlands Program: Endoscopy and Beyond} (17 Dec 2018 - 18 Jan 2019), at Institute for Mathematical Sciences, National University of Singapore. We would like to thank the IMS, NUS, for invitation and for hospitality. Over the years, we make a few improvements on the original version based on comments from D. Blasius and L. Clozel. The paper is finalized when the first named author 
enjoys his visit in The Institute for Advanced Study in Mathematics, Zhejiang University. We would like to thank D. Blasius and L. Clozel for their helpful comments on the preliminary version of this paper.



\section{Preliminaries }\label{sec-PPThm}


We start with some relevant notions and results that are needed for Theorem \ref{thm: global period relation}. 
Let $\Rk$ be a number field and $\BA$ be its ring of adeles.  Fix a non-trivial additive character
$\psi=\otimes_v' \psi_v:\Rk\backslash \BA\rightarrow \BC^\times$, where $\psi_v$ is a character of the local completion  $\Rk_v$ of $\Rk$ at a place $v$. Note that $\psi_v$ must be unitary and non-trivial. Let $\chi=\otimes'_{v} \chi_v: \Rk^\times\backslash  \BA^\times \rightarrow \BC^\times$ be an automorphic character.
Let $\BK$ be a local field of characteristic zero. Fix a non-trivial unitary character
$
\psi_\BK: \BK\rightarrow \BC^\times
$.
 Let $\chi_\BK: \BK^\times \rightarrow \BC^\times $ be a character.


\subsection{Rationality fields} \label{subsection: Rationality Field}
We recall the notion of rationality field from \cite[Section 3]{ClozelAutomorphicFormsMotive}. Suppose that $G_0$ is a totally disconnected, locally compact Hausdorff topological group. Given an  irreducible smooth representation $\Pi_0$ of $G_0$, we define the rationality field of $\Pi_0$, denoted by $\BQ(\Pi_0)$, to be the fixed field of the group of field automorphisms $\sigma\in\Aut(\BC)$ such that
$\Pi_0\otimes_{\BC, \sigma} \BC\cong\Pi_0$ as representations of $G_0$.
Given an irreducible cuspidal automorphic representation $\Pi=\Pi_\infty\otimes \Pi_f$ of $\GL_m(\BA)$ as in the Introduction, we define its rationality field to be
$\BQ(\Pi):=\BQ(\Pi_f)$. 
In this way, we define the rationality field $\BQ(\chi_\BK)$ when $\BK$ is non-archimedean and  the rationality field $\BQ(\chi)$.   We define
   \begin{equation}\label{eq: number field Q(pi eta chi)}
\BQ(\Pi,\chi) := \BQ(\Pi)\BQ(\chi),\end{equation}
   the compositum of $\BQ(\Pi)$ and $\BQ(\chi)$. Similar notation for the compositum fields will be used without further explanation.

Let $T_0$ be an algebraic torus defined over $\BQ$. Given an algebraic character $\chi_0: T_0(\BC)\rightarrow \BC^\times$, we define its rationality field,  denoted by $\BQ(\chi_0)$, to be the fixed field of the group of field automorphisms $\sigma\in\Aut(\BC)$ such that
$\chi_0(\sigma.x)=\sigma(\chi_0(x))$ for all $x\in T_0(\BC)$. Then $\BQ(\chi_0)$ is always a number field.

Let $F_0$ be an irreducible algebraic representation of $\GL_m(\Rk\otimes_\BQ \BC)$ ($m\geq 1$).
 Suppose that $T_0$ is the Weil restriction of $(\GL_1/_\Rk)^m$ from $\Rk$ to $\BQ$. Then an extremal weight of $F_0$ is an algebraic character of $T_0(\BC)$. We define the rationality field of $F_0$, which is denoted by $\BQ(F_0)$, to be the rationality field of  an extremal weight of $F_0$. This definition is independent of the choice of the extremal weight.


  \subsection{Gauss sums}\label{subsection: gauss}

We assume that $\BK$ is non-archimedean and define the local Gauss sum $\CG(\chi_\BK)$ of the character $\chi_\BK$ as follows.  Let $\fc(\chi_\BK)$ and $\fc(\psi_\BK)$ be the conductors of $\chi_\BK$ and $\psi_\BK$, respectively. Then $\fc(\chi_\BK)$ is a nonzero ideal of the ring $\CO_\BK$ of integers of $\BK$, and $\fc(\psi_\BK)$ is a fractional ideal of $\CO_\BK$. Fix  $y_\BK\in \BK^\times$  such that
$\fc(\psi_\BK)=y_\BK\cdot \fc(\chi_\BK)$. 
The local Gauss sum is defined to be
   $$\CG(\chi_\BK):=\CG(\chi_\BK, \psi_\BK, y_\BK) := \int_{\CO_\BK^\times} \chi_\BK(x)^{-1}\cdot \psi_\BK(y_\BK x) \, d x,$$
   where $d x$ is the normalized Haar measure so that $\CO_\BK^\times$ has total volume $1$.
Note that $\CG(\chi_\BK)$ is always nonzero and equals $1$ when $\chi_\BK$ is unramified.

For the automorphic character $\chi=\otimes'_{v} \chi_v$, fix  $y = \{y_v\}_{v\nmid \infty}\in \BA_f^\times$ such that
$\fc(\psi_v)=y_v\cdot \fc(\chi_v)$ for all $v\nmid \infty$.
The Gauss sum $\CG(\chi)$  is defined to be the product of local Gauss sums:
\begin{equation}\label{eq: gauss sum}\CG(\chi) := \CG(\chi, \psi, y):=\prod_{v\nmid\infty} \CG(\chi_v, \psi_v, y_v).\end{equation}
This product is indeed a finite product and is nonzero.

   \begin{rk}
       It is clear that the definition of Gauss sum $\CG(\chi)$ may depend on the choice of $\psi$ and $y = \{y_v\}_{v\nmid \infty}$. 
       This dependency will not effect the essence of Theorem \ref{thm: global period relation}, because with different choices of $\psi$ and $y = \{y_v\}_{v\nmid \infty}$, the Gauss sums only differ by multiplication of a number in $\BQ(\chi)^\times$.
   \end{rk}

\subsection{Cuspidal representations of symplectic type}\label{ssec-nd}
As usual, write $\mathrm{GSpin}_{2n+1}$ for the split odd spin similitude group of rank $n$ ($n\geq 1$). Its complex dual group is the symplectic similitude group $\mathrm{GSp}_{2n}(\BC)$.

    \begin{defn}\label{def:LT}
       Let $\Sigma = \widehat \otimes_v' \Sigma_v$ be an irreducible generic cuspidal automorphic representation of $\mathrm{GSpin}_{2n+1}(\BA)$  and let $\Pi = \widehat \otimes_v'\Pi_v$ be an irreducible cuspidal automorphic representation of
$\GL_{2n}(\BA)$. We say that $\Pi$ is the functorial transfer of $\Sigma$ if  at all but finitely many non-archimedean local places $v$ where $\Sigma_v$ and $\Pi_v$ are unramified,
       the Langlands parameter $ \CW_{\Rk_v} \rightarrow \GL_{2n}(\BC)$ for $\Pi_v$ is equal to the composition
       $$\CW_{\Rk_v}\xrightarrow{\phi_v} \mathrm{GSp}_{2n}(\BC) \hookrightarrow \GL_{2n}(\BC),$$
       up to conjugation by $\GL_{2n}(\BC)$, where $\CW_{\Rk_v}$ is the local Weil group of $\Rk_v$, and $\phi_v$ is the Langlands parameter of $\Sigma_v$.
    \end{defn}
   \begin{defn}\label{def: global repn of symplectic type}
          An irreducible cuspidal automorphic representation $\Pi$ of $\GL_{2n}(\BA)$ $(n\geq 1)$ is said to be of symplectic type if it is the functorial transfer of an irreducible generic cuspidal automorphic representation  of $\mathrm{GSpin}_{2n+1}(\BA)$, in the sense of Definition \ref{def:LT}.
     \end{defn}

We  write $G:=\GL_{2n}$ and define the Shalika subgroup of $G$ to be
   \begin{equation}\label{eq: global shalika subgroup}
         S :=S_n:= \left\{\mtrtwo{h}{0}{0}{h}\mtrtwo{1_n}{x}{0}{1_n}\mid h\in \GL_n, x\in\Mat_n\right\}
   \end{equation}
where $\Mat_n$ is the algebra of $n\times n$ matrices, $1_n$ is the $n\times n$ identity matrix. For an automorphic  character $\eta: \Rk^\times\quo \BA^\times \rightarrow \BC^\times$, denote by $\eta\otimes\psi$  the character of $S(\BA)$ sending $\mtrtwo{h}{0}{0}{h}\mtrtwo{1_n}{x}{0}{1_n}$ to $\eta(\det(h))\psi(\mathrm{tr}(x))$.
For an irreducible cuspidal automorphic representation $\Pi$ of $G(\BA)$, we may identify $\Pi$ with its realization in the space of smooth cuspidal automorphic forms on $G(\BA)$,
because of the Multiplicity One Theorem on the cuspidal spectrum of general linear groups (\cite[Theorem 5.5]{Sl74} and \cite[Page 209]{PS79}).
   \begin{defn}
           We say that an irreducible cuspidal automorphic representation $\Pi$ of $G(\BA)$ has a nonzero $(\eta, \psi)$-Shalika integral if its central character equals $\eta^n$, and there exists $\varphi\in \Pi$ such that
           $$\int_{(S(\Rk)\BR^\times_+)\backslash S(\BA)} \varphi(g) \cdot (\eta\otimes\psi)^{-1}(g)\, dg\neq 0,$$
where $\BR^\times_+$ is viewed as a subgroup of $G(\BA_\infty)\subset G(\BA)$ via the diagonally embedding, and  $dg$ is an  $S(\BA)$-invariant positive Borel measure on  $(S(\Rk)\BR^\times_+)\backslash S(\BA)$.
   \end{defn}
   The following theorem gives some characterizations for irreducible cuspidal automorphic representations of $G(\BA)$ to be of symplectic type, which is due to a collection of works of many people
(\cite[Section 8, Theorem 1]{JacquetShalikaExteriorSquareLfun}, \cite[Theorem 1.1]{AsgariShahidiGenericTransferforGSpin} and \cite[Theorem A]{HundleySayagDescentConstructionforGSpin}, for instance). We refer to \cite[Section 5]{GanRaghramArithmeticityPeriodsAutomorphicForms} for
more detailed discussion.
     \begin{prop}\label{prop: strong transfer}
          Let $\Pi$ be an irreducible cuspidal automorphic representation of $G(\BA)$  and let $\eta$ be an automorphic character of $\BA^\times$. Then the following statements are equivalent.
          \begin{enumerate}
            \item The representation  $\Pi$ is the functorial transfer of an irreducible generic cuspidal automorphic representation $\Sigma$ of $\mathrm{GSpin}_{2n+1}(\BA)$ with central character $\omega_\Sigma = \eta$.
            \item The representation $\Pi$ has a nonzero $(\eta, \psi)$-Shalika integral.
            \item The exterior square $L$-function $L(s, \Pi, \wedge^2\otimes\eta^{-1})$ has a pole at $s=1$.
          \end{enumerate}
          Moreover, if any of the above condition holds, then the transfer from $\mathrm{GSpin}_{2n+1}(\BA)$ is strong at all archimedean local places, in the sense that it respects local $L$-parameters at
each archimedean local place.
     \end{prop}
     \begin{rk}
     In the literature, the third statement in the above Proposition is often stated as `` the partial $L$-function $L^S(s, \Pi, \wedge^2\otimes\eta^{-1})$ has a pole at $s=1$".  With the classification of the unitary dual of general linear group over local fields of characteristic zero, as in \cite{V86} and \cite{T86}, and also in \cite[Theorem 7.2]{BadulescuRenardUnitaryDualJLCorrespondence}, it is not hard to show that the partial exterior square $L^S(s, \Pi, \wedge^2\otimes\eta^{-1})$ has a pole at $s=1$ if and only if so does the complete $L$-function $L(s, \Pi, \wedge^2\otimes\eta^{-1})$.
  \end{rk}

  Note that if any of the above statements holds, one must have that $\Pi^\vee \cong\Pi\otimes\eta^{-1}.$   It follows that if $\Pi = \widehat \otimes_v'\Pi_v$ is of symplectic type, then for every $v\mid\infty$, $\Pi_v$ is a Casselman-Wallach representation of $G(\Rk_v)$ of symplectic type in the following sense.
   \begin{defn}\label{def: local repn of symplectic type}
     Let $\BK$ be an archimedean local field. An irreducible Casselman-Wallach representation $\Pi_\BK$ of $G(\BK)$ is said to be of symplectic type if the image of  its Langlands parameter $W_\BK\rightarrow \GL_{2n}(\BC)$ is contained in a conjugation of $\mathrm{GSp}_{2n}(\BC)$.
    \end{defn}

For the  archimedean local Langlands correspondence, see \cite{L89} (for all real reductive groups) or \cite[Appendix]{JacquetArchimedeanRankinSelberg}, \cite[Sections 3 and 4]{Kn94} (for the general linear groups).

    \subsection{Shalika functionals  and Friedberg-Jacquet integrals}\label{subsection: Archimedean Friedberg-Jacquet Integral and Uniform Cohomological Vector}

Fix a  character $\eta_\BK: \BK^\times\rightarrow \BC^\times$.
Recall the Shalika subgroup $S$ from \eqref{eq: global shalika subgroup} and write $S_\BK:=S(\BK)$.
  We write $\eta_\BK\otimes\psi_\BK$ for the character of $S_\BK$ that sends $\mtrtwo{h}{0}{0}{h} \mtrtwo{1_n}{x}{0}{1_n}$ to $\eta_\BK(\det(h)) \psi_\BK(\mathrm{tr}(x))$.

 Write $G_\BK:=G(\BK)$. Let $\Pi_\BK$ be an irreducible Casselman-Wallach representation of $G_\BK$ when
$\BK$ is archimedean, or an irreducible smooth representation of $G_\BK$ when $\BK$ is non-archimedean.
A  linear  functional  $\lambda_\BK$ (continuous in the archimedean case) on $\Pi_\BK$ is called a
 $(\eta_\BK,\psi_\BK)$-Shalika functional if
     \begin{equation}\label{Eq00: Def of Shalika Functional}
        \langle \lambda_\BK, g.u \rangle =(\eta_\BK\otimes\psi_\BK)(g)\cdot\langle \lambda_\BK, u\rangle ,
    \end{equation}
    for all $g\in S_\BK$ and $u\in \Pi_\BK$.
   The following result, which asserts the uniqueness of Shalika functionals, plays an important role in this article.

     \begin{lemma}\label{uniquesh}
    The space of $(\eta_\BK,\psi_\BK)$-Shalika functionals is at most one-dimensional, namely
    \begin{equation}\label{eq: uniqueness of shalika}\dim \Hom_{S_\BK}(\Pi_\BK, \eta_\BK\otimes\psi_\BK)\leq 1.\end{equation}
    \end{lemma}

\begin{proof}
 If $\eta_\BK$ is trivial, this is proved  in \cite[Proposition 6.1]{JR}
for the non-archimedean case and \cite[Theorem 1.1]{AizenbudGourevitchJacquetUniquenessShalika} for the archimedean case. In general, this is proved in \cite[Theorem A]{ChenSunUniquenessLinearModel}.
\end{proof}

    Suppose that $\lambda_\BK$ is a nonzero  $(\eta_\BK,\psi_\BK)$-Shalika functional on $\Pi_\BK$ so that \eqref{Eq00: Def of Shalika Functional} is satisfied.
    \begin{lemma}\label{sphshalika}
    If $\BK$ is non-archimedean, both $\Pi_\BK$ and $\eta_\BK$ are unramified, and the conductor of $\psi_\BK$ is the ring of integers in $\BK$, then
    $\lambda_\BK(v_\BK^\circ)=1$ 
    for a unique spherical vector $v_\BK^\circ\in \Pi_\BK$.
    \end{lemma}
\begin{proof}
 See \cite[Section 2.3]{FriedbergJacquetLinearPeriods}.
\end{proof}
    Here and as usual, a vector in $\Pi_\BK$ is said to be spherical if it is fixed by the standard maximal compact subgroup $\GL_{2n}(\CO_\BK)$ of $G_\BK$, as $\BK$ is non-archimedean.

For the character $\chi_\BK:\BK^\times\rightarrow \BC^\times$, the local Friedberg-Jacquet integral is defined by
    \begin{equation}\label{eq: F-J integral}
        Z(v,s,\chi_\BK) := \int_{\GL_n(\BK)} \left \langle\lambda_\BK, \mtrtwo{g}{0}{0}{1_n}.v\right \rangle\cdot \abs{\det g}_\BK^{s-\frac{1}{2}}\cdot \chi_\BK(\det g)\, dg, \quad v\in \Pi_\BK,\, s\in \BC.
    \end{equation}
    Here $|\cdot|_\BK$ denotes the normalized absolute value on $\BK$, and $dg$ is a fixed Haar measure. When $\BK$ is non-archimedean,
we always  take the normalized Haar measure on $\GL_n(\BK)$ so that a maximal open compact subgroup of it has total volume $1$.

      Let us define
    \begin{equation}\label{eq: temp 24}
       H:= \left\{\mtrtwo{g_1}{0}{0}{g_2}\Big\lvert g_1,g_2\in \GL_n \right\} =\GL_n \times \GL_n\subset G.
    \end{equation}
    Write $H_\BK:=H(\BK)$.
    For each $t\in \BC$, we define a character of $H_\BK$ as follows:
    \begin{equation}\label{eq: xichi}
         \xi_{\chi_\BK, t}:=\xi_{\eta_\BK, \chi_\BK, t}:=\left( (\chi_\BK\circ\det)\cdot \abs{\det }_\BK^{t}\right)\otimes\left(((\chi_\BK^{-1}\cdot \eta^{-1}_\BK)\circ\det) \cdot \abs{\det }_\BK^{-t} \right).
    \end{equation}
The following proposition was proved in \cite[Proposition 3.1]{FriedbergJacquetLinearPeriods} and \cite[Theorem 3.1]{AizenbudGourevitchJacquetUniquenessShalika}.
    \begin{prop}\label{prop: properties of F-J int}
With the notation introduced above, the following hold.
    \begin{enumerate}
      \item When the real part of $s\in\BC$ is sufficiently large, the local Friedberg-Jacquet integral $Z(v,s,\chi_\BK)$ converges absolutely for all $v\in \Pi_\BK$.
      \item There is a unique (continuous in the archimedean case) map
      \[
       \Pi_\BK\times \BC\rightarrow \BC, \qquad (v,s)\mapsto Z^\circ(v,s,\chi_\BK)
      \]
      that is  homomorphic in the second variable such that for all $s\in \BC$ whose real part is sufficiently large,
      \[
        Z^\circ(v,s,\chi_\BK) = \frac{1}{L(s,\Pi_\BK\otimes\chi_\BK)}Z(v,s,\chi_\BK)\quad \, \textrm{  for all $v\in \Pi_\BK$}.
        \]
       \item
       For every $s\in \BC$,
       \[
          Z^\circ(\,\cdot\,,s,\chi_\BK) \in \Hom_{H_\BK}(\Pi_\BK\otimes\xi_{\chi_\BK, s-\frac{1}{2}}, \BC)\setminus \{0\}.
          \]
Here $\Pi_\BK\otimes\xi_{\chi_\BK, s-\frac{1}{2}}$ is obviously identified with $\Pi_\BK$ as a vector space so that the linear functional $ Z^\circ(\,\cdot\,,s,\chi_\BK)$ is defined on it.

     \item There exists $v\in \Pi_\BK$ such that $Z^\circ(v,s,\chi_\BK)=1$ for all $s\in \BC$.
       \item When $\BK$ is non-archimedean, for every $v\in \Pi_\BK$, the function $s\mapsto Z^\circ(v,s,\chi_\BK)$
 is a polynomial of $q^{s}, q^{-s}$, where $q$ is the cardinality of the residue field of $\BK$.
      \item If $\BK$ is non-archimedean,  $\Pi_\BK$, $\eta_\BK$ and $\chi_\BK$ are unramified, and the conductor of $\psi_\BK$ is the ring of integers in $\BK$, then
   $Z^\circ(v_\BK^\circ,s,\chi_\BK)=1$ for all $s\in \BC$, where $v_\BK^\circ\in \Pi_\BK$ is the spherical vector such that $\lambda_\BK(v_\BK^\circ)=1$.
    \end{enumerate}
    \end{prop}

\subsection{Algebraic  representations and critical places}\label{algssec-nd}

Suppose that $\BK$ is archimedean in this subsection.  Recall that its Weil group is
       $$            W_\BK:=\begin{cases}
                   \overline \BK^\times\sqcup \mathrm j\cdot \overline \BK^\times, \quad & \textrm{if $\BK\cong\BR$}; \\
                   \BK^\times, \quad & \textrm{if $\BK\cong\BC$,}
                    \end{cases}$$
                   where $\overline \BK$ is an algebraic closure of $\BK$.
Every  completely reducible finite dimensional representation $\rho$ of $\overline \BK^\times$ has the form
\[
 \iota^{a_1}\bar \iota^{b_1}\oplus \iota^{a_2}\bar \iota^{b_2}\oplus \cdots\oplus \iota^{a_k}\bar \iota^{b_k},\quad (k \geq 0, \,a_i, b_i\in \BC,\, a_i-b_i\in \BZ),
  \]
where  $\iota, \bar \iota : \overline \BK\rightarrow \BC$ are the two topological isomorphisms,
 and  $\iota^a\bar \iota^b$ is the character
 \begin{equation}\label{ii}
 z\mapsto (\iota(z))^{a-b} \cdot (\iota(z)\bar \iota(z))^b,
 \end{equation}
  of $\overline \BK^\times$,  for all $a,b\in \BC$ with $a-b\in \BZ$.
Following \cite{ClozelAutomorphicFormsMotive}, we say that $\rho$ is {\sl algebraic} if all $a_i$'s and $b_i$'s are integers, and say that $\rho$ is {\sl regular} if $a_i$'s are pairwise distinct and  $b_i$'s are also pairwise distinct. We say a completely reducible finite dimensional representation  of the Weil group $W_\BK$ is algebraic (or regular) if so is its restriction to $\overline \BK^\times$.

\begin{defn}\label{def: arglobal critical place}
 An irreducible Casselman-Wallach representation $\Pi_\BK$ of $\GL_k(\BK)$ is said to be {\bf algebraic} (or {\bf regular}) if so is the Langlands parameter of $\Pi_\BK\otimes \abs{\det}_\BK^{\frac{1-k}{2}}$.
\end{defn}



An irreducible cuspidal automorphic representation $\Pi=\widehat \otimes_v' \Pi_v$ of $\GL_k(\BA)$ is said to be {\bf algebraic} (or {\bf regular}) if so is $\Pi_v$ for all $v\mid \infty$.
It is easy to see that all possible poles of the $L$-function $L(s, \Pi_\BK)$ are in the set $\frac{k-1}{2}+\BZ$, for every algebraic irreducible Casselman-Wallach representation $\Pi_\BK$ of $\GL_k(\BK)$.
 \begin{defn}\label{def: global critical place}
        Let $\Pi_\BK$ be an algebraic irreducible Casselman-Wallach representation  of $\GL_k(\BK)$.  A number in $\frac{k-1}{2}+\BZ$ is called a {\bf critical place} for   $\Pi_\BK$ if it is not a pole of the local L-function $L(s, \Pi_\BK)$ or $L(1-s, \Pi_\BK^\vee)$.
  \end{defn}

Given   an algebraic irreducible cuspidal automorphic representation $\Pi=\widehat \otimes_v' \Pi_v$ of $\GL_k(\BA)$, a number in $\frac{k-1}{2}+\BZ$ is called a  critical place for   $\Pi$ if it is a critical place for $\Pi_v$ for all $v\mid \infty$.

Now we suppose that $\Pi$ is a regular algebraic irreducible cuspidal automorphic representation of $\GL_{k}(\BA)$. Write $\Pi=\Pi_\infty\otimes \Pi_f$ as in \eqref{decpi}. According to \cite[Lemma 3.14, Lemma 4.9]{ClozelAutomorphicFormsMotive}, $\Pi_\infty$ is essentially tempered and cohomological. Here  ``essentially tempered" means that after twisting a character,  $\Pi_\infty$ becomes unitarizable and tempered. The notion ``cohomological" amounts to saying that  there is a unique irreducible algebraic representation $F$ of $\GL_k(\Rk\otimes_\BQ \BC)$, called the coefficient system of $\Pi$, such that the {\sl total continuous cohomology}
    \begin{equation}\label{eq: temp 23}  \RH^*_{\mathrm{ct}}(\GL_k(\BA_\infty)^0; \Pi_\infty\otimes F^\vee)\neq\{ 0\}. \end{equation}
Here and henceforth, for each Lie group $M$, we write $M^0$ for its identity connected component. 


      \begin{rk}
In this paper, we use the notion of continuous cohomology to state the relevant theorems, although in the main references
\cite{ClozelAutomorphicFormsMotive}, \cite{VoganZuckermanUnitaryRepresentationNonzeroCohomology} and \cite[Chapter 9]{WallachRealReductiveGroups1}  for cohomological representations, the authors use the notion of relative Lie algebra cohomology. It is a theorem of G. Hochschild and G. Mostow in \cite[Theorem 6.1]{HochshildMostowCohomologyofLieGroups} that the continuous cohomology agrees with the relative Lie algebra cohomology.
Hence we feel free to switch the two cohomological languages during the discussions and proofs in this paper.
    \end{rk}

\begin{lemma}
Let the notation and assumptions be as above. Then $\BQ(\Pi)$ is a number field and $\BQ(F)\subset \BQ(\Pi)$.
\end{lemma}
\begin{proof}
The first assertion is proved in  \cite[Theorem 3.13]{ClozelAutomorphicFormsMotive}. In view of the Strong Multiplicity One Theorem proved by  Piatetski-Shapiro (\cite[Page 209]{PS79}) and Jacquet-Shalika (\cite[Corollary 4.10]{JacquetShalikaEulerProduct1}), the second assertion  also follows from  \cite[Theorem 3.13]{ClozelAutomorphicFormsMotive}.
\end{proof}


    \subsection{Cohomological representations of symplectic type}\label{subsection: Cohomological representations of symplectic type}

  When $E$  is either a number field or an archimedean local field, we write $\CE_E$ for the set of all field embeddings (continuous field embeddings in the archimedean case) $\iota: E \rightarrow \BC$. For any $\iota\in\CE_E$, we write $\bar \iota$ for the composition of
    $$  E\xrightarrow{\iota} \BC\xrightarrow{\textrm{complex conjugation}} \BC.$$

   As mentioned before,  the archimedean local components of an irreducible regular algebraic cuspidal automorphic representation of $G(\BA)$ are essentially tempered and cohomological. In this subsection and the next one, we suppose that $\BK$ is archimedean, and $\Pi_\BK$ is an essentially tempered irreducible Casselman-Wallach representation of $G_\BK$ such that  the total continuous cohomology
   \begin{equation}\label{eq: archimedean coh nonzero}
          \RH^*_\mathrm{ct}(G_\BK^0,\Pi_\BK\otimes F^\vee_\BK) \ne \{0\},
    \end{equation}
    where $F_\BK$ is an irreducible algebraic representation of $G(\BK \otimes_\BR \BC)$. Then $\Pi_\BK$ is regular and algebraic. Write $F_\BK=\otimes_{\iota\in \CE_\BK} F_\iota$ and write
    \begin{equation}\label{eq2: global highest weight}\nu_\iota=(\nu^\iota_1 \geq  \nu^\iota_2\geq \cdots \geq \nu^\iota_{2n})\in \BZ^{2n}\end{equation}
    for the highest weight of $F_\iota$.  By \cite[Lemma 4.9]{ClozelAutomorphicFormsMotive}, there is an integer $w_\BK$ such that
         \begin{equation}\label{eq2: global purity lemma}
   \nu^\iota_1+ \nu^{\bar \iota}_{2n}=\nu^\iota_2+ \nu^{\bar \iota}_{2n-1}=\cdots= \nu^\iota_{2n}+ \nu^{\bar \iota}_{1}=w_\BK
    \end{equation}
    for all  $\iota\in \CE_\BK$.

Following \cite{ClozelAutomorphicFormsMotive} (and also \cite[Proposition 1.1]{JiangLinTianExplicitCohomologicalVectorReal} and \cite[Proposition 1.1]{LinTianExplicitCohomologicalComplex}), we are going to write down the local
Langlands datum for $\Pi_\BK$. To simplify our notation, throughout this paper, if $\sigma_j$ is a Casselman-Wallach representation of $\GL_{n_j}(\BK)$ $(1\leq j \leq r)$, we write
    $$\sigma_1\dot \times\sigma_2\dot \times\cdots\dot \times\sigma_r:=\mathrm{Ind}_{P_{n_1, n_2, \cdots, n_r}}^{\GL_{n_0}(\BK)}\sigma_1\widehat \otimes\sigma_2\widehat \otimes\cdots\widehat \otimes\sigma_r $$
    for the normalized smoothly induced representation of  $\GL_{n_0}(\BK)$, where
    \begin{equation}\label{eq: subsection: coh rep}
        n_0 = n_1+n_2+\cdots+n_r,
    \end{equation}
    $P_{n_1, n_2, \cdots, n_r}$ denotes the standard parabolic subgroup of  $\GL_{n_0}(\BK)$   corresponding to the partition \eqref{eq: subsection: coh rep}, and the representation
$\sigma_1\widehat \otimes\sigma_2\widehat \otimes\cdots\widehat \otimes\sigma_r $ of the Levi subgroup $\GL_{n_1}(\BK)\times\GL_{n_2}(\BK)\times\cdots\times\GL_{n_r}(\BK)$ is viewed as a representation of $P_{n_1, n_2, \cdots, n_r}$ by inflation as usual. The infinitesimal character of the algebraic representation $F_\iota$ is
     \begin{equation}\label{eq: temp 9999}
      \hat \nu_\iota := (\hat \nu^\iota_1 >  \hat \nu^\iota_2>\cdots > \hat \nu^\iota_{2n})\in (\frac{1}{2}+\BZ)^{2n},
    \end{equation}
    where
    \begin{equation}\label{eq: infinitesmal char}
\hat \nu^\iota_i=\nu^\iota_i+\frac{2n+1-2i}{2},\qquad i=1,2,\cdots, 2n.
\end{equation}

    In the real case, $\BK$ is identified with $\BR$,  $\CE_\BK=\{\iota\}$ only contains one element, and $\Pi_\BR$ is isomorphic to
    \begin{equation}\label{Eq: cohomomogical repn parabolic induction parameter real}
       D_{\hat \nu^\iota_1, \hat \nu^\iota_{2n}}\dot \times D_{\hat \nu^\iota_2, \hat \nu^\iota_{2n-1}} \dot \times \cdots\dot \times D_{\hat \nu^\iota_n, \hat \nu^\iota_{n+1}},
    \end{equation}
    where $D_{a,b}$ is the relative discrete series of $\GL_2(\BR)$ with infinitesimal character $(a, b)$, for all $a,b\in \BC$ with $a-b\in \BZ\setminus \{0\}$.
 In the complex case $\BK\cong\BC$, $\CE_\BK=\{\iota, \bar \iota\}$ has two elements, and $\Pi_\BK$ is isomorphic to a principal series
    \begin{equation}\label{Eq: cohomomogical repn parabolic induction parameter complex}
           \iota^{\hat \nu_1^{\iota}} \bar \iota^{\hat \nu_{2n}^{\bar \iota}} \dot \times \iota^{\hat \nu_2^{\iota}} \bar \iota^{\hat \nu_{2n-1}^{\bar \iota}} \dot \times\cdots\dot \times
             \iota^{\hat \nu_{2n}^{\iota}} \bar \iota^{\hat \nu_{1}^{\bar \iota}},
    \end{equation}
     where $\iota^a\bar \iota^b$ is defined as in \eqref{ii}.

     \begin{prop}\label{prop: L-function formula}
         Let the notation be as above. Then
                 \begin{equation}\label{eq: L fun formula}
             L(s,\Pi_\BK)=\prod_{\iota\in \CE_\BK} \prod_{r=1}^{n} 2\cdot(2\pi)^{-(s+\max\{\hat \nu_r^\iota, \hat \nu_{2n+1-r}^{\bar \iota}\})}\cdot \Gamma(s+\max\{\hat \nu_r^\iota, \hat \nu_{2n+1-r}^{\bar \iota}\}). 
         \end{equation}
         Here $\Gamma$ denotes the classical gamma function. 
     \end{prop}
     \begin{proof}
         First assume that  $\BK\cong\BR$. Then $\CE_\BK = \{\iota\}$ has only one element, and $\bar \iota=\iota$. It is clear from \eqref{eq: temp 9999} that  $\hat{\nu}_{2n+1-r}^{\bar \iota} = \hat{\nu}_{2n+1-r}^{\iota}< \hat{\nu}_{r}^{\iota}$ for all $r = 1,2,\cdots,n$. Hence (\cite[Theorem 5.15]{JL70})
         $$\begin{aligned}
              L(s, D_{\hat{\nu}_{r}^{\iota},\hat{\nu}_{2n+1-r}^{\iota}})
              = &2\cdot(2\pi)^{-(s+\hat{\nu}_r^\iota)}\cdot \Gamma(s+\hat{\nu}_r^\iota)\\
              = &2\cdot(2\pi)^{-(s+\max\{\hat \nu_r^\iota, \hat \nu_{2n+1-r}^{\bar \iota}\})}\cdot \Gamma(s+\max\{\hat \nu_l^\iota, \hat \nu_{2n+1-r}^{\bar \iota}\}).\end{aligned}$$
         Using the local Langlands correspondence at the archimedean place (\cite{L89}, also see \cite[Appendix]{JacquetArchimedeanRankinSelberg} and \cite{Kn94}), we have that
         $L(s, \Pi_\BK) = \prod_{r=1}^n L(s, D_{\hat{\nu}_{r}^{\iota},\hat{\nu}_{2n+1-r}^{\iota}})$, which proves the proposition 
         in the real case.

       Now assume that  $\BK\cong\BC$. Then $\CE_\BK = \{\iota,\bar \iota\}$ has two elements. For all $r=1,2,\cdots,2n$, the standard $L$-function for the character $\iota^{\hat \nu_r^{\iota}} \bar \iota^{\hat \nu_{2n+1-r}^{\bar \iota}}$ is (\cite[page 97]{JL70})
         $$L(s, \iota^{\hat \nu_r^{\iota}} \bar \iota^{\hat \nu_{2n+1-r}^{\bar \iota}}) = 2\cdot(2\pi)^{-(s+\max\{\hat \nu_r^\iota, \hat \nu_{2n+1-r}^{\bar \iota}\})}\cdot \Gamma(s+\max\{\hat \nu_r^\iota, \hat \nu_{2n+1-r}^{\bar \iota}\}).$$
As before, using the   local Langlands correspondence at the archimedean place, the proposition in the complex case easily follows.
     \end{proof}

 \begin{rk}\label{rk: finite L-fun arithmetic property proof}
       Note that $\Gamma(k)$ is a positive integer whenever $k\in \BZ$ is not a pole of the gamma function $\Gamma$. Proposition  \ref{prop: L-function formula} then easily implies that for every critical place $\frac{1}{2}+j$ of $\Pi_\BK$,
         $$ \frac{L(\frac{1}{2}+j, \Pi_\BK)}{ (2\pi)^{-jn\cdot [\BK:\BR]}\cdot\Omega_{\Pi_\BK}^{-1}} \in \BQ^\times,$$
         where  
         \begin{equation}
     \label{omegapik}
         \Omega_{\Pi_\BK}:= \prod_{\iota\in \CE_\BK} \prod_{r=1}^{n}(2\pi)^{1/2 +\max\{\hat \nu_r^\iota, \hat \nu_{2n+1-r}^{\bar \iota}\}}.    
         \end{equation}
          Hence the rationality result for $L(s,\Pi_f\otimes\chi_f)$, as in \eqref{eq: main thm finite}, follows directly from Theorem \ref{thm: global period relation}, by noting that
          $\Pi_\BK\otimes \chi_{\BK,2n}\cong\Pi_\BK$ whenever $\chi_{\BK,2n}$ is a quadratic character of $\GL_{2n}(\BK)$.
     \end{rk}


 From the structure of the induction parameters in \eqref{Eq: cohomomogical repn parabolic induction parameter real}, all irreducible essentially tempered cohomological representations of $\GL_{2n}(\BR)$ are of symplectic type (see Definition \ref{def: local repn of symplectic type}). This has been verified in \cite[Page 382]{ChenevierClozelMinimallyRamifiedNumberFields}. While in the complex case, a cohomological representation of $\GL_{2n}(\BC)$ as in \eqref{Eq: cohomomogical repn parabolic induction parameter complex} is not necessarily of symplectic type.  Recall the non-trivial unitary character $\psi_\BK$ of $\BK$ from the beginning of Section \ref{sec-PPThm}.
    \begin{prop}\label{prop: classification of coh repn with shalika model in terms of Langlands parameter}
         The following conditions are equivalent for the irreducible essentially tempered cohomological representation $\Pi_\BK$ given in \eqref{Eq: cohomomogical repn parabolic induction parameter real} and \eqref{Eq: cohomomogical repn parabolic induction parameter complex}.
        \begin{enumerate}
          \item It is of symplectic type.
          \item There exist integers $\{w_\iota\}_{\iota\in \CE_\BK}$ such that
    \begin{equation}\label{eq: def of w iota}\nu_1^\iota+\nu_{2n}^\iota=\nu^\iota_2+ \nu^{\iota}_{2n-1}=\cdots= \nu^\iota_{2n}+ \nu^{\iota}_{1} =w_\iota\qquad \textrm{for all $\iota\in \CE_\BK$}.\end{equation}
          \item For some character $\eta_\BK$ of $\BK^\times$, $\Pi_\BK$  has a nonzero $(\eta_\BK, \psi_\BK)$-Shalika functional on it.
        \end{enumerate}
    \end{prop}
    \begin{proof}
        When $\BK \cong\BR$, both $(1)$ and $(2)$ hold trivially. The third statement is exactly \cite[Corollary 2.6]{JiangLinTianExplicitCohomologicalVectorReal}. In the case of $\BK \cong \BC$, the equivalence of all the three statements is proved in \cite[Theorem 2.1]{LinTianExplicitCohomologicalComplex}.
    \end{proof}


From Proposition \ref{prop: classification of coh repn with shalika model in terms of Langlands parameter} and \cite[Theorem 2.1]{JiangLinTianExplicitCohomologicalVectorReal}, we immediately have the following.
   \begin{cor}\label{cor: coh repn symp type induced parameter}
     Suppose that $\Pi_\BK$ is an irreducible essentially tempered cohomological representation of $G_\BK$  as in \eqref{Eq: cohomomogical repn parabolic induction parameter real} or \eqref{Eq: cohomomogical repn parabolic induction parameter complex}. If $\Pi_\BK$ is  of symplectic type, then $\Pi_\BK$ is isomorphic to
       \begin{equation}\label{eq: coh repn symp type induced parameter}
             \sigma_1\dot \times\sigma_2\dot \times\cdots\dot \times\sigma_n,
       \end{equation}
      where $\sigma_j = D_{\hat \nu^\iota_j, \hat \nu^\iota_{2n+1-j}}$ if $\BK\cong\BR$ so that $\CE_\BK=\{\iota\}$; and $\sigma_j = \iota^{\hat \nu_j^{\iota}} \bar \iota^{\hat \nu_{2n+1-j}^{\bar \iota}} \dot \times
       \iota^{\hat \nu_{2n+1-j}^{\iota}} \bar \iota^{\hat \nu_{j}^{\bar \iota}}$ if $\BK\cong\BC$ so that $\CE_\BK=\{\iota, \bar \iota\}$.
Moreover, the representation  $\sigma_j$ ($j=1,2,\cdots,n$) of $\GL_2(\BK)$ has central character
\[
 \eta_\BK:=\prod_{\iota\in \CE_\BK} \iota^{w_\iota} : \BK^\times \rightarrow \BC^\times,
\]
and $\Pi_\BK$ has a nonzero $(\eta_\BK,\psi_\BK)-$Shalika functional, where $w_\iota$ is as in \eqref{eq: def of w iota}.
   \end{cor}

   \subsection{Balanced coefficient systems}\label{subsection: Coefficient System of Regular Algebraic Cuspidal Representation}
    We introduce the notion of a balanced coefficient system of a cohomological representation of symplectic type, both locally and globally, and discuss its relation with the critical places of the representation defined in Definition \ref{def: global critical place}. Let us start with the local balanced coefficient system.

Assume that the representation $\Pi_\BK$ is of symplectic type. Thus, the integers $\{\nu^\iota_l\}_{\iota\in\CE_\BK, 1\leq l \leq 2n}$ must satisfy \eqref{eq: def of w iota}.

     \begin{defn}\label{defn: balanced coeffi system}
         For an integer $j$, we say that the coefficient system $F_\BK=\otimes_{\iota\in \CE_\BK} F_\iota$ is $j$-balanced if
         \begin{equation}\label{jbform}
          \Hom_{H_\BC}(F_\iota^\vee, {\det}^j\otimes {\det}^{-j-w_\iota})\neq \{0\}\qquad \textrm{for all $\iota\in \CE_\BK$}.
          \end{equation}
         We say the coefficient system $F_\BK$ is balanced if it is $j$-balanced for some integer $j$.
     \end{defn}
 The following proposition explains the relation between the existence of a critical place and the balanceness of the coefficient system.

 \begin{prop}\label{prop: critical place}
Let $w_\iota$ be as in \eqref{eq: def of w iota}.
     \begin{enumerate}
       \item If  $w_\iota=w_{\bar{\iota}}$  for all  $\iota\in \CE_\BK$,  then $F_\BK$ is balanced.
       \item Assume that $F_\BK$ is balanced. Then  for every $j\in \BZ$,  $\frac{1}{2}+j$ is  a critical place for $\Pi_\BK$ if and only if $F_\BK$ is $j$-balanced. Moreover, the set of critical places of $\Pi_\BK$ is
     \[
      \left\{\frac{1}{2}+j\mid j\in \BZ, -\nu^\iota_n\leq j\leq -\nu^\iota_{n+1} \ \textrm{ for all } \iota\in \CE_\BK\right\}.
     \]
       \end{enumerate}
     \end{prop}
     \begin{proof}
      As an instance of H. Schlichtkrull's generalization of the Cartan-Helgason Theorem (\cite[Theorem 7.2]{Sch}, see also \cite[Theorem 2.1]{KnappBranchingRuleCptSymmetricSpace}), we know that for all $\iota\in \CE_\BK$,
     \begin{equation}\label{eq: temp 9998}
       \Hom_{H_\BC}(F_\iota^\vee, {\det}^j\otimes {\det}^{-j-w_\iota})\neq \{0\}\quad\textrm{ if and only if }\quad -\nu_n^\iota\leq j\leq -\nu^\iota_{n+1}.
     \end{equation}
If $w_\iota=w_{\bar{\iota}}$ for all  $\iota\in \CE_\BK$, then  \eqref{eq2: global purity lemma} and \eqref{eq: def of w iota} imply that
$\nu_i^\iota=\nu_i^{\bar \iota}$ for all $\iota\in \CE_\BK$ with $i=1,2, \cdots, 2n$.
Hence the first assertion of the proposition easily follows.

Now we assume that $F_\BK$ is balanced. Then it is clear from \eqref{eq: temp 9998} that
$\nu_n^\iota\geq \nu^{\bar \iota}_{n+1}$ for all $\iota\in \CE_\BK$.
    Therefore, for all $\iota\in\CE_\BK$, we have that
    \begin{equation}\label{eq: temp 9997}
       \hat \nu^\iota_{1}>\hat \nu^\iota_{2}>\cdots>\hat \nu^\iota_{n}>\hat \nu^{\bar{\iota}}_{n+1}>\hat \nu^{\bar{\iota}}_{n+2}>\cdots>\hat \nu^{\bar{\iota}}_{2n}.
    \end{equation}
    Note that the contragredient representation $\Pi_\BK^\vee$ is isomorphic to
\[
D_{-\hat \nu^\iota_{2n}, -\hat \nu^\iota_{1}}\dot \times D_{-\hat \nu^\iota_{2n-1}, -\hat \nu^\iota_{2}} \dot \times \cdots\dot \times D_{-\hat \nu^\iota_{n+1}, -\hat \nu^\iota_{n}}
\]
in the real case, and is isomorphic to
$\iota^{-\hat \nu_{2n}^{\iota}} \bar \iota^{-\hat \nu_{1}^{\bar \iota}} \dot \times \iota^{-\hat \nu_{2n-1}^{\iota}} \bar \iota^{-\hat \nu_{2}^{\bar \iota}} \dot \times\cdots\dot \times
             \iota^{-\hat \nu_{1}^{\iota}} \bar \iota^{-\hat \nu_{2n}^{\bar \iota}}$
    in the complex case.
Using \eqref{eq: temp 9997} and Proposition \ref{prop: L-function formula}, we find that

\begin{equation}\label{eq: L fun formula balanced repn}
             L(s,\Pi_\BK)= \prod_{\iota\in \CE_\BK} \prod_{r=1}^{n} 2\cdot(2\pi)^{-(s+\hat \nu_r^\iota)}\cdot \Gamma(s+\hat \nu_r^\iota)
         \end{equation}
         and
         \begin{equation}\label{eq: L fun formula balanced contra repn}
             L(1-s,\Pi_\BK^\vee)= \prod_{\iota\in \CE_\BK} \prod_{r=1}^{n} 2\cdot(2\pi)^{-(1-s-\hat \nu_{2n+1-r}^\iota)}\cdot \Gamma(1-s-\hat \nu_{2n+1-r}^\iota).
         \end{equation}
These imply that the set of critical places of $\Pi_\BK$ is
     $\left\{\frac{1}{2}+j\mid j\in \BZ, -\nu^\iota_n\leq j\leq -\nu^\iota_{n+1} \ \textrm{ for all } \iota\in \CE_\BK\right\}$.
     Hence the second assertion of the proposition also follows.
     \end{proof}

\begin{rk}
If $\BK\cong\BR$, then $\iota=\bar{\iota}$ and hence the first assertion of Proposition \ref{prop: critical place} says that $F_\BK$ is always balanced. In this case, the Proposition also  follows from \cite[Propositions 6.1.1 and 6.3.1]{GrobnerRaguhramShalikaModelArithmetic}.
\end{rk}

    Globally, we continue to assume that $\Pi$ is an irreducible regular algebraic cuspidal automorphic representation of $G(\BA)$ of symplectic type, with coefficient system $F=\otimes_{\iota\in \CE_\Rk} F_\iota$. As before, write
    \begin{equation}\label{eq2: 22global highest weight}
\nu_\iota=(\nu^\iota_1 \geq  \nu^\iota_2\geq \cdots \geq \nu^\iota_{2n})\in \BZ^{2n}\end{equation}
    for the highest weight of $F_\iota$.
By \cite[Lemma 4.9]{ClozelAutomorphicFormsMotive}, there is an integer $w$ such that
         \begin{equation}\label{eq2: 22global purity lemma}
   \nu^\iota_1+ \nu^{\bar \iota}_{2n}=\nu^\iota_2+ \nu^{\bar \iota}_{2n-1}=\cdots= \nu^\iota_{2n}+ \nu^{\bar \iota}_{1}=w
    \end{equation}
    for all  $\iota\in \CE_\Rk$. By Propositions \ref{prop: strong transfer} and  \ref{prop: classification of coh repn with shalika model in terms of Langlands parameter}, there are integers $\{w_\iota\}_{\iota\in \CE_\Rk}$ such that
    \begin{equation}\label{eq3: def of w iota}\nu_1^\iota+\nu_{2n}^\iota=\nu^\iota_2+ \nu^{\iota}_{2n-1}=\cdots= \nu^\iota_{2n}+ \nu^{\iota}_{1} =w_\iota\qquad \textrm{for all $\iota\in \CE_\Rk$}.\end{equation}
Similar to Definition \ref{defn: balanced coeffi system}, we make the following definition.
   \begin{defn}\label{defn: gbalanced coeffi system}
         For an integer $j$, we say that the coefficient system $F=\otimes_{\iota\in \CE_\Rk} F_\iota$ is $j$-balanced if
         $\Hom_{H_\BC}(F_\iota^\vee, {\det}^j\otimes {\det}^{-j-w_\iota})\neq \{0\}$ for all $\iota\in \CE_\Rk$.
         We say that the coefficient system $F$ is balanced if it is $j$-balanced for some integer $j$.
     \end{defn}

        \begin{rk}\label{rk217}
           If the number field $\Rk$ does not contain a $CM$-field, for example, if $\Rk$ has at least one real place, then by the classification of algebraic automorphic characters, all $w_\iota$'s defined in \eqref{eq3: def of w iota} are equal (see \cite{SerreAbelianladicRepnandEllipticCurves} or \cite[Proposition 1.12]{Far}). In this case, by Proposition \ref{prop: critical place}, if $\Pi$ has a critical place, then the coefficient system $F$ for $\Pi$ is automatically balanced. Hence Theorem \ref{thm: global period relation} holds without the assumption of a balanced coefficient system of $\Pi$.
    \end{rk}

\section{Four Local Theorems}\label{sectionfourt}

We discuss four local results that are essential to the proof of Theorem \ref{thm: global period relation}.
The first one  establishes the rationality of the Friedberg-Jacquet integrals in the non-archimedean case (Theorem \ref{0lemma: Z circ f rational}). The other three  are all archimedean results that establish the existence of uniform
cohomological test vectors (Theorem \ref{thm: uniform coh test vector}), the non-vanishing of archimedean modular symbols (Theorem \ref{thm: Non-vanishing hypothesis}), and
relations of archimedean modular symbols at different critical places (Theorem \ref{thm: archi period relations}).
We will prove the first result in this section, and leave the proofs of the other three results to Sections \ref{section: Proof of Critical place and Non-vanishing Modular Symbol} and
\ref{section: Identification}.

\subsection{Non-archimedean period relations}\label{nonpr}

Suppose that $\BK$ is non-archimedean. Recall the character $\eta_\BK\otimes \psi_\BK$ of $S_\BK$ from the beginning of Section \ref{subsection: Archimedean Friedberg-Jacquet Integral and Uniform Cohomological Vector}.
Suppose that $\Pi_\BK$ is an irreducible  smooth representation of $G_\BK$ with a  nonzero $(\eta_\BK,\psi_\BK)$-Shalika functional. 
We fix such a nonzero functional
$\lambda_\BK\in \Hom_{S_\BK}(\Pi_\BK, \eta_\BK\otimes \psi_\BK)$.
The rationality property that will be discussed below depends on this fixed $\lambda_\BK$.

By \cite[Proposition 3.1]{ClozelAutomorphicFormsMotive},   the representation $\Pi_\BK$ has a $\BQ(\Pi_\BK)$-rational structure.
In this subsection, we assume further that $\BQ(\Pi_\BK)$ is a number field.  Note that the character $\eta_\BK^n$ equals the central character of $\Pi_\BK$. Thus $\BQ(\eta_\BK)$ is also a number field.

For the character   $\chi_\BK: \BK^\times \rightarrow \BC^\times$, we recall  the nonzero functional
 \[
       Z^\circ(\,\cdot\,,s,\chi_\BK) \in \Hom_{H_\BK}(\Pi_\BK\otimes\xi_{\chi_\BK, s-\frac{1}{2}}, \BC)
        \]
from Proposition \ref{prop: properties of F-J int}, which depends on $\lambda_\BK$. Note that if $s_0\in \frac{1}{2}+\BZ$, then $\xi_{\chi_\BK, s_0-\frac{1}{2}}$ is defined over $\BQ(\eta_\BK, \chi_\BK)$.
  \begin{thm}[Non-archimedean Period Relation]\label{0lemma: Z circ f rational}
  There exists a unique $\BQ(\Pi_\BK, \eta_\BK)$-rational structure on the $G_\BK$-module $\Pi_\BK$ with the following property: for all $s_0\in \frac{1}{2}+\BZ$ and
   all characters  $\chi_\BK: \BK^\times \rightarrow \BC^\times$ such that $\BQ(\chi_\BK)$ is a number field, the linear functional
    \[
          \CG(\chi_\BK)^n\cdot Z^\circ(\,\cdot\,,s_0,\chi_\BK) \in \Hom_{H_\BK}(\Pi_\BK\otimes  \xi_{\chi_\BK,s_0-\frac{1}{2}},\BC)
   \]
       is defined over  $\BQ(\Pi_\BK, \eta_\BK,\chi_\BK)$. Here $\CG(\chi_\BK)$ is the local Gauss sum defined in Section \ref{subsection: gauss}. Moreover, if $\Pi_\BK$ and $\eta_\BK$ are unramified, the conductor of $\psi_\BK$ is the ring of integers in $\BK$, and $v_\BK^\circ\in \Pi_\BK$ is the spherical vector such that $\lambda_\BK(v_\BK^\circ)=1$, then
  $v_\BK^\circ$ is defined over $\BQ(\Pi_\BK, \eta_\BK)$ with this rational structure.
   \end{thm}

We give a proof of Theorem \ref{0lemma: Z circ f rational} as follows. Let $\overline \BQ\subset \BC$ denote the field of the algebraic numbers.
Unless otherwise specified,  a subscript group or a superscript group respectively indicates the coinvariant space or the invariant space of a group representation over any field.
\begin{lemma}\label{lemus}
There is a $\BQ(\Pi_\BK)$-rational structure $\Pi_\BK^\circ$ of the representation $\Pi_\BK$ such that
$\lambda_{\BK}(\Pi_\BK^\circ)\subset \overline \BQ$.
\end{lemma}

\begin{proof}
Let $\Pi_\BK'$ be a $\BQ(\Pi_\BK)$-rational structure of $\Pi_\BK$. Then $\Pi_\BK'\otimes_{\BQ(\Pi_\BK)} {\overline \BQ}$ is a $\overline \BQ$-rational structure of $\Pi_\BK$. The character $\eta_\BK^{-1}\otimes\psi_\BK^{-1}$ takes value in $\overline \BQ^\times$ and we write $(\eta_\BK^{-1}\otimes\psi_\BK^{-1})_{\overline \BQ}$ for $\overline \BQ^\times$-valued character of the Shalika subgroup $S_\BK$ corresponding to the character $\eta_\BK^{-1}\otimes\psi_\BK^{-1}$. Then, we have an identification
\begin{equation}\label{coinvs}
  ((\Pi_\BK'\otimes_{\BQ(\Pi_\BK)} \overline \BQ)\otimes (\eta_\BK^{-1}\otimes \psi_\BK^{-1})_{\overline \BQ})_{S_\BK}\otimes_{\overline \BQ} \BC= (\Pi_\BK\otimes (\eta_\BK^{-1}\otimes \psi_\BK^{-1}))_{S_\BK}.
\end{equation}
 The right-hand side of \eqref{coinvs} is nonzero, so is the left-hand side. Thus, there exists a nonzero linear functional
$\lambda_\BK' \in \Hom_{S_\BK}(\Pi_\BK'\otimes_{\BQ(\Pi_\BK)}(\eta^{-1}_\BK\otimes\psi^{-1}_\BK)_{\overline \BQ}, \overline \BQ)$. 
By the uniqueness of Shalika functionals (Lemma \ref{uniquesh}), there exists a nonzero complex number $c$ such that
$\lambda_\BK' = c\cdot\lambda_\BK\lvert_{\Pi_\BK'}$. 
Then $\Pi_\BK^\circ := c \cdot \Pi_\BK'$ is the desired $\BQ(\Pi_\BK)$-rational structure of $\Pi_\BK$.
\end{proof}

Write $p$ for the residue characteristic of $\BK$.
Let $\mu_p^\infty\subset \BC^\times$ denote the subgroup consisting of all the $p$-th power roots of unity and $\BQ(\mu_p^\infty)\subset \BC$ be the subfield generated  by  $\mu_p^\infty$. Given a field $E_1$ and a subfield $E_0$ of it, we write $\mathrm{Aut}(E_1/E_0)$ for the group of all field automorphisms of $E_1$ fixing all elements of $E_0$.

Recall the cyclotomic character
\[
  \mathrm{Aut}(\BQ(\mu_p^\infty)/\BQ) \rightarrow \BZ_p^\times,\quad \sigma\mapsto t_{\sigma,p},
\]
which is a topological isomorphism specified by requiring that
$\sigma(x)=x^{t_{\sigma,p}}$, for all $x\in \mu_p^\infty$.
Write $\sigma\mapsto t_{\sigma,\BK}$ for the composition of
 \[
 \Aut(\BC/\BQ)\xrightarrow{\textrm{restriction} }\mathrm{Aut}(\BQ(\mu_p^\infty)/\BQ) \xrightarrow{\sigma\mapsto t_{\sigma,p}} \BZ_p^\times \subset \BK^\times.
 \]
Denote
$
\mathbf{t}_{\sigma,\BK}:= \left[
                      \begin{array}{cc}
                        t_{\sigma,\BK} \cdot 1_n & 0 \\
                        0 & 1_n \\
                      \end{array}
                    \right]\in H_\BK\subset G_\BK.
                    $
The Galois group $\Aut(\overline \BQ/\BQ(\Pi_\BK))$ is a topological group as usual. We equip $\Aut(\BC/\BQ(\Pi_\BK))$ with the coarsest topology such that the natural map
   $\Aut(\BC/\BQ(\Pi_\BK))\rightarrow  \Aut(\overline \BQ/\BQ(\Pi_\BK))$ is continuous.

Let $\Pi_\BK^\circ$ be a $\BQ(\Pi_\BK)$-rational structure of $\Pi_\BK$ as in Lemma \ref{lemus}. Write $\Pi_\BK^*$ for the space of all linear functionals on $\Pi_\BK$. For each $\sigma\in \Aut(\BC/\BQ(\Pi_\BK))$ and $\lambda\in \Pi_\BK^*$, write $\,^\sigma\lambda$ for the composition of
\[
  \Pi_\BK\xrightarrow{\mathbf{t}_{\sigma,\BK}^{-1}}\Pi_\BK\xrightarrow{\beta_\sigma^{-1}} \Pi_\BK\xrightarrow{\lambda}\BC\xrightarrow{\sigma} \BC,
\]
where $\beta_\sigma: \Pi_\BK\rightarrow \Pi_\BK$ is the $\sigma$-sesquilinear map specified by
$\beta_\sigma(au)=\sigma(a)u$, for $u\in \Pi_\BK^\circ,\, a\in \BC$.
This defines a sesquilinear action of   $\Aut(\BC/\BQ(\Pi_\BK))$ on $\Pi_\BK^*$.

\begin{lemma}\label{fixsf}
Under the above action, the Shalika functional $\lambda_\BK$ is fixed by an open subgroup of   $\Aut(\BC/\BQ(\Pi_\BK))$.
\end{lemma}
\begin{proof}
Fix an element $v_0\in \Pi_\BK^\circ$ such that $\langle \lambda_\BK, v_0\rangle\neq 0$. By Lemma \ref{lemus}, there is an open subgroup $\Gamma_0$ of  $\Aut(\BC/\BQ(\Pi_\BK))$ such that for all $\sigma\in \Gamma_0$, 
$\sigma(\langle \lambda_\BK, v_0\rangle)=\langle \lambda_\BK, v_0\rangle$ and $\mathbf{t}_{\sigma,\BK}.v_0=v_0$ hold. 
It is routine to check that
${^\sigma \lambda}_\BK\in \Hom_{S_\BK}(\Pi_\BK, \eta_\BK\otimes \psi_\BK)$ whenever $\sigma\in \Aut(\BC/\BQ(\Pi_\BK, \eta_\BK))$.
By the uniqueness of  Shalika functionals (Lemma \ref{uniquesh}),  this implies that
${^\sigma \lambda}_\BK=a_\sigma \lambda_\BK$ for some $a_\sigma\in \BC$.

Now assume that $\sigma\in \Gamma_0\cap\Aut(\BC/\BQ(\Pi_\BK, \eta_\BK))$. Then
$$
  a_\sigma\langle \lambda_\BK, v_0\rangle
  = \langle \,^\sigma \lambda_\BK, v_0\rangle
 = \sigma(\langle  \lambda_\BK, \beta_\sigma^{-1} (\mathbf{t}_{\sigma,\BK}^{-1}. v_0)\rangle)
  = \langle  \lambda_\BK,  v_0\rangle.
$$
Thus $a_\sigma=1$ and the lemma follows.
\end{proof}

 \begin{lemma}\label{actgamma}
There is a unique  action of $\Aut(\BC/\BQ(\eta_\BK, \Pi_\BK))$ on $\Pi_\BK$ such that
   \begin{equation}\label{eqacts}
   \langle \lambda_\BK, g.(\sigma.v)\rangle=\sigma(\langle \lambda_\BK, (\mathbf t_{\sigma, \BK}^{-1} g).v\rangle),
   \end{equation}
   for all $v\in \Pi_\BK$, $g\in G_\BK$ and  $\sigma\in \Aut(\BC/\BQ(\eta_\BK, \Pi_\BK))$.
   \end{lemma}
  \begin{proof} The uniqueness is clear. For the proof of the existence, let $\Aut(\BC/\BQ(\eta_\BK))$  act on the induced representation $\Ind_{S_\BK}^{G_\BK} (\eta_\BK\otimes \psi_\BK)$ by
$(\sigma.f)(x):=\sigma(f(\mathbf{t}_{\sigma,\BK}^{-1}. x))$, for $\sigma\in \Aut(\BC/\BQ(\eta_\BK))$, $f\in \Ind_{S_\BK}^{G_\BK} (\eta_\BK\otimes \psi_\BK)$ and $x\in G_\BK$.
 By the uniqueness of  Shalika functionals (Lemma \ref{uniquesh}), and using the embedding $\Pi_\BK\hookrightarrow \Ind_{S_\BK}^{G_\BK} (\eta_\BK\otimes \psi_\BK)$ induced by $\lambda_\BK$, this action
  induces  an action as required in the lemma.
   \end{proof}

\begin{lemma}\label{fixpicirc}
If $\sigma\in \Aut(\BC/\BQ(\eta_\BK, \Pi_\BK))$ fixes $\lambda_\BK$, then $\sigma$ fixes all elements in  $\Pi_\BK^\circ$ 
under the action of $\Aut(\BC/\BQ(\eta_\BK, \Pi_\BK))$ on $\Pi_\BK$ as in Lemma \ref{actgamma}.
\end{lemma}

\begin{proof}
Suppose that $\sigma\in \Aut(\BC/\BQ(\eta_\BK, \Pi_\BK))$ fixes $\lambda_\BK$. For every $v\in \Pi_\BK^\circ$,  \eqref{eqacts} 
implies that
$\langle \lambda_\BK, g.(\sigma.v)\rangle=\langle \lambda_\BK, g.v\rangle$ for all $g\in G_\BK$.
This implies $\sigma.v=v$ by the irreducibility of $\Pi_\BK$.
\end{proof}

The following is the key ingredient of the proof of Theorem  \ref{0lemma: Z circ f rational}.

 \begin{prop}\label{rationalsh}
 The space $ \Pi_\BK^{\Aut(\BC/\BQ(\eta_\BK, \Pi_\BK))}$ is a $\BQ(\eta_\BK, \Pi_\BK)$-rational structure of $\Pi_\BK$, that is,
$\Pi_\BK^{\Aut(\BC/\BQ(\eta_\BK, \Pi_\BK))}\otimes_{\BQ(\eta_\BK, \Pi_\BK)} \BC=\Pi_\BK$.
 \end{prop}
 
\begin{proof}
By Lemmas \ref{fixsf} and \ref{fixpicirc}, an open subgroup of $\Aut(\BC/\BQ(\Pi_\BK))$ fixes $\Pi_\BK^\circ$ pointwise. In particular
$\Pi_\BK^{\Aut(\BC/\overline \BQ)}\supset\Pi_\BK^\circ$. 
This easily implies that
\begin{equation}\label{rpi1}
   \Pi_\BK^{\Aut(\BC/\overline \BQ)}=\Pi_\BK^\circ \otimes_{\BQ(\Pi_\BK)} \overline \BQ,\quad \textrm{and it is a $\overline \BQ$-rational form of $\Pi_\BK$}.
\end{equation}
Hence every element of $\Pi_\BK^{\Aut(\BC/\overline \BQ)}$ is fixed by an open subgroup of $\Aut(\overline \BQ/\BQ(\Pi_\BK, \eta_\BK))$. This implies that
$\Pi_\BK^{\Aut(\BC/\BQ(\eta_\BK, \Pi_\BK))}=\left(\Pi_\BK^{\Aut(\BC/\overline \BQ)}\right)^{\Aut(\overline \BQ/\BQ(\eta_\BK, \Pi_\BK))}$
is a $\BQ(\eta_\BK, \Pi_\BK)$-rational sturcture of $\Pi_\BK^{\Aut(\BC/\overline \BQ)}$, by \cite[Proposition 11.1.6]{TASpringerLinearAlgebraicGroups}. With \eqref{rpi1}, this proves the proposition.
\end{proof}

\begin{rk}\label{rkrs}
 The corresponding result of Proposition \ref{rationalsh} for Whittaker models has been proved by Harder for the $\GL_2$-case in \cite[page 80]{Hd83} and by
J. Mahnkopf for the general case in \cite[page 592]{M}. See also \cite[Lemma 3.2]{RS08}. Their proof depends crucially on the fact that the Whittaker functional does not vanish on new vectors.
It is suggested in the published version of  \cite{GrobnerRaguhramShalikaModelArithmetic} that the same idea may be applied to the Shalika functionals and
prove the statement of our Proposition 3.6.
However, to the best of our knowledge, the non-vanishing of the Shalika functional on new
vectors was not known before, and has only been established recently in \cite{GM23} for some unramified characters $\eta_\BK$, which is still not enough
for general ramified situations. Our proof here does not use the theory of new vectors.

\end{rk}

Now we come to the proof of Theorem \ref{0lemma: Z circ f rational}.
By \cite[Proposition 3.1]{ClozelAutomorphicFormsMotive},  a $\BQ(\Pi_\BK,\eta_\BK)$-rational structure on the $G_\BK$-module $\Pi_\BK$ is unique up to scalar multiplication. Then the uniqueness assertion follows by taking $\chi_\BK$ to be the trivial character and $s=\frac{1}{2}$. For the proof of the existence,  we equip  $\Pi_\BK$ with the   $\BQ(\Pi_\BK, \eta_\BK)$-rational structure given by Proposition \ref{rationalsh}. Assume that $\BQ(\chi_\BK)$ is a number field in the rest of this subsection.

By Proposition \ref{prop: properties of F-J int}, for every $v\in \Pi_\BK$, the map $s\mapsto Z^\circ(v, s, \chi_\BK)$ is an element of the ring $\BC[q^{s-\frac{1}{2}}, q^{\frac{1}{2}-s}]$. This ring is defined over $\BQ$ with the obvious $\BQ$-rational structure $\BQ[q^{s-\frac{1}{2}}, q^{\frac{1}{2}-s}]$. By the theory of Godement-Jacquet $L$-functions, it is easy to see that
\[
  \frac{1}{L(s, \Pi_\BK\otimes\chi_{\BK})}\in \BC[q^{s-\frac{1}{2}}, q^{\frac{1}{2}-s}]
\]
is defined over $\BQ(\Pi_\BK, \chi_\BK)$. Hence the fractional ideal
$L(s, \Pi_\BK\otimes\chi_{\BK}) \cdot \BC[q^{s-\frac{1}{2}}, q^{\frac{1}{2}-s}]$ 
also has a natural $\BQ(\Pi_\BK, \chi_\BK)$-rational structure and carries an action of  $\Aut(\BC/\BQ(\Pi_\BK,  \chi_\BK))$ corresponding to this rational structure.
Using this action, for each $\sigma\in \Aut(\BC/\BQ(\Pi_\BK, \eta_\BK, \chi_\BK))$ and each $v\otimes 1\in \Pi_\BK\otimes \xi_{\chi_\BK,j}$ that is defined over $\BQ(\Pi_\BK, \eta_\BK, \chi_\BK)$, we have that
\begin{eqnarray*}
&&\sigma. (Z^\circ(v, s,\chi_\BK) )\\
&=& \frac{1}{L(s, \Pi_\BK\otimes\chi_{\BK})}\int_{\GL_n(\BK)} \sigma\left(\left \langle \lambda_\BK, \mtrtwo{g}{0}{0}{1_n}.v \right \rangle\right) \cdot \chi_\BK(\det g) \cdot \abs{\det g}_{\BK}^{s-\frac{1}{2}} \, dg\\
&=& \frac{1}{L(s, \Pi_\BK\otimes\chi_{\BK})}\int_{\GL_n(\BK)} \left \langle \lambda_\BK, \mtrtwo{t_{\sigma,\BK}\cdot g}{0}{0}{1_n}.(\sigma.v) \right \rangle\cdot \chi_\BK(\det g) \cdot \abs{\det g}_{\BK}^{s-\frac{1}{2}} \, dg\quad(\textrm{by \eqref{eqacts}})\\
&=&  \chi_\BK(t_{\sigma,\BK}^{-n})\cdot \frac{1}{L(s, \Pi_\BK\otimes\chi_{\BK})}\int_{\GL_n(\BK)} \left \langle \lambda_\BK, \mtrtwo{g}{0}{0}{1_n,}.v \right \rangle\cdot \chi_\BK(\det g) \cdot \abs{\det g}_{\BK}^{s-\frac{1}{2}}\, dg\\
&=&  \chi_\BK(t_{\sigma,\BK}^{-n})\cdot Z^\circ(v, s,\chi_\BK).
\end{eqnarray*}
Since $\sigma(\CG(\chi_\BK)) = \chi_\BK(t_{\sigma,\BK})\cdot \CG(\chi_\BK)$, we obtain 
$\sigma. ( \CG(\chi_\BK)^n\cdot Z^\circ(v, s,\chi_\BK) )= \CG(\chi_\BK)^n\cdot Z^\circ(v, s,\chi_\BK)$. 
Therefore, $\CG(\chi_\BK)^n\cdot Z^\circ(v, s,\chi_\BK)\in \BC[q^{s-\frac{1}{2}}, q^{\frac{1}{2}-s}]$ is defined over $\BQ(\Pi_\BK, \eta_\BK, \chi_\BK)$.
 Specifying $s$ to $s_0\in \frac{1}{2}+\BZ$, we have that
  $\CG(\chi_\BK)^n\cdot Z^\circ(v, s_0,\chi_\BK)\in \BQ(\Pi_\BK, \eta_\BK, \chi_\BK).$ This proves the existence assertion of  Theorem  \ref{0lemma: Z circ f rational}.
 Since the group $\Aut(\BC/\BQ(\Pi_\BK, \eta_\BK))$ stabilizes the space of spherical vectors in $\Pi_\BK$,  the last assertion easily follows from Lemma \ref{actgamma}.
This completes the proof of Theorem \ref{0lemma: Z circ f rational}.

\begin{rk}\label{rkrs2}
Using Proposition \ref{rationalsh}, a variant of Theorem \ref{0lemma: Z circ f rational} is proved in \cite[Section 3.9]{GrobnerRaguhramShalikaModelArithmetic}. But  the authors are not quite able to follow the argument in Section 3.9.2 of that paper. 
\end{rk}


    \subsection{Uniform cohomological test vectors}\label{subsection: Uniform Cohomological Vector}

  In the rest of this section, we assume that $\BK$ is archimedean. We retain the notation in
Sections \ref{subsection: Archimedean Friedberg-Jacquet Integral and Uniform Cohomological Vector}--\ref{subsection: Coefficient System of Regular Algebraic Cuspidal Representation}. In particular, $\Pi_\BK$ is an essentially tempered cohomological representation of symplectic type as in \eqref{eq: coh repn symp type induced parameter}.  From now on, we assume that $\eta_\BK$ is the following character of $\BK^\times$:
     \begin{equation}\label{eq: choice of eta}
             \eta_\BK=\prod_{\iota \in \CE_\BK} (\iota|_{\BK^\times})^{w_\iota}.
     \end{equation}
By Corollary \ref{cor: coh repn symp type induced parameter}, $\Pi_\BK$ has a nonzero $(\eta_\BK, \psi_\BK)$-Shalika functional.
As in Section \ref{subsection: Archimedean Friedberg-Jacquet Integral and Uniform Cohomological Vector}, fix a nonzero Shalika functional
$\lambda_{\BK}\in \Hom_{S_\BK}(\Pi_\BK, \eta_\BK\otimes\psi_\BK)$.
Recall the character $\xi_{\chi_\BK,t}$ of $H_\BK$ from \eqref{eq: xichi}. Using the Shalika functional $\lambda_{\BK}$,  the normalized  Friedberg-Jacquet integral produces a nonzero element
\begin{equation}\label{nz-lf}
Z^\circ(\,\cdot\,,s, \chi_\BK)\in \Hom_{H_\BK}(\Pi_\BK\otimes\xi_{\chi_\BK, s-\frac{1}{2}}, \BC), \quad \forall s\in \BC.
\end{equation}

Write $K_\BK$ for the standard maximal compact subgroup of $G_\BK$, namely
\begin{equation}\label{mkk}
  K_\BK:=\begin{cases}
            \mathrm O(2n):=\{g\in G_\BK\mid g^{\mathrm t} g=1_{2n}\} , & \hbox{if $\BK\cong \BR$;}\\
            \mathrm U(2n):=\{g\in G_\BK\mid \bar g^{\mathrm t} g=1_{2n}\} , & \hbox{if $\BK\cong \BC$.} 
           \end{cases}
\end{equation}
where $g^{\mathrm t}$ is the transpose of $g$, and $\bar g^{\mathrm t}$ is the conjugate transpose of $g$. Recall from \cite[Theorem 4.9]{V86} that every irreducible Casselman-Wallach representation of $G_\BK$ has a unique minimal $K_\BK$-type, which is denoted by $\tau_\BK\subset \Pi_\BK$. Write $C_\BK:=K_\BK\cap H_\BK$, which is a maximal compact subgroup of $H_\BK$.


The following theorem is the most technical ingredient
in the proof of Theorem \ref{thm: global period relation}. Its proof is left to Section \ref{section: Identification}.
   \begin{thm}[Uniform Cohomological Test Vector]\label{thm: uniform coh test vector}
      Let the notation and assumptions be as above.   Assume further that
        \begin{equation}\label{btau}
        \Hom_{C_\BK}( \tau_\BK\otimes \xi_{\chi_\BK,0}, \BC)\neq \{0\}.
        \end{equation}
        Then  there exists a vector $v_0\in \tau_\BK$ such that  $Z^\circ(v_0, s, \chi_\BK)=1$ for all $s\in \BC$.
   \end{thm}



    \begin{rk}
We point out some delicate features of the vector $v_0\in \tau_\BK$ in Theorem \ref{thm: uniform coh test vector}.
        \begin{enumerate}
          \item Based on the proof of \cite[Theorem 3.1]{AizenbudGourevitchJacquetUniquenessShalika}, it is hard to tell if the test vector $v$ in Part (4) of Proposition \ref{prop: properties of F-J int}
is $K_\BK$-finite or not, and hence it seems impossible to deduce that it lives in the minimal $K_\BK$-type.
          \item  For $\BK=\BR$, it was first proved in \cite[Theorem 5.1]{SunCohomologicalDistinguishedRepn} that for every $s\in \BC$, there is a vector $v_{0,s}\in \tau_\BK$ such that  $Z^\circ(v_{0,s}, s, \chi_\BK)=1$. Such a vector is called a
{\bf cohomological test vector}.  Since the cohomological test vector $v_0$ in Theorem \ref{thm: uniform coh test vector} is uniform for all $s$, it seems appropriate to call it a
{\bf uniform cohomological test vector}. As far as the authors know, Theorem \ref{thm: uniform coh test vector} establishes the existence of such a uniform cohomological test vector for high rank groups for the
first time.
\item Theorem \ref{thm: uniform coh test vector} is more general than that is needed for the proof of Theorem \ref{thm: global period relation}. In fact, Theorem \ref{thm: global period relation} assumes that
the coefficient system of $\Pi$ is balanced. This assumption will imply the branching condition \eqref{btau} needed in Theorem \ref{thm: uniform coh test vector}.  More detailed discussion on this issue can be found in the proof of Theorem \ref{thm: Non-vanishing hypothesis} in Section \ref{section: Proof of Critical place and Non-vanishing Modular Symbol}.
        \end{enumerate}
    \end{rk}


     \subsection{Non-vanishing hypothesis at infinity}\label{subsection: Critical Place and The Non-Vanishing Hypothesis at Infinity}
 In the proof of Theorem \ref{thm: global period relation}, we also need another technical result: a {\sl non-vanishing result of archimedean modular symbols} for the case under consideration.
Such a result was already obtained in \cite[Theorem A.3]{SunCohomologicalDistinguishedRepn} when the local field $\BK$ is real.
In the literature, the modular symbols are only studied when  a numerical coincidence holds: the real dimension of the modular symbol is either the top or bottom non-vanishing degree of cohomology. Although such a numerical coincidence is no longer valid when $\BK$ is complex, our work in this paper shows that such a {\sl non-vanishing} result can still be
established. Consequently, we are able to attack the global period relation problem for any number field under a mild assumption that the coefficient system of $\Pi$ is balanced.


Recall that $\Pi_\BK$ is assumed to be of symplectic type. In the \textit{rest of this subsection}, we  assume that the coefficient system $F_\BK=\otimes_{\iota\in \CE_\BK} F_\iota$ of  $\Pi_\BK$ is balanced.
Let $\frac{1}{2}+j\in \frac{1}{2}+\BZ$ be a critical place of $\Pi_\BK$. Write
\begin{equation}\label{xikj}
\xi_{\BK,j}:=\otimes_{\iota\in \CE_\BK} ({\det}^j\otimes {\det}^{-j-w_\iota}),
\end{equation}
which may be viewed as an algebraic character of $H(\BK\otimes_\BR \BC)=\prod_{\iota\in \CE_\BK} H_\BC$.
Using Proposition \ref{prop: critical place}, we  fix a nonzero element
     \begin{equation}\label{eq: linear function coeffi system}
\lambda_{F_\BK, j} \in  \Hom_{H_\BK}(F_\BK^\vee\otimes\xi_{\BK, j}^\vee, \BC).
\end{equation}
Write $\sgn_\BK$ for the quadratic character of $\BK^\times$ that is non-trivial if and only if $\BK\cong\BR$.  Specifying the linear functional \eqref{nz-lf} to the case when $\chi_\BK$ equals $\sgn_\BK^j$ and $s=\frac{1}{2}+j$, we get a nonzero linear functional
\begin{equation}\label{nz-lf2}
Z^\circ(\,\cdot\,,\frac{1}{2}+j, \sgn_\BK^j)\in \Hom_{H_\BK}(\Pi_\BK\otimes\xi_{\BK,j}, \BC).
\end{equation}

Identify $\BR_+^\times$
with a subgroup of  $G_\BK$ via the diagonal embedding. Then $\BR^\times_+\subset H_\BK^0\subset G_\BK^0$. The Lie algebra of $\BR^\times_+$ is identified with $\BR$. 
Define a real vector space
$\fq_\BK:=(\fc_\BK\oplus \BR)\backslash\fh_\BK$, 
where $\fh_\BK$ is the Lie algebra of $H_\BK$ and $\fc_\BK$ is the Lie algebra of $C_\BK$.    Define
     \begin{equation}\label{eq: local dimension}
d_\BK :=\dim \fq_\BK= 
    \begin{cases}
                          n^2+n-1, & \hbox{if $\BK\simeq\BR$;} \\
                            2n^2-1, & \hbox{if $\BK\simeq\BC$.}
                        \end{cases}
                      \end{equation}
The {\sl archimedean modular symbol}, which is denoted by $\CP_{\BK,j}$, is defined to be the composition of the following maps:
\begin{eqnarray}
 \label{cpkj} 
 \RH^{d_{\BK}}_\mathrm{ct}(\BR^\times_+\backslash G_\BK^0; \Pi_\BK\otimes F_\BK^\vee)
 \nonumber &\rightarrow & \RH^{d_{\BK}}_\mathrm{ct}(\BR^\times_+\backslash H_\BK^0; \Pi_\BK\otimes F_\BK^\vee)\\
  &&=\RH^{d_{\BK}}_\mathrm{ct}(\BR^\times_+\backslash H_\BK^0; (\Pi_\BK\otimes \xi_{\BK, j})\otimes (F_\BK^\vee\otimes\xi_{\BK, j}^\vee ))\\
 \nonumber &\rightarrow & \RH^{d_{\BK}}_\mathrm{ct}(\BR^\times_+\backslash H_\BK^0;  \BC),
\end{eqnarray}
where the first arrow is the cohomology restriction map, and the last arrow is the map induced by the linear functional
\begin{equation}\label{zcirc0}
(\Pi_\BK\otimes \xi_{\BK, j})\otimes (F_\BK^\vee\otimes\xi_{\BK, j}^\vee )\xrightarrow{Z^\circ(\,\cdot\, , \frac{1}{2}+j, \sgn_\BK^j)\otimes \lambda_{F_\BK,j}}   \BC.
\end{equation}
The following non-vanishing hypothesis at infinity, which be will proved in Section \ref{section: Proof of Critical place and Non-vanishing Modular Symbol}, asserts that the archimedean modular symbol is nonzero.
     \begin{thm}[Non-vanishing Hypothesis at Infinity]\label{thm: Non-vanishing hypothesis}
          Assume that the coefficient system $F_\BK$ of $\Pi_\BK$ is balanced. Let $\frac{1}{2}+j$ be a critical place of $\Pi_\BK$. Then the archimedean modular symbol
        $$\CP_{\BK,j} : \RH^{d_{\BK}}_\mathrm{ct}(\BR^\times_+\backslash G_\BK^0; \Pi_\BK\otimes F_\BK^\vee)  \rightarrow  \RH^{d_{\BK}}_\mathrm{ct}(\BR^\times_+\backslash H_\BK^0;  \BC)$$
          defined above is nonzero.
     \end{thm}


     \subsection{Archimedean period relations}\label{subsection: local period relations}
     Fix an embedding
          \begin{equation}\label{eq: embedding no i}
               \gamma_{\BK}: \BK^{2n}\hookrightarrow \fgl_{2n}(\BK) = \fg_\BK
           \end{equation}
     which sends $(a_1,a_2,\cdots,a_{2n})$ to the matrix
    \begin{equation}
    \left[
  \begin{array}{cccccccc}
    \frac{a_1+a_{2n}}{2} & 0 & \cdots & 0 & 0 &\cdots & 0 & \frac{a_1-a_{2n}}{2}  \\
     0&\frac{a_2+a_{2n-1}}{2}  & \cdots & 0 & 0 &\cdots & \frac{a_2-a_{2n-1}}{2}  & 0 \\
     \cdots & \cdots & \cdots & \cdots & \cdots &\cdots & \cdots &\cdots  \\
   0 & 0 & \cdots & \frac{a_n+a_{n+1}}{2} & \frac{a_n-a_{n+1}}{2} &\cdots & 0 & 0  \\
   0 & 0 & \cdots &\frac{a_n-a_{n+1}}{2}   & \frac{a_{n+1}+a_n}{2} &\cdots & 0 & 0  \\
  \cdots & \cdots & \cdots & \cdots & \cdots &\cdots & \cdots &\cdots  \\
   0& \frac{a_2-a_{2n-1}}{2} & \cdots & 0 & 0 &\cdots & \frac{a_{2n-1}+a_2}{2}  & 0 \\
   \frac{a_{1}-a_{2n}}{2} & 0 & \cdots & 0 & 0 &\cdots &  & \frac{a_{2n}+a_1}{2}  \\
  \end{array}
   \right].
   \end{equation}
   This embedding identifies $\BK^{2n}$ as a Cartan subalgebra $\ft_\BK$ of $\fg_\BK$. With the so obtained Cartan subalgebra $\ft_\BK$, the root system of $\fg_\BK$ is
$\Phi(\fg_\BK, \ft_\BK) = \set{\pm(\eps_i-\eps_j)}{1\leq i< j \leq 2n}.$ 
   Here $\eps_1, \eps_2, \cdots, \eps_{2n}$ is the standard basis of $\BK^{2n}$. The positive root system
$\Phi^+(\fg_\BK, \ft_\BK) = \set{\eps_i-\eps_j}{1\leq i< j \leq 2n}$
defines a Borel subalgebra $\fb_\BK$ of $\fg_\BK$. One checks that $\fb_\BK$ is transversal to $\fh_\BK$. More precisely, we have that
\begin{equation}\label{transvb}
     \begin{cases}
        \fb_\BK+\fh_\BK=\fg_\BK,\\
      \fb_\BK\cap \fh_\BK=\{\gamma_\BK(a_1, a_2, \cdots, a_n,a_n, \cdots, a_2, a_1)\mid a_1, a_2, \cdots, a_n\in \BK\}.
     \end{cases}
  \end{equation}
Write $\fu_\BK$ for the nilpotent radical of $\fb_\BK=\ft_\BK\ltimes \fu_\BK$. The $\fu_\BK$-invariants in 
$F_\BK^\vee$ is denoted by $(F_\BK^\vee)^{\fu_\BK}$, which is one-dimensional.
The following theorem will be proved in Section~\ref{section: Proof of Critical place and Non-vanishing Modular Symbol}.
     
\begin{thm}[Archimedean Period Relations]\label{thm: archi period relations}
        Suppose that the coefficient system $F_\BK$ is balanced, and for some $v_0\in (F_\BK^\vee)^{\fu_\BK}$, $\lambda_{F_\BK, j}(v_0\otimes 1)=1$ for all critical places $\frac{1}{2}+j$ of $\Pi_\BK$. Then  the linear map
$$
\epsilon_\BK^{j} \cdot \mathrm i^{-jn\cdot[\BK:\BR]}\cdot  \CP_{\BK,j}\ \colon\ \RH^{d_{\BK}}_\mathrm{ct}(\BR^\times_+\backslash G_\BK^0; \Pi_\BK\otimes F_\BK^\vee)  \rightarrow  \RH^{d_{\BK}}_\mathrm{ct}(\BR^\times_+\backslash H_\BK^0;  \BC)
$$
is independent of the critical place $\frac{1}{2}+j$ of $\Pi_\BK$. Here $\epsilon_\BK=1$ when $\BK\cong \BC$, and $\epsilon_\BK\in \{\pm 1\}$ is defined as in \eqref{epsilonk} when $\BK\cong \BR$. 
\end{thm}

\begin{rk}
When $\BK\cong\BR$, assuming the existence of uniform cohomological vectors (Theorem \ref{thm: uniform coh test vector}), a weak form of Theorem \ref{thm: archi period relations} is proved in \cite[Section 7]{Jan18}.
\end{rk}

   \section{Proof of Theorem \ref{thm: global period relation}}\label{subsection: proof of global period relation}

We introduce some more notation needed for the proof of Theorem \ref{thm: global period relation}.
We write $G_v$ and $ H_v$ for the local components at place $v$ of $G(\BA)$ and $H(\BA)$, respectively.
Write
   $G_\infty = \prod_{v\mid\infty} G_v$ and $H_\infty = \prod_{v\mid\infty} H_v.$
Write $K_\infty$ for the standard maximal compact subgroup of $G_\infty$, i.e., 
$K_\infty = \prod_{v\mid\infty} K_v,$
where $K_v$ is the  standard maximal compact subgroup of $G_v$ as in \eqref{mkk}. As before, we have a maximal compact subgroup $C_v:=H_v\cap K_v$ of $H_v$. Write 
$C_\infty := \prod_{v\mid\infty} C_v$, which is a maximal compact subgroup of $H_\infty$.

Write $\fg_\infty$, $\fh_\infty$, $\fk_\infty$ and $\fc_\infty$ for the Lie algebras of $G_\infty$, $H_\infty$, $K_\infty$ and $C_\infty$, respectively.
Recall that for $v\mid\infty$,  we identify $\BR_+^\times$
with a subgroup of $G_v$ via the diagonal embedding.
In this way, we  further identify $\BR^\times_+$ with a subgroup of  $G_\infty$  via diagonal embedding.

\subsection{Reformulation of Friedberg-Jacquet}

    Recall the character $\eta\otimes \psi$ defined right below \eqref{eq: global shalika subgroup}. Only in this subsection, let $\Pi$ be an irreducible cuspidal automorphic representation of $G(\BA)$ with a nonzero $(\eta, \psi)$-Shalika integral, but not necessarily regular or algebraic.

    Note that the group $S(\BA)$ is unimodular. Fix once and for all an $S(\BA)$-invariant positive Borel measure on  $(S(\Rk)\BR^\times_+)\backslash S(\BA)$. The nonzero $(\eta, \psi)$-Shalika integral of $\Pi$ defines a continuous $S(\BA)$-invariant linear functional
    \begin{equation}\label{gsp}
       \lambda_\BA: \Pi\otimes (\eta\otimes\psi)^{-1}\rightarrow \BC, \quad \varphi\otimes 1\mapsto \int_{(S(\Rk)\BR^\times_+)\backslash S(\BA)} \varphi(g)\cdot (\eta\otimes\psi)^{-1}(g)  \, dg.
    \end{equation}
As  before  we write $\Pi=\Pi_\infty\otimes \Pi_f$ with the infinite part $\Pi_\infty$ and the finite part $\Pi_f$. Similarly write $\eta=\eta_\infty\otimes \eta_f$ and $\psi=\psi_\infty\otimes \psi_f$.
By using the uniqueness of Shalika models (Lemma \ref{uniquesh}), we are able to factorize the global Shalika integral $\lambda_\BA$ as $\lambda_\BA=\lambda_\infty\otimes \lambda_f$, where
\begin{equation}\label{flambda}
\lambda_\infty\in \Hom_{S_\infty}(\Pi_\infty\otimes (\eta_{\infty}\otimes\psi_{\infty})^{-1}, \BC)\quad 
{\rm and}\quad 
\lambda_f\in \Hom_{S(\BA_f)}(\Pi_f\otimes (\eta_f\otimes\psi_f)^{-1}, \BC).
\end{equation}
Here $S_\infty:=\prod_{v\mid \infty} S_v$ and $S_v:=S(\mathrm k_v)$.
For the automorphic character $\chi: \Rk^\times \backslash \BA^\times\rightarrow \BC^\times$, we define a character
     \begin{equation}\label{eq: temp 01}
      \xi_{\chi,t}:=((\chi\circ\det)\cdot \abs{\det}_\BA^t)\otimes (((\chi^{-1}\cdot \eta^{-1})\circ \det)\cdot\abs{\det}_\BA^{-t}): H(\BA)\rightarrow \BC^\times,\quad \forall t\in \BC.
     \end{equation}
     Here $\abs{\,\cdot\,}_\BA$ is the normalized absolute value on $\BA^\times$. 
     We write $\abs{\,\cdot\,}_\BA=\abs{\,\cdot\,}_{\BA_\infty}\otimes \abs{\,\cdot\,}_{\BA_f}$ and $\xi_{\chi,t}=\xi_{\chi_\infty, t}\otimes \xi_{\chi_f, t}$ as usual.

We take the product at all finite local places of the normalized local Friedberg-Jacquet integrals as defined in Proposition \ref{prop: properties of F-J int} and denote it by
    \begin{equation}\label{eq: normalized finite F-J int}
    Z^\circ(\,\cdot\,,  s, \chi_f) ,\in \Hom_{H(\BA_f)}(\Pi_f\otimes  \xi_{\chi_f,s-\frac{1}{2}},\BC),\qquad s\in \BC.
       \end{equation}
When the real part of $s$ is sufficiently larger, this is explicitly given by
     \[
           Z^\circ(v, s, \chi_f) := \frac{1}{L(s, \Pi_f\otimes\chi_{f})}\int_{\GL_n(\BA_f)} \left \langle \lambda_f, \mtrtwo{g}{0}{0}{1_n}.v \right \rangle \cdot \chi_f(\det g) \cdot \abs{\det g}_{\BA_f}^{s-\frac{1}{2}} \,dg.
   \]
        Here $dg$ is the  Haar measure on $\GL_n(\BA_f)$ such that a maximal compact subgroup of it has total volume 1.

Define a real vector space
$\fq_\infty:=(\fc_\infty\oplus \BR)\backslash\fh_\infty$, where $\BR$ is identified with the Lie algebra of $\BR^\times_+$. Set
\begin{equation}\label{eq: temp 13}d_\infty := \dim \fq_\infty.\end{equation}
Unless otherwise specified, we use a superscript $\BC$ over a real vector space to indicates the complexification.
Given a $\wedge^{d_\infty}\fq_\infty^\BC$-valued Haar measure $\mu$ on $\GL_n(\BA_\infty)$, we similarly define the  normalized Friedberg-Jacquet integral at infinity
     $$Z_\mu^\circ(\,\cdot\,, s, \chi_\infty)\in \Hom_{H_\infty^0}( (\wedge^{d_\infty}\fq_\infty^\BC)^*\otimes\Pi_\infty\otimes \xi_{\chi_\infty, s-\frac{1}{2}}, \BC),\qquad s\in \BC,$$
where $\wedge^{d_\infty}\fq_\infty^\BC$ carries the trivial action  of $H_\infty^0$, and  ``$\,^*$" indicates the (linear) dual space.
When the real part of $s$ is sufficiently large, this is   explicitly given by
      \begin{equation}\label{eq: infinite place F-J differential form int}
         \begin{aligned}
         &Z_\mu^\circ(\omega\otimes v\otimes 1, s, \chi_\infty)\\
         :=& \frac{1}{L(s,\Pi_\infty\otimes \chi_{\infty})}\int_{\GL_n(\BA_\infty)} \left \langle \lambda_\infty, \mtrtwo{g}{0}{0}{1_n}.v\right \rangle \cdot \chi_\infty(\det g)\cdot \abs{\det g}_\infty^{s-\frac{1}{2}}\, d \langle \mu, \omega\rangle (g).
         \end{aligned}
               \end{equation}
Write $\CX_H := (H(\Rk)\BR^\times_+)\quo H(\BA)/C^0_\infty$.      
Fix once and for all an $H_\infty$-invariant orientation on $\BR^\times_+\backslash H_\infty/C_\infty^0$. 
Then $\CX_H/C_f $ is an oriented manifold, whenever $C_f$ is a sufficiently small open compact subgroup of $H(\BA_f)$.
We define a global linear period integral
               \begin{eqnarray}\label{eq: global linear period}
                \mathcal L_s : (\wedge^{d_\infty}\fq_\infty^\BC)^*\otimes \Pi\otimes \xi_{\chi, s-\frac{1}{2}} &\rightarrow &\BC,\nonumber\\
                    \omega\otimes  \varphi\otimes 1 &\mapsto & \frac{1}{[C_{\mathrm{max},f}: C_f]} \int_{\CX_H/C_f} [\varphi]\cdot [\omega],
                  \end{eqnarray}
where $C_f$ is sufficiently small so that both  $\varphi$ and $\xi_{s-\frac{1}{2}}$ are $C_f$-invariant, $C_{\mathrm{max},f}$ is a maximal compact subgroup of $H(\BA_f)$ containing $C_f$, $[\varphi]$ denotes the function
                \[
                 h\mapsto \int_{C_\infty^0} \varphi(h c)\cdot \xi_{\chi, s-\frac{1}{2}}(hc) dc\qquad (\textrm{$dc$ is the Haar measure on $C_\infty^0$ with total volume 1})
                \]
               on $\CX_H/C_f$, and $[\omega]$ is the differential form on  $\CX_H/C_f$ whose pull-back to $\BR^\times_+\backslash H(\BA)/(C_\infty^0 C_f) $ is the $H(\BA)$-invariant differential form corresponding to $\omega$.
The following is just a reformulation of \cite[Proposition 2.3]{FriedbergJacquetLinearPeriods}.
      \begin{prop}\label{prop: F-j int equal linear period}
               There exists a $\wedge^{d_\infty}\fq_\infty^\BC$-valued Haar measure $\mu$ on $\GL_n(\BA_\infty)$ such that for all automorphic characters $\eta, \chi: \Rk^\times \backslash \BA^\times\rightarrow \BC$, all irreducible cuspidal automorphic representations $\Pi$ of $\GL_{2n}(\BA)$ with nonzero $(\eta, \psi)$-Shalika integral, and all $s\in\BC$, the diagram
                 \[
                   \begin{CD}
          ((\wedge^{d_\infty}\fq_\infty^\BC)^*\otimes\Pi_\infty\otimes \xi_{\chi_\infty, s-\frac{1}{2}})\otimes (\Pi_f\otimes \xi_{\chi_f,s-\frac{1}{2}})  @> Z_\mu^\circ(\,\cdot\,, s, \chi_\infty)\otimes Z^\circ(\,\cdot\,, s, \chi_f)>> \BC \\
            @V= VV           @V L(s, \Pi\otimes \chi) VV\\
          (\wedge^{d_\infty}\fq_\infty)^*\otimes \Pi\otimes \xi_{\chi, s-\frac{1}{2}}  @>\mathcal L_{s} >> \BC\\
                     \end{CD}
           \]
           is commutative.
      \end{prop}

\begin{rk}
The measure $\mu$ in Proposition \ref{prop: F-j int equal linear period} is determined by the measure on the quotient $(S(\Rk)\BR^\times_+)\backslash S(\BA)$,
which is used to define the global Shalika integral in \eqref{gsp}.

\end{rk}
\subsection{Modular symbols}\label{subsection: sheaf cohomology}

 As in Section \ref{subsection: Coefficient System of Regular Algebraic Cuspidal Representation}, suppose that $\Pi=\Pi_\infty\otimes \Pi_f$ 
     is an irreducible regular algebraic cuspidal automorphic representation of $G(\BA)$ of symplectic type, with coefficient system $F=\otimes_{\iota\in \CE_\Rk} F_\iota$ whose highest weight satisfying \eqref{eq2: 22global purity lemma} and \eqref{eq3: def of w iota}.

Similar to $\CX_H$, we define
      $\CX_G := (G(\Rk)\BR^\times_+)\quo G(\BA) / K_\infty^0.$
     Following \cite[Section 1.1]{HarderEisensteinCohomologyGL2}, for every open compact subgroup $K_f$ of $G(\BA_f)$, the finite dimensional representation $F^\vee$ defines a sheaf on $\CX_G/K_f$, which,
by abuse of notation, is still denoted by $F^\vee$. Write $\RH^{d_\infty}_c(\CX_G/K_f, F^\vee)$ for the sheaf-cohomology group with compact support and define
\[
  \RH^{d_\infty}_c(\CX_G, F^\vee):=\varinjlim_{K_f} \RH^{d_\infty}_c(\CX_G/K_f, F^\vee),
\]
where $K_f$ runs over the directed set of all open compact subgroups of $G(\BA_f)$. Some other cohomology spaces, such as $\RH^{d_\infty}_c(\CX_H, F^\vee)$, are similarly defined without further explanation. It is an important fact for us that the sheaf-cohomology group with compact support can be calculated by the complex of the spaces of fast decreasing differential forms. See \cite[Theorem 5.2]{Bo81} for details.

The cuspidal cohomology $\RH_{\mathrm{cusp}}^{d_\infty}(\CX_G, F^\vee)$
naturally injects into $\RH^{d_\infty}_c(\CX_G, F^\vee)$ (see \cite[Section 5]{Bo81} and \cite[page 123]{ClozelAutomorphicFormsMotive}). Hence by \cite[Lemma 3.15]{ClozelAutomorphicFormsMotive}, we have an injection of cohomology groups:
      \begin{equation}\label{eq: injection into cpt supp coh}
           \begin{aligned} \iota_\Pi : \RH^{d_\infty}_\mathrm{ct}(\BR_+^\times\quo G_\infty^0; \Pi\otimes F^\vee) &\hookrightarrow \RH_{\mathrm{cusp}}^{d_\infty}(\CX_G, F^\vee) \hookrightarrow \RH_c^{d_\infty}(\CX_G, F^\vee).
           \end{aligned}
      \end{equation}
      A well-known result of A. Borel and G. Prasad (see \cite[Lemma 2.7]{AshNonSquareIntegrableCohomologyArithmeticGroups}) asserts that the natural inclusion $\CX_H\hookrightarrow \CX_G$ is a proper map, and defines a homomorphism of cohomological groups
      \begin{equation}\label{eq: iota star}
          \iota^*: \RH_\mathrm{c}^{d_\infty}(\CX_G, F^\vee) \rightarrow \RH_\mathrm{c}^{d_\infty}(\CX_H, F^\vee).
      \end{equation}

As in Theorem \ref{thm: global period relation}, suppose that $\eta$ is an automorphic character such that  $L(s, \Pi, \wedge^2\otimes \eta^{-1})$ has a pole at $s=1$.
By Proposition \ref{prop: strong transfer}, $\Pi$ has a nonzero $(\eta, \psi)$-Shalika integral.
 It was proved in \cite[Theorem 5.3]{GanRaghramArithmeticityPeriodsAutomorphicForms} (also see \cite[Lemma 3.6.1]{GrobnerRaguhramShalikaModelArithmetic}) that the archimedean component $\eta_{\infty}$ of $\eta$
 equals the composition of
 \[
    \BA^\times_\infty=(\Rk\otimes_\BQ \BR)^\times \hookrightarrow (\Rk\otimes_\BQ \BC)^\times =\prod_{\iota\in\CE_\Rk} \BC^\times\xrightarrow{\{x_\iota\}_{\iota\in\CE_\Rk}\mapsto \prod_{\iota\in\CE_\Rk} x_\iota^{w_\iota}}\BC^\times. \]
Let $\frac{1}{2}+j\in \frac{1}{2}+\BZ$ be a critical place of $\Pi$ and
assume that $\chi$ is of finite order and the coefficient system $F$ is balanced.  Recall  from \eqref{eq: temp 01} the character
     \begin{equation}\label{eq2: temp 01}
      \xi_{\chi,j}:=((\chi\circ\det)\cdot \abs{\det}_\BA^{j})\otimes (((\chi^{-1}\cdot \eta^{-1})\circ \det)\cdot\abs{\det}_\BA^{-j})
     \end{equation}
   of $H(\BA)$.  Similar to \eqref{xikj}, define a character
\begin{equation}\label{xii}
\xi_{\infty,j}:=\otimes_{\iota\in \CE_\Rk} ({\det}^j\otimes {\det}^{-j-w_\iota}): H(\Rk\otimes_\BQ \BC)=\prod_{\iota\in\CE_\Rk} H(\BC)\rightarrow \BC^\times.
\end{equation}
It is easily checked that $\xi_{\chi_\infty,j}$, which denotes the archimedean component  of $\xi_{\chi,j}$, agrees with  $\xi_{\infty,j}$ on the group $H_\infty^0$.
We have an obvious injective linear map
      \begin{equation}\label{eq: injection iota j}
          \iota_j\ \colon\ \RH^0_\mathrm{ct}(\BR^\times_+\quo H^0_\infty; \xi_{\chi,j}\otimes\xi_{\infty,j}^\vee)\hookrightarrow\RH_{\mathrm{}}^0(\CX_H, \xi_{\infty,j}^\vee).
      \end{equation}
  By Proposition \ref{prop: critical place}, there is a nonzero element
      \begin{equation}\label{lambdakj}
      \lambda_{F,j}\in\Hom_{H_\infty}(F^\vee\otimes\xi_{\infty,j}^\vee,\BC).
      \end{equation}
       It induces a linear map
      \begin{equation}\label{lcoh}
        \RH_\mathrm{c}^{d_\infty}(\CX_H, F^\vee)\otimes  \RH^0(\CX_H, \xi_{\infty,j}^\vee) \xrightarrow{\lambda_{F,j}}  \RH_{\mathrm c}^{d_\infty}(\CX_H, \BC).
      \end{equation}

Recall that with the fixed $H_\infty$-invariant orientation on $\BR^\times_+\backslash H_\infty/C_\infty^0$, 
$\CX_H/C_f$ is an oriented manifold, for every sufficiently small open compact subgroup $C_f$ of $H(\BA)$. 
The pairing with the fundamental class yields a linear map
      \[
       \int_{\CX_H/C_f} :\quad \RH_{\mathrm c}^{d_\infty}(\CX_H/C_f, \BC)\rightarrow \BC.
      \]
      Now we define a linear map
      \[
       \int_{\CX_H}:\quad \RH_{\mathrm c}^{d_\infty}(\CX_H, \BC)\rightarrow \BC,\quad {\rm with}\ 
       \int_{\CX_H} \alpha= \frac{1}{[C_{\mathrm{max},f}: C_f]} \int_{\CX_H/C_f}\alpha,
      \]
for every maximal open compact subgroup $C_{\mathrm{max},f}$ of $H(\BA_f)$, every sufficiently small open subgroup $C_f$ of $C_{\mathrm{max},f}$, and every $\alpha\in \RH_{\mathrm c}^{d_\infty}(\CX_H/C_f, \BC)\subset \RH_{\mathrm c}^{d_\infty}(\CX_H, \BC)$.

Finally, we define the {\sl modular symbol map} $\CP_{\chi,j}$ to be the composition of the following morphisms:
      \begin{eqnarray*}
   \label{cpj} \RH_\mathrm{ct}^{d_\infty}(\BR^\times_+\quo G_\infty^0; \Pi\otimes F^\vee)\otimes \RH_\mathrm{ct}^0(\BR^\times_+\quo H_\infty^0; \xi_{\chi,j}\otimes\xi_{\infty,j}^\vee)
   \nonumber  &\xrightarrow{\iota_\Pi\otimes\iota_j} & \RH^{d_\infty}_\mathrm{c}(\CX_G, F^\vee)\otimes \RH^0(\CX_H, \xi_{\infty,j}^\vee)\\
     \nonumber   &\xrightarrow{\iota^*\otimes 1} & \RH^{d_\infty}_\mathrm{c}(\CX_H, F^\vee)\otimes \RH^0(\CX_H, \xi_{\infty,j}^\vee)\\
          &\xrightarrow{\eqref{lcoh}} & \RH^{d_\infty}_\mathrm{c}(\CX_H, \BC)\\
       \nonumber   &\xrightarrow{\int_{\CX_H}} & \BC.
      \end{eqnarray*}

    \subsection{Modular symbols at infinity and a commutative diagram}
     Let $\mu$ be a $\wedge^{d_\infty}\fq_\infty^\BC$-valued Haar measure  on $\GL_n(\BA_\infty)$  as in Proposition \ref{prop: F-j int equal linear period}. Recall from \eqref{eq: infinite place F-J differential form int} 
     the normalized Friedberg-Jacquet integral
     $$
    Z_\mu^\circ(\,\cdot\,, s, \chi_\infty)\in \Hom_{H_\infty^0}( (\wedge^{d_\infty}\fq_\infty^\BC)^*\otimes\Pi_\infty\otimes \xi_{\chi_\infty, s-\frac{1}{2}}, \BC)= \Hom_{H_\infty^0}(\Pi_\infty\otimes \xi_{\chi_\infty, s-\frac{1}{2}} ,\wedge^{d_\infty}\fq_\infty^\BC).$$
With the given nonzero element $\lambda_{F,j}\in\Hom_{H_\infty}(F^\vee\otimes\xi_{\infty,j}^\vee,\BC)$, 
      analogous to the archimedean modular symbol $\CP_{\BK,j}$ as defined in Section \ref{subsection: Critical Place and The Non-Vanishing Hypothesis at Infinity}, we define the
{\sl modular symbol at infinity}, which is denoted by $\CP_{\chi_\infty,j}$,  to be the composition of the following maps:
\begin{eqnarray}\label{pinfj}
   \CP_{\chi_\infty,j} &:& \RH_\mathrm{ct}^{d_\infty}(\BR^\times_+ \quo G_\infty^0; \Pi_\infty\otimes F^\vee) \otimes \RH_\mathrm{ct}^0(\BR^\times_+\quo H_\infty^0; \xi_{\chi_\infty,j}\otimes\xi_{\infty,j}^\vee)\\ 
   \nonumber &\rightarrow & \RH_\mathrm{ct}^{d_\infty}(\BR^\times_+ \quo H_\infty^0; \Pi_\infty\otimes F^\vee) \otimes \RH_\mathrm{ct}^0(\BR^\times_+\quo H_\infty^0; \xi_{\chi_\infty,j}\otimes\xi_{\infty,j}^\vee)\\
  \nonumber &&=\RH^{d_{\infty}}_\mathrm{ct}(\BR^\times_+ \quo H_\infty^0; (\Pi_\infty\otimes \xi_{\chi_\infty,j})\otimes (F^\vee\otimes\xi_{\infty, j}^\vee )) 
  \\
  \nonumber &\rightarrow & 
    \RH^{d_{\infty}}_\mathrm{ct}(\BR^\times_+ \quo H_\infty^0; \wedge^{d_\infty}\fq_\infty^\BC)=  \BC,
\end{eqnarray}
where  the first arrow is induced by the restriction map of cohomology, and the last arrow is the map induced by the linear functional
\[
Z_\mu^\circ(\,\cdot\,, \frac{1}{2}+j, \chi_\infty) \otimes \lambda_{F,j}:(\Pi_\infty\otimes \xi_{\chi_\infty,j})\otimes (F^\vee\otimes\xi_{\infty,j}^\vee)\rightarrow \wedge^{d_\infty}\fq_\infty^\BC.
\]
Recall the linear functional $Z^\circ(\,\cdot\,, s, \chi_f)\in \Hom_{H(\BA_f)}(\Pi_f\otimes  \xi_{\chi_f,s-\frac{1}{2}},\BC)$ from \eqref{eq: normalized finite F-J int}.
     Proposition \ref{prop: F-j int equal linear period} has the following cohomological reformulation.
      \begin{prop}\label{prop: main diagram}
         Let the notation and the assumptions be as above.  Then the diagram
         $$\CD
         \RH_\mathrm{ct}^{d_\infty}(\BR^\times_+\quo G_\infty^0; \Pi_\infty\otimes F^\vee)\otimes  \RH_\mathrm{ct}^0(\BR^\times_+\quo H_\infty^0; \xi_{\chi_\infty,j}\otimes\xi_{\infty,j}^\vee) \otimes (\Pi_f\otimes\xi_{\chi_f,j}) @>\CP_{\chi_\infty, j}\otimes Z^\circ(\,\cdot\,, \frac{1}{2}+j, \chi_f)>> \BC \\
           @V = VV @V L(\frac{1}{2}+j, \Pi\otimes\chi) VV  \\
            \RH_\mathrm{ct}^{d_\infty}(\BR^\times_+\quo G_\infty^0; \Pi\otimes F^\vee)\otimes \RH_\mathrm{ct}^0(\BR^\times_+\quo H_\infty^0; \xi_{\chi,j}\otimes\xi_{\infty,j}^\vee) @>\CP_{\chi,j}>> \BC
         \endCD$$
commutes.
      \end{prop}

\begin{proof}
         Set $\fs_\infty:=\fg_\infty/(\BR\oplus\fk_\infty)$. 
         By identifying the continuous group cohomology with the relative Lie algebra cohomology, we get (see \cite[Proposition 9.4.3]{WallachRealReductiveGroups1})
          \begin{equation}
              \begin{aligned}
              \RH_\mathrm{ct}^{d_\infty}(\BR^\times_+\quo G_\infty^0; \Pi_\infty\otimes F^\vee) 
              = \Hom_{K_\infty^0}(\wedge^{d_\infty} \fs_\infty^\BC, \Pi_\infty\otimes F^\vee)
               = \big((\wedge^{d_\infty} \fs_\infty^\BC)^*\otimes \Pi_\infty\otimes F^\vee\big)^{K^0_\infty}.
              \end{aligned}
          \end{equation}
          By unwinding the definition of  $\CP_{\chi_\infty, j}$, the top arrow in the above diagram is identified with the composition of
          the following maps:
          \begin{eqnarray*}
&&\big((\wedge^{d_\infty} \fs_\infty^\BC)^*\otimes \Pi_\infty\otimes F^\vee\big)^{K^0_\infty}\otimes(\xi_{\chi_\infty,j}\otimes\xi_{\infty,j}^\vee)\otimes (\Pi_f\otimes\xi_{\chi_f,j})\\
           &\xrightarrow{\textrm{restriction}} & (\wedge^{d_\infty} \fq_\infty^\BC)^*\otimes \Pi_\infty\otimes F^\vee \otimes(\xi_{\chi_\infty,j}\otimes\xi_{\infty,j}^\vee) \otimes (\Pi_f\otimes\xi_{\chi_f, j})\\
         &&= (\wedge^{d_\infty} \fq_\infty^\BC)^*\otimes (\Pi_\infty\otimes \xi_{\chi_\infty,j})\otimes (\Pi_f\otimes\xi_{\chi_f,j}) \otimes (F^\vee\otimes\xi_{\infty,j}^\vee) \\
         &\xrightarrow{ Z_\mu^\circ(\,\cdot\,, \frac{1}{2}+j, \chi_\infty) \otimes Z^\circ(\,\cdot\,,\frac{1}{2}+j,\chi_f)\otimes \lambda_{F,j}} & \BC.\\
                  \end{eqnarray*}
     By the arguments of \cite[Section 5.6]{Bo81}, using fast decreasing differential forms, it is routine to check that the bottom arrow in the above diagram is identified with the composition of
          the following maps:
          \begin{eqnarray*}
\big((\wedge^{d_\infty} \fs_\infty^\BC)^*\otimes \Pi\otimes F^\vee\big)^{K^0_\infty}\otimes (\xi_{\chi,j}\otimes \xi_{\infty,j}^\vee)
           &\xrightarrow{\textrm{restriction}} & (\wedge^{d_\infty} \fq_\infty^\BC)^*\otimes \Pi\otimes F^\vee\otimes (\xi_{\chi,j}\otimes \xi_{\infty,j}^\vee) \\
         &&= (\wedge^{d_\infty} \fq_\infty^\BC)^*\otimes (\Pi\otimes \xi_{\chi,j}) \otimes (F^\vee\otimes\xi_{\infty,j}^\vee) \\
         &\xrightarrow{\mathcal L_{\frac{1}{2}+j}\otimes \lambda_{F,j}} & \BC.\\
                  \end{eqnarray*}
Finally the Proposition follows directly from Proposition \ref{prop: F-j int equal linear period}.
      \end{proof}

\subsection{Rationality of cohomology groups and the Shalika periods of $\Pi$}

Recall the non-archimedean Shalika functional
$\lambda_f\in \Hom_{S(\BA_f)}(\Pi_f\otimes (\eta_f\otimes\psi_f)^{-1}, \BC)$ 
from \eqref{flambda}, and  the normalized Friedberg-Jacquet integrals
$Z^\circ(\,\cdot\,,  s, \chi_f) ,\in \Hom_{H(\BA_f)}(\Pi_f\otimes  \xi_{\chi_f,s-\frac{1}{2}},\BC)$ with $s\in \BC$, 
from \eqref{eq: normalized finite F-J int}.
\begin{prop}\label{rpif}
 There is a unique  $\BQ(\Pi_f,\eta_f)$-rational structure on the $G(\BA_f)$-module $\Pi_f$ such that
    $\CG(\chi'_f)^n\cdot Z^\circ(\,\cdot\, ,s_0,\chi'_f)\in \Hom_{H(\BA_f)}(\Pi_f\otimes  \xi_{\chi'_f,s_0-\frac{1}{2}},\BC)$
    is defined over $\BQ(\Pi_f,\eta_f, \chi'_f)$, for all algebraic  automorphic characters $\chi': \Rk^\times\backslash \BA^\times\rightarrow \BC^\times$ and all $s_0\in \frac{1}{2}+\BZ$.
\end{prop}
\begin{proof}
The uniqueness is proved as in the proof of Theorem \ref{0lemma: Z circ f rational}. For the proof of  the existence, we write
$\Pi_f=\otimes'_{v\nmid \infty} \Pi_v$, $\psi_f=\otimes_{v\nmid \infty} \psi_v$, and $\eta_f=\otimes_{v\nmid \infty} \eta_v$.
   Using the uniqueness of Shalika functionals (Lemma \ref{uniquesh}), we write $\lambda_f=\otimes_{v\nmid \infty} \lambda_v$, where $\lambda_v\in \Hom_{S(\Rk_v)}(\Pi_v\otimes (\eta_v\otimes\psi_v)^{-1}, \BC)$. Using Lemma \ref{sphshalika}, we assume that
$\lambda_v(u_v^\circ)=1$ for almost all $v\nmid \infty$ such that $\Pi_v$ is unramified, 
   where $u_v^\circ\in \Pi_v$ is the spherical vector that is implicitly fixed for the definition of the restricted tensor product $\otimes'_{v\nmid \infty} \Pi_v$.

    Applying Theorem \ref{0lemma: Z circ f rational} to all $\Rk_v$ with $v\nmid \infty$, we obtain a $\BQ(\Pi_f, \eta_f)$-rational structure on  the representation $\Pi_f=\otimes'_{v\nmid \infty} \Pi_v$.
Then Theorem \ref{0lemma: Z circ f rational} and Part (6) of Proposition \ref{prop: properties of F-J int} imply that this rational structure has the desired property.
\end{proof}

The following elementary result is important to our proof. For completeness, we include a proof below. 

\begin{lemma}\label{uniqr}
Let $E$ be a subfield of $\BC$ and let $R$ be an associative $E$-algebra with identity element. Let $F_0$ be an irreducible $R\otimes_E \BC$-module that is finite dimensional as a $\BC$-vector space. Then up to scalar multiplication by a number in $\BC^\times$, there exists at most one $E$-rational structure of $F_0$ which is $R$-stable.
\end{lemma}

\begin{proof}
 Write $I$ for the annihilator of $F_0$ in $R$. Suppose that $F_E$ is a $R$-stable $E$-rational structure of $F_0$. Then  $I$ is also  the annihilator of $F_E$ in $R$ and  we have an injective homomorphism
    \begin{equation}\label{homat1}
  R/I\hookrightarrow \End(F_E).
  \end{equation}
  Tensoring wiht $\BC$ over $E$, we get an injective homomorphism
    \begin{equation}\label{homat2}
   (R/I)\otimes_E\BC \hookrightarrow \End(F_0).
  \end{equation}
  Note that $F_0$ is a faithful irreducible $(R/I)\otimes_E\BC$-module. Then it follows from Schur's Lemma (\cite[Section 3.5]{JacobsonBasicAlgebraBook1}) and the Jacobson Density Theorem for completely reducible modules (\cite[Section 4.3]{JacobsonBasicAlgebraBook2}) that the homomorphism \eqref{homat2} is an isomorphism. Thus \eqref{homat1} is also an isomorphism, and $R/I$ is isomorphic to a matrix algebra over $E$. Hence $R/I$ has only one irreducible module up to isomorphism.

  Suppose that we have another $R$-stable $E$-rational structure $F_E'$ of $F_0$. Then the previous argument shows that both $F_E$ and $F_E'$ are irreducible $R/I$-modules. Hence there is an isomorphism $\varphi: F_E\rightarrow F_E'$ of $R/I$-modules. This uniquely extends to a $(R/I)\otimes \BC$-module isomorphism $\varphi_\BC: F_0\rightarrow F_0$. By Schur's Lemma
  $\varphi_\BC$ is a scalar multiplication in $\BC^\times$. This proves the lemma.
\end{proof}

Using Proposition \ref{rpif}, we view $\Pi_f$ as a representation defined over $\BQ(\Pi_f,\eta_f)$.
Similar to the argument of Section \ref{subsection: local period relations}, we define a Borel sualgebra $\mathfrak b_\Rk=\ft_\Rk\ltimes \fu_\Rk$ of $\fg\fl_{2n}(\Rk)$. Then the space
 $(F^\vee)^{\mathfrak u_\Rk}$
     is one-dimensional. Fix a nonzero element $v_F^\vee$ in this space.
     \begin{lemma}\label{raff}
     There exists a unique $\fg\fl_{2n}(\Rk)$-stable $\BQ(F)$-rational structure on $F^\vee$ such that $v_F^\vee$ is defined over $\BQ(F)$. Moreover, this rational structure is $\GL_{2n}(\Rk)$-stable.
     \end{lemma}
    \begin{proof}This may be known to experts. We sketch a proof for the convenience of the reader. The uniqueness assertion is implied by Lemma \ref{uniqr}. For the existence assertion,
consider the Weil restriction $\mathrm{Res}_{\Rk/\BQ} G$, which is an algebraic group over $\BQ$. The Borel sualgebra $\mathfrak b_\Rk$ corresponds a Borel subgroup of $\mathrm{Res}_{\Rk/\BQ} G$,
which is denoted by $B_\Rk$. Then $v_F^\vee$ determines an algebraic character $\chi_{F^\vee}$ of $B_\Rk(\BC)$, and induces an isomorphism
    \[
          F\xrightarrow{\sim} \Ind_{B_\Rk(\BC)}^{\GL_{2n}(\Rk\otimes_\BQ \BC)} \chi_{F^\vee}^{-1}, \quad u\mapsto (g\mapsto \langle v_F^\vee, g.u\rangle),
       \]
       where $ \Ind_{B_\Rk(\BC)}^{\GL_{2n}(\Rk\otimes_\BQ \BC)}\chi_{F^\vee}$ denotes the algebraically induced representation. Note that the algebraic character $\chi_{F^\vee}$ is defined over $\BQ(F)$.  Hence  $\Ind_{B_\Rk(\BC)}^{\GL_{2n}(\Rk\otimes_\BQ \BC)} \chi_{F^\vee}^{-1}$ has a natural $\BQ(F)$-rational structure. This induces a $\BQ(F)$-rational structure on $F$, which further induces a $\BQ(F)$-rational structure on $F^\vee$ with the desired properties.

\end{proof}

Using the rational structure of Lemma \ref{raff}, we view $F^\vee$  as a representation of $\GL_{2n}(\Rk)$ defined over $\BQ(F)$.  This yields a $\BQ(F)$-rational structure on $\RH^{d_\infty}_c(\CX_G, F^\vee)$ and $\RH^{d_\infty}_c(\CX_H, F^\vee)$ (see \cite[pages 122-123]{ClozelAutomorphicFormsMotive}).
     By using \cite[Theorem 3.19]{ClozelAutomorphicFormsMotive} and the Strong Multiplicity One Theorem  (\cite[Page 209]{PS79} and \cite[Corollary 4.10]{JacquetShalikaEulerProduct1}), we know that the cohomology group $\RH^{d_\infty}_\mathrm{ct}(\BR_+^\times\quo G_\infty^0; \Pi\otimes F^\vee)$ has a unique $\BQ(\Pi)$-rational structure such that the injection map $\iota_\Pi$, as in \eqref{eq: injection into cpt supp coh} is defined over $\BQ(\Pi)$.

   \begin{lemma}\label{ratcoh}
   On the vector space $\RH_\mathrm{ct}^{d_\infty}(\BR^\times_+\quo G_\infty^0; \Pi_\infty\otimes F^\vee)$, 
       there exists a unique $\BQ(\Pi, \eta)$-rational structure such that the natural isomorphism
       \begin{equation}
       \label{raco}\RH_\mathrm{ct}^{d_\infty}(\BR^\times_+\quo G_\infty^0; \Pi_\infty\otimes F^\vee)\otimes \Pi_{f} \xrightarrow{\sim} \RH_\mathrm{ct}^{d_\infty}(\BR^\times_+\quo G_\infty^0; \Pi\otimes F^\vee)
       \end{equation}
       is defined over $\BQ(\Pi, \eta)$.
   \end{lemma}
   \begin{proof}
      This is clear by identifying the isomorphism \eqref{raco} with the following isomorphism:
        \begin{eqnarray*}
        \RH_\mathrm{ct}^{d_\infty}(\BR^\times_+\quo G_\infty^0; \Pi_\infty\otimes F^\vee)
         &\xrightarrow{\sim} &\Hom_{G(\BA_f)}(\Pi_f, \RH_\mathrm{ct}^{d_\infty}(\BR^\times_+\quo G_\infty^0; \Pi\otimes F^\vee))\\
         &&=(\Pi_f^\vee\otimes \RH_\mathrm{ct}^{d_\infty}(\BR^\times_+\quo G_\infty^0; \Pi\otimes F^\vee))_{G(\BA_f)}.
         \end{eqnarray*}
   \end{proof}

We remark that the rational structure of Lemma \ref{ratcoh} is determined by $v_F^\vee\in (F^\vee)^{\mathfrak u_\Rk}$ and the Shalika functional $\lambda_f$. Using this rational structure, we view $\RH_\mathrm{ct}^{d_\infty}(\BR^\times_+\quo G_\infty^0; \Pi_\infty\otimes F^\vee)$ as a vector space defined over $\BQ(\Pi, \eta)$.

The modular symbol at infinity map $\CP_{\chi_\infty, j}$, as defined in \eqref{pinfj}, depends on the nonzero functional
$\lambda_{F,j}\in\Hom_{H_\infty}(F^\vee\otimes\xi_{\infty,j}^\vee,\BC)$.
Note that the algebraic character $\xi_{\infty,j}$ is defined over $\BQ(F)$, in other words,
$\BQ(\xi_{\infty,j})\subset \BQ(F)$.
Thus as a one-dimensional vector space, $\xi_{\infty, j}$ has a natural $\BQ(F)$-rational structure.

\begin{lemma}\label{lj}
There exists a unique linear functional $\lambda_{F,j}\in\Hom_{H_\infty}(F^\vee\otimes\xi_{\infty,j}^\vee,\BC)$ such that $\lambda_{F,j}(v_F^\vee\otimes 1)=1$. Moreover, this linear  functional is defined over $\BQ(F)$.
\end{lemma}
\begin{proof}
This easily follows from the fact that $\fg\fl_{2n}(\Rk)=\fh_\Rk+\fb_\Rk$, where $\fh_\Rk=\fg\fl_n(\Rk)\times \fg\fl_n(\Rk)$.
\end{proof}

Now we suppose that $\lambda_{F,j}$ is  as in Lemma \ref{lj}. Write 
\begin{equation}\label{epsilonpsi}
\epsilon_{\psi_\infty}:= \prod_{v\textrm{ is a real place of $\Rk$}} \epsilon_{\Pi_v, \psi_v}\in \{\pm 1\},
\end{equation}
where $\epsilon_{\Pi_v, \psi_v}$ is defined as in \eqref{epsilonk}. 

   \begin{prop}\label{prop: period relation at infinity} 
      The linear functional
       \begin{equation}\label{p49}
          \eps_{\psi_\infty}^j\cdot \mathrm{i}^{-jn\cdot [\Rk:\BQ]} \cdot \CP_{\chi_\infty, j}\ \colon\
          \RH_\mathrm{ct}^{d_\infty}(\BR^\times_+ \quo G_\infty^0; \Pi_\infty\otimes F^\vee) \otimes \RH_\mathrm{ct}^0(\BR^\times_+\quo H_\infty^0; \xi_{\chi_\infty,j}\otimes\xi_{\infty,j}^\vee)\rightarrow \BC
       \end{equation}
     only depends on the quadratic character $\xi_{\chi_\infty,j}\otimes\xi_{\infty,j}^\vee$ of $H_\infty$ and is nonzero.
   \end{prop}
   \begin{proof}

Write
$F^\vee=\otimes_{v\mid \infty} F_v^\vee$ and and $v_F^\vee=\otimes_{v\mid \infty} v_{F_v}^\vee$,
where $F_v:=\otimes_{\iota\in \CE_{\Rk_v}\subset \CE_\Rk} F_\iota$ and $v_{F_v}^\vee\in F_v^\vee$.
Then
$\lambda_{F,j}=\otimes_{v\mid \infty} \lambda_{F_v,j}$,
where $   \lambda_{F_v,j}\in\Hom_{H_\BC}(F_v^\vee\otimes \xi_{\Rk_v,j}^\vee,\BC)$ is the element such that $\lambda_{F_v,j}(v_{F_v}^\vee\otimes 1)=1$, and
$\xi_{\Rk_v,j}:=\otimes_{\iota\in \CE_{\Rk_v}\subset \CE_\Rk}({\det}^{j}\otimes {\det}^{-j-w_\iota})$ as before.
       Using $\lambda_{F_v,j}$, we have the archimedean modular symbol
      $\CP_{\Rk_v,j}:  \RH_\mathrm{ct}^{d_{\Rk_v}}(\BR_+^\times\quo G_v^0; \Pi_v\otimes F_v^\vee)\rightarrow  \RH^{d_{\Rk_v}}_\mathrm{ct}(\BR_+^\times\quo H_v^0;  \BC)$
       as in \eqref{cpkj}. By Theorem \ref{thm: archi period relations}, the linear map 
       $\eps_{\Pi_v, \psi_v}^j\cdot \mathrm{i}^{-jn\cdot [\Rk_v:\BR]} \cdot \CP_{\Rk_v,j}$ 
      is independent of the critical place $\frac{1}{2}+j$. Here $\eps_{\Pi_v, \psi_v}=1$ if $v$ is a complex place. 

Similar to \eqref{cpkj}, we have a liner map  \begin{eqnarray}
 \label{cpkj2} \CP_{\infty,j}\ \colon\ \RH^{d_\infty}_\mathrm{ct}(\BR^\times_+\backslash G_\infty^0; \Pi_\infty\otimes F^\vee)
 \nonumber &\rightarrow & \RH^{d_{\infty}}_\mathrm{ct}(\BR^\times_+\backslash H_\infty^0; \Pi_\infty\otimes F^\vee)\\
  &&=\RH^{d_{\infty}}_\mathrm{ct}(\BR^\times_+\backslash H_\infty^0; (\Pi_\infty\otimes \xi_{\infty, j})\otimes (F^\vee\otimes\xi_{\infty, j}^\vee ))\\
   \nonumber  &\rightarrow & \RH^{d_{\infty}}_\mathrm{ct}(\BR^\times_+\backslash H_\infty^0;  \BC),
\end{eqnarray}
where the first arrow is the cohomology restriction map, and the last arrow is the map induced by the linear functional
$(\Pi_\infty\otimes \xi_{\infty, j})\otimes (F^\vee\otimes\xi_{\infty, j}^\vee )\rightarrow  \BC$
that equals the tensor product of the functionals as in \eqref{zcirc0} over all  archimedean places  of $\Rk$. 

       We have Lie group decompositions
   \begin{equation}\label{eq: subsection: proof of global period relation 01}
       \BR_+^\times\quo G_\infty = \left(\prod_{v\mid\infty} \BR_+^\times\quo G_v\right) \times A_G' \quad\textrm{and}\quad \BR_+^\times\quo H_\infty =\left(\prod_{v\mid\infty}  \BR_+^\times\quo H_v \right)\times A_G',
   \end{equation}
   where $A_G':=\BR^\times_+\backslash (\BR^\times_+)^r$ and $r$ is the number of archimedean places of $\Rk$.
 By the K\"{u}nneth formula, we have that
  $\RH_\mathrm{ct}^{d_\infty}(\BR_+^\times\quo G_\infty^0; \Pi_\infty\otimes F^\vee) =\RH'\oplus \RH''$,
where
                         \[
                          \RH':= \left(\otimes_{v\mid \infty} \RH_\mathrm{ct}^{d_{\Rk_v}}(\BR_+^\times\quo G_v^0; \Pi_v\otimes F_v^\vee)\right)\otimes \RH_\mathrm{ct}^{r-1}(A_G'; \BC)
                         \]
                         and
               \[
               \RH'':= \bigoplus_{l\neq r-1, \textrm{ or $a_v\neq d_{\Rk_v}$ for some $v\mid \infty$}}\left(\otimes_{v\mid \infty} \RH_\mathrm{ct}^{a_v}(\BR_+^\times\quo G_v^0; \Pi_v\otimes F_v^\vee)\right)\otimes\RH_\mathrm{ct}^{l}(A_G'; \BC).
               \]
           It is easy to see that $\CP_{\infty, j}$ vanishes on $\RH''$, and $\CP_{\infty, j}|_{\RH'}$ equals the composition of
               \begin{eqnarray*}
               \RH'&=&\left(\otimes_{v\mid \infty}  \RH_\mathrm{ct}^{d_{\Rk_v}}(\BR_+^\times\quo G_v^0; \Pi_v\otimes F_v^\vee)\right)\otimes \RH_\mathrm{ct}^{r-1}(A_G'; \BC)\\
               &\xrightarrow{\left(\otimes_{v\mid \infty}\CP_{\Rk_v,j}\right)\otimes 1} &\left(\otimes_{v\mid \infty} \RH_\mathrm{ct}^{d_{\Rk_v}}(\BR_+^\times\quo H_v^0; \BC)\right)\otimes  \RH_\mathrm{ct}^{r-1}(A_G'; \BC)\\
               &\xrightarrow{\sim}&\BC,
                              \end{eqnarray*}
               where the last arrow is a linear isomorphism independent of $j$. Thus the linear map $\CP_{\infty,j}$ is independent of $j$. 
               Replacing $\Pi_\infty$ by the rerepsentation $\Pi_\infty\otimes((\chi_\infty\cdot \sgn^j)\circ \det)$ (which is isomrohic to $\Pi_\infty$), this result easily implies  the first assertion of the proposition. The second assertion clearly follows from Theorem \ref{thm: Non-vanishing hypothesis}.
   \end{proof}

      For every   quadratic character $\epsilon_{\infty}$  of $\BA_\infty^\times$, define a character 
      $\tilde \epsilon_\infty:=(\epsilon_\infty\circ \det)\otimes (\epsilon_\infty \circ \det)$ of $H(\BA_\infty)$.
       For example, if $\epsilon_\infty= \sgn^j\cdot \chi|_\infty$, then $\tilde \epsilon_\infty= \xi_{\chi,j}\otimes\xi_{\infty,j}^\vee$. 
By Proposition \ref{prop: period relation at infinity}, there exists a $\BQ(\Pi,\eta)$-rational cohomological class $[\omega]_{\epsilon_\infty}\in\RH_\mathrm{ct}^{d_\infty}(\BR^\times_+\quo G_\infty^0; \Pi_\infty\otimes F^\vee)$ such that the functional \eqref{p49} has nonzero value at $[\omega]_{\epsilon_\infty}\otimes 1$ whenever $\sgn^j\cdot \chi_\infty =\epsilon_\infty$. 
We fix $[\omega]_{\epsilon_\infty}$ as above and define the {\sl Shalika period} of $\Pi$ to be
       \begin{equation}\label{eq: def of Omega Pi}
          \omega^{\epsilon_\infty}(\Pi) := \omega_{\psi}^{\epsilon_\infty}(\Pi):=\left(\eps_{\psi_\infty}^j\cdot \mathrm{i}^{-jn\cdot [\Rk:\BQ]} \cdot \CP_{\chi_\infty, j}([\omega]_{\epsilon_\infty}\otimes1)\right)^{-1},
       \end{equation}
which is independent of $j$, by Proposition \ref{prop: period relation at infinity}.

\begin{rk}
We expect that the functional $\eps_{\psi_\infty}^j\cdot \mathrm{i}^{-jn\cdot [\Rk:\BQ]} \cdot \CP_{\chi_\infty, j}$ as given in \eqref{p49} 
maps the $\BQ(\Pi,\eta)$-rational structure of $\RH_\mathrm{ct}^{d_\infty}(\BR^\times_+\quo G_\infty^0; \Pi_\infty\otimes F^\vee)$ to a $\BQ(\Pi,\eta)$-rational structure of $\BC$ (here $\RH_\mathrm{ct}^0(\BR^\times_+\quo H_\infty^0; \xi_{\chi_\infty,j}\otimes\xi_{\infty,j}^\vee)$ is obviously identified with $\BC$).
This holds at least when the number field $\Rk$ is totally real, or $L(\frac{1}{2}+j', \Pi\otimes \chi')\neq 0$ for some critical place $\frac{1}{2}+j'$ and some finite order character $\chi'$ of $\Rk^\times \quo \BA^\times$ such that $\sgn^{j'} \cdot \chi'|_{\BA_\infty^\times}=\sgn^j \cdot \chi|_{\BA_\infty^\times}$ and $\BQ(\chi')\subset \BQ(\Pi, \eta)$. When our expectation holds, up to a scalar multiplication in $\BQ(\Pi,\eta)^\times$, the Shalika period $\omega^{\epsilon_\infty}(\Pi)$  is independent of the choice of the rational cohomology class $[\omega]_{\epsilon_\infty}$.  Januszewski informed us that in many other cases, our expectation follows from \cite[Theorem 6.4]{Jan18}.

\end{rk}

\subsection{Proof of Theorem \ref{thm: global period relation}}

We need the following lemma, whose proof is simple and is left to the reader.
   \begin{lemma}\label{lemma: compare rational functional}
       Suppose that $E$ is a subfield of $\BC$ and $V$ is a complex vector space with an $E$-rational structure. Let $\lambda_1$, $\lambda_2$ be two linear functionals on $V$ defined over $E$ and let $c\in \BC$. If $\lambda_1\ne 0$ and $c\cdot \lambda_1 = \lambda_2$, then $c\in E.$
   \end{lemma}

Finally, we are ready to prove Theorem \ref{thm: global period relation}. Suppose that $\epsilon_\infty= \sgn^j \cdot \chi_\infty$.
With the fixed $[\omega]_{\epsilon_\infty}$, we consider the following commutative diagram
$$\CD
           \Pi_f\otimes\xi_{\chi_f,j} @>\CP_{\chi_\infty, j}([\omega]_{\epsilon_\infty}\otimes 1)\,\cdot \,Z^\circ(\,\cdot\,, s, \chi_f)>> \BC \\
           @V ([\omega]_{\epsilon_\infty}\otimes 1)\otimes(\,\cdot\,) VV @V L(\frac{1}{2}+j, \Pi\otimes\chi) VV  \\
            \RH_\mathrm{ct}^{d_\infty}(\BR^\times_+\quo G_\infty^0; \Pi\otimes F^\vee)\otimes \RH_\mathrm{ct}^0(\BR^\times_+\quo H_\infty^0, \xi_{\chi,j}\otimes\xi_{\infty,j}^\vee) @>\CP_{\chi,j}>> \BC
         \endCD$$
as in Proposition \ref{prop: main diagram}.
Note that  the functional
$\CP_{\chi,j} \circ (([\omega]_{\epsilon_\infty} \otimes 1)\otimes(\cdot) )\colon \Pi_f\otimes\xi_{\chi_f,j}\rightarrow\BC$
is defined over $\BQ(\Pi,\eta,\chi)$. By Proposition \ref{rpif}, the nonzero functional
$\CG(\chi_f)^n\cdot Z^\circ(\cdot, s, \chi_f)$ on $\Pi_f\otimes\xi_{\chi_f,j}$ 
is defined over $\BQ(\Pi,\eta,\chi)$. Thus Lemma \ref{lemma: compare rational functional} implies that
\[
   L(\frac{1}{2}+j, \Pi\otimes\chi)\cdot \CP_{\chi_\infty, j}([\omega]_{\epsilon_\infty} \otimes 1)\cdot\CG(\chi_f)^{-n}\in \BQ(\Pi,\eta,\chi),
\]
or equivalently,
\[
    \frac{L(\frac{1}{2}+j,\Pi\otimes \chi)}{\mathrm{i}^{ jn\cdot  [\Rk\, :\, \BQ]}\cdot \CG(\chi_f)^n \cdot \omega^{\epsilon_\infty}(\Pi)}\in \BQ(\Pi,\eta,\chi).
\]
This proves Theorem \ref{thm: global period relation}.


\section{Proofs of Theorems \ref{thm: Non-vanishing hypothesis} and  \ref{thm: archi period relations}}\label{section: Proof of Critical place and Non-vanishing Modular Symbol}

Assume that $\BK$ is archimedean in this section.
As before,  $\fg_\BK$, $\fh_\BK$, $\fk_\BK$ and $\fc_\BK$ are the Lie algebras of $G_\BK$, $H_\BK$, $K_\BK$ and $C_\BK$, respectively.
\subsection{Some Borel subalgebras}
As a slight modification of the embedding $\gamma_\BK$ defined in  \eqref{eq: embedding no i}, we define an embedding
           \begin{equation}\label{eq: embedding i-1}
               \gamma': \BC^{2n}\hookrightarrow \fgl_{2n}(\BC)=\fg_\BC
           \end{equation}
     which sends $(a_1,a_2,\cdots,a_{2n})$ to the matrix
    \begin{equation}
    \left[
  \begin{array}{cccccccc}
    \frac{a_1+a_{2n}}{2} & 0 & \cdots & 0 & 0 &\cdots & 0 & \frac{a_1-a_{2n}}{-2\mathrm i}  \\
     0&\frac{a_2+a_{2n-1}}{2}  & \cdots & 0 & 0 &\cdots & \frac{a_2-a_{2n-1}}{-2\mathrm i}  & 0 \\
     \cdots & \cdots & \cdots & \cdots & \cdots &\cdots & \cdots &\cdots  \\
   0 & 0 & \cdots & \frac{a_n+a_{n+1}}{2} & \frac{a_n-a_{n+1}}{-2\mathrm i} &\cdots & 0 & 0  \\
   0 & 0 & \cdots &\frac{a_n-a_{n+1}}{2\mathrm i}   & \frac{a_{n+1}+a_n}{2} &\cdots & 0 & 0  \\
  \cdots & \cdots & \cdots & \cdots & \cdots &\cdots & \cdots &\cdots  \\
   0& \frac{a_2-a_{2n-1}}{2\mathrm i} & \cdots & 0 & 0 &\cdots & \frac{a_{2n-1}+a_2}{2}  & 0 \\
   \frac{a_{1}-a_{2n}}{2\mathrm i} & 0 & \cdots & 0 & 0 &\cdots &  & \frac{a_{2n}+a_1}{2}  \\
  \end{array}
   \right].
   \end{equation}
   This embedding identifies $\BC^{2n}$ as a Cartan subalgebra $\ft'$ of $\fg_\BC$. With the so obtained Cartan subalgebra $\ft'$, the root system of $\fg_\BC$ is
$\Phi(\fg_\BC, \ft') = \set{\pm(\eps_i-\eps_j)}{1\leq i< j \leq 2n}.$
   Here $\eps_1, \eps_2, \cdots, \eps_{2n}$ is the standard basis of $\BC^{2n}$.
The positive root system
$\Phi^+(\fg_\BC, \ft') = \set{\eps_i-\eps_j}{1\leq i< j \leq 2n}$
defines a Borel subalgebra $\fb'$ of $\fg_\BC$. Write $\fu'$ for its nilpotent radical so that $\fb'=\ft'\ltimes \fu'$.


For each $a\in \BC^\times$, put
$J_a= \mtrtwo{a \cdot 1_n}{0}{0}{1_n}\in  H_\BC$.
Then $\gamma'$ equals the composition of
\[
  \BC^{2n}\xrightarrow{\gamma_\BC} \fg_\BC\xrightarrow{\Ad_{J_\mathrm i}} \fg_\BC\qquad(\textrm{$\Ad$ stands for the adjoint action}).
\]
Hence $\ft'=\Ad_{J_\mathrm i}(\ft_\BC)$, $\fb'=\Ad_{J_\mathrm i}(\fb_\BC)$ and $\fu'=\Ad_{J_\mathrm i}(\fu_\BC)$.
Now we define
        $\ft'_\BK :=\prod_{\iota\in \CE_\BK} \ft'$, $\fb'_\BK :=\prod_{\iota\in \CE_\BK} \fb'$ and 
        $\fu'_\BK :=\prod_{\iota\in \CE_\BK} \fu'. $
        Put
        $$J_{a,\BK} := \prod_{\iota\in \CE_\BK} J_{a}\in H(\BK\otimes_\BR \BC)=\prod_{\iota\in \CE_\BK} H_\BC,\qquad (a\in \BC^\times). $$
 Then we have that
$\ft'_\BK=\Ad_{J_{\mathrm i,\BK}}(\ft_\BK^\BC)$, $\fb'_\BK=\Ad_{J_{\mathrm i, \BK}}(\fb_\BK^\BC)$ and
$\fu'_\BK=\Ad_{J_{\mathrm i, \BK}}(\fu_\BK^\BC)$,
and that
$\Ad_{J_{-1, \BK}}(\ft'_\BK)=\ft_\BK'$ and the Borel subalgebra $\Ad_{J_{-1, \BK}}(\fb_\BK')$ is opposite to $\fb_\BK'$. 
We remark that $\ft_\BK'$, $\fb_\BK'$, $\fu_\BK'$, $\Ad_{J_{-1, \BK}}(\fb_\BK')$, $\Ad_{J_{-1, \BK}}(\fu_\BK')$ and $\fh_\BK$ are all $\theta_\BK$-stable, where  $\theta_\BK : \fg_\BK^\BC\rightarrow \fg_\BK^\BC$ is the complexified differential of the Cartan involution corresponding to the maximal compact subgroup $K_\BK$ of $G_\BK$.

For each $\iota\in \CE_\BK$, define a linear embedding
\begin{equation}\label{eqpr: embedding i-2}
  \gamma'_\iota: \BC^{n\cdot [\BK: \BR]}\hookrightarrow \fg_\BK^\BC
\end{equation}
by
$\gamma'_\iota(a_1, a_2, \cdots, a_n):=\gamma'(a_1, a_2, \cdots, a_n, -a_n, \cdots, -a_2, -a_1)$ 
if $\BK\cong\BR$, and
\[
\gamma'_\iota(a_1, a_2, \cdots, a_{2n}):=(\iota\mapsto \gamma'(a_1, a_2, \cdots, a_{2n}), \, \bar \iota\mapsto \gamma'(-a_{2n}, \cdots, -a_2, -a_1))
\]
if $\BK\cong\BC$.
Here all $a_i$'s a complex numbers, and $\fg_\BK^\BC=\prod_{\iota\in \CE_\BK} \fg_\BC$ is identified with the set of all maps from $\CE_\BK$ to $\fg_\BC$. Then $\gamma'_\iota$ identifies $\BC^{n\cdot [\BK: \BR]}$ with
 $\ft'_\BK\cap\fk_\BK^\BC$, which is a Cartan subalgebra of $\fk_\BK^\BC$.
Using this identification, we write $\tau_{\BK,0}$ for the irreducible representation of $K_\BK$ with extremal weight
$(2n, 2n-2,\cdots, 2)$  or $(2n-1, 2n-3,\cdots, 3-2n, 1-2n)$,
when $\BK\cong\BR$ or $\BC$, respectively. We remark that in the complex case, the representation $\tau_{\BK,0}$ is independent of the choice of $\iota\in\CE_\BK$.


Write
$$r_\BK:=\dim \fu'_\BK/(\fu'_\BK\cap \fk_\BK^\BC)=\begin{cases}
                                                      n^2, & \hbox{if $\BK\cong\BR$;} \\
                                                      2n^2-n, & \hbox{if $\BK\cong\BC$.}
                                                    \end{cases}
$$
Set $\fs_\BK:=\fg_\BK/(\BR\oplus\fk_\BK)$.


\begin{lemma}\label{tauk0}
     The representation  $\tau_{\BK,0}$ of $K_\BK$ occurs in $\wedge^{d_\BK} \fs_\BK^\BC$ with multiplicity
\[
  \dim \wedge^{d_\BK-r_\BK} \, \ft'_\BK/(\ft'_\BK\cap \fk_\BK^\BC\oplus \BC)=
                                                    \begin{cases}
                                                      1, & \hbox{if $\BK\cong\BR$;} \\
                                                    {2n-1 \choose n-1}, & \hbox{if $\BK\cong\BC$.}
                                                    \end{cases}
                                                  \]
\end{lemma}
     \begin{proof}
     This is easy to check and is left to the reader.
     \end{proof}

\subsection{Archimedean modular symbols}
We follow the notation and assumptions as in Section \ref{subsection: Critical Place and The Non-Vanishing Hypothesis at Infinity}. More precisely, the irreducible Casselman-Wallach representation $\Pi_\BK$ of $G_\BK$ is essentially tempered, cohomological, and of symplectic type. The coefficient system $F_\BK=\otimes_{\iota\in \CE_\BK} F_\iota$ of  $\Pi_\BK$ is assumed to be balanced. The highest weight $\{\nu_\iota=(\nu^\iota_1 \geq  \nu^\iota_2\geq \cdots \geq \nu^\iota_{2n})\}_{\iota\in \CE_\BK}$ satisfies \eqref{eq2: global purity lemma}   and \eqref{eq: def of w iota}.
As before, we fix a nonzero element $v_{F_\BK}^\vee$ of the one-dimensional space $(F_\BK^\vee)^{\mathfrak u_\BK}$.  Let $\frac{1}{2}+j\in \frac{1}{2}+\BZ$ be a critical place of $\Pi_\BK$.
Recall that we have  fixed a nonzero Shalika functional
$\lambda_{\BK}\in \Hom_{S_\BK}(\Pi_\BK, \eta_\BK\otimes\psi_\BK)$,
where $\psi_\BK$ is a non-trivial unitary character of $\BK$, and $\eta_\BK$ is  the character
     \begin{equation}\label{eq: choice of eta}
             \eta_\BK=\prod_{\iota \in \CE_\BK} (\iota|_{\BK^\times})^{w_\iota}: \BK^\times \rightarrow \BC^\times.
     \end{equation}
 Depending on $\lambda_\BK$ and a fixed Haar measure on $\GL_n(\BK)$, we have the normalized Friedberg-Jacquet integral
$Z^\circ(\,\cdot\,,\frac{1}{2}+j, \sgn_\BK^j)\in \Hom_{H_\BK}(\Pi_\BK\otimes\xi_{\BK,j}, \BC)$
 as in \eqref{nz-lf2}, where
$\xi_{\BK,j}:=\otimes_{\iota\in \CE_\BK} ({\det}^j\otimes {\det}^{-j-w_\iota})$ 
is an algebraic character of $H(\BK\otimes_\BR \BC)=\prod_{\iota\in \CE_\BK} H_\BC$.
Using Proposition \ref{prop: critical place}, we  fix a nonzero element
     \begin{equation}\label{eq: linear function coeffi system} \lambda_{F_\BK, j} \in  \Hom_{H_\BK}(F_\BK^\vee\otimes\xi_{\BK, j}^\vee, \BC).\end{equation}
As in Lemma \ref{lj}, we assume without loss of generality that
$\lambda_{F_\BK,j}(v_{F_\BK}^\vee\otimes 1)=1$.

  By \cite[Theorem 6.1]{HochshildMostowCohomologyofLieGroups} and \cite[Proposition 9.4.3]{WallachRealReductiveGroups1}, we have identifications
    \begin{eqnarray*}
    \RH^{d_\BK}_\ct(\BR^\times_+\backslash G_\BK^0; \Pi_\BK\otimes F_\BK^\vee) 
    =\RH^{d_\BK}(\fg_\BK^\BC/\BC, K_\BK^0; \Pi_\BK\otimes F_\BK^\vee)
    = \Hom_{K_\BK^0}(\wedge^{d_\BK} \fs_\BK^\BC, \Pi_\BK\otimes F_\BK^\vee)
    \end{eqnarray*}
     and
         \begin{eqnarray*}
         \RH^{d_\BK}_\ct(\BR^\times_+\backslash H_\BK^0; \BC)
         =\RH^{d_\BK}(\fh_\BK^\BC/\BC, C_\BK^0; \BC) 
         = \Hom_{C_\BK^0}(\wedge^{d_\BK}\fq_\BK^\BC,  \BC),
             \end{eqnarray*}
    where $\fq_\BK:=\fh_\BK/(\fc_\BK\oplus \BR)$ as before.
     Moreover, the archimedean modular symbol map $\CP_{\BK,j}$, which is defined in \eqref{cpkj}, is identical to the map
      \begin{equation}\label{cpbkj2}
        \begin{array}{rcl}
          \CP_{\BK,j}: \Hom_{K_\BK^0}(\wedge^{d_\BK} \fs_\BK^\BC, \Pi_\BK\otimes F_\BK^\vee)&\rightarrow& \Hom_{C_\BK^0}(\wedge^{d_\BK}\fq_\BK^\BC,  \BC)\\
        \end{array}
      \end{equation}
     which sends an element $\phi$ to the composition of
        \begin{eqnarray*}
     \wedge^{d_\BK} \fq_{\BK}^\BC &\hookrightarrow &\wedge^{d_\BK} \fs_\BK^\BC\\
     &\xrightarrow{\phi} &\Pi_\BK\otimes F_\BK^\vee
     = (\Pi_\BK\otimes \xi_{\BK,j})\otimes (F_\BK^\vee\otimes \xi_{\BK,j}^\vee) \xrightarrow{Z^\circ(\cdot,\frac{1}{2}+j,\sgn_{\BK}^j)\otimes \lambda_{F_\BK,j}} \BC.
         \end{eqnarray*}

     Now a more explicit  description of the first space in \eqref{cpbkj2} is in order. Recall that $\tau_\BK\subset \Pi_\BK$ is the minimal $K_\BK$-type.
     The space $F_\BK^{\fu_\BK'}\subset F_\BK$ generates an irreducible representation of $K_\BK$, which we denote by $\delta_\BK$. It occurs in $F_\BK$ with multiplicity one. Thus we get that
     $\delta_\BK^\vee\subset F_\BK^\vee$. 
Recall from  \eqref{eqpr: embedding i-2} that when $\BK\cong\BR$, $\BC^n$ is identified with a Cartan subalgebra of $\fk_\BK^\BC$.
     \begin{lemma}\label{dt1}
     If $\BK\cong\BR$ with $\CE_\BK=\{\iota\}$, then
      \begin{equation}\label{dectau1}
        \delta_\BK=\delta_\BK^+ + \delta_\BK^-, \quad \tau_{\BK,0}=\tau_{\BK,0}^+\oplus \tau_{\BK,0}^-\quad \textrm{and} \quad  \tau_\BK=\tau_\BK^+\oplus \tau_\BK^-,
      \end{equation}
       where $\delta_\BK^+$, $\delta_\BK^-$, $\tau_{\BK,0}^+$, $\tau_{0,\BK}^-$, $\tau_\BK^+$ and $\tau_\BK^-$ are the irreducible representations of $K_\BK^0$ with extremal weights
\[
   \left\{
     \begin{array}{l}
       \mu_\BK^+:=(\nu_1^\iota-\nu_{2n}^\iota, \nu_{2}^\iota-\nu_{2n-1}^\iota,\cdots, \nu_{n-1}^\iota-\nu_{n+2}^\iota , \nu_{n}^\iota-\nu_{n+1}^\iota ), \medskip\\
       \mu_\BK^-:=(\nu_1^\iota-\nu_{2n}^\iota, \nu_{2}^\iota-\nu_{2n-1}^\iota,\cdots, \nu_{n-1}^\iota-\nu_{n+2}^\iota, -(\nu_{n}^\iota-\nu_{n+1}^\iota)), \medskip\\
    (2n, 2n-2,\cdots, 4, 2),\medskip\\
 (2n, 2n-2,\cdots, 4, -2),\medskip\\
  \mu_\BK^+ +(2n, 2n-2,\cdots, 4, 2),\  \textrm{ and}\medskip\\
\mu_\BK^- +(2n, 2n-2,\cdots, 4, -2),
           \end{array}
   \right.
\]
respectively. Moreover, the first sum in \eqref{dectau1} is a direct sum if and only if $\nu_n^\iota>\nu_{n+1}^\iota$.
     \end{lemma}
 \begin{proof}
     This is  well known and elementary to check. See \cite[Section 5.2]{SunCohomologicalDistinguishedRepn} for example.
     \end{proof}

When $\BK\cong\BR$, let
$\delta_\BK=\delta_\BK^+ + \delta_\BK^-$, $\tau_{\BK,0}=\tau_{\BK,0}^+\oplus \tau_{\BK,0}^-$ and 
$\tau_\BK=\tau_\BK^+\oplus \tau_\BK^-$
     be  as in \eqref{dectau1}. By taking the contragredients of the first equality, we have that
     $\delta_\BK^\vee=\delta_\BK^{+,\vee} + \delta_\BK^{-,\vee}\subset F_\BK^\vee$.

 \begin{lemma}\label{dt2}
     If $\BK\cong\BC$ with $\CE_\BK=\{\iota, \bar \iota\}$, then the representations
$\delta_\BK$ and $\tau_\BK$ have extremal weights
          \[
    \mu_\BK^\iota:= (\nu_1^\iota-\nu_{2n}^{\bar \iota} ,\nu_{2}^\iota-\nu_{2n-1}^{\bar \iota},\cdots, \nu_{2n-1}^\iota-\nu_{2}^{\bar \iota}, \nu_{2n}^\iota-\nu_{1}^{\bar \iota})
     \]
and
\[
    \mu_\BK^\iota+ (2n-1,2n-3,\cdots,3-2n, 1-2n)
     \]
   respectively.  Here $\BC^{2n}$ is identified with
 a Cartan subalgebra of $\fk_\BK^\BC$ by $\gamma'_\iota$.

     \end{lemma}

 \begin{proof}
     This is also well known and elementary to check. 
     \end{proof}

 \begin{lemma}\label{dt2222}
   If $\BK\cong\BR$, then
$\dim \Hom_{K_\BK^0}(\tau_{\BK,0}^+, \delta_\BK^{+,\vee}\otimes \tau_\BK^+)=\dim \Hom_{K_\BK^0}(\tau_{\BK,0}^-, \delta_\BK^{-,\vee}\otimes \tau_\BK^-)=1$.
If $\BK\cong\BC$, then
$\dim \Hom_{K_\BK}(\tau_{\BK,0}, \delta_\BK^{\vee}\otimes \tau_\BK)=1$.
     \end{lemma}
\begin{proof}
Note that
$\Hom_{K_\BK}(\tau_{\BK,0}, \delta_\BK^{\vee}\otimes \tau_\BK)=\Hom_{K_\BK}(\tau_{\BK,0}\otimes \delta_\BK, \tau_\BK)$.
The lemma follows from Lemma \ref{dt2} in the complex case and from Lemma \ref{dt1} in the real case.
\end{proof}

Finally, we have the following description of the cohomology space.
      \begin{prop}\label{propdcoh}
       If  $\BK\cong\BR$, then the natural map
         \begin{eqnarray}
     \label{idh1}   \Hom_{K_\BK^0}(\wedge^{d_\BK} \fs_\BK^\BC, (\tau_\BK^+\otimes \delta_\BK^{+,\vee})\oplus (\tau_\BK^-\otimes \delta_\BK^{-, \vee}))
        &\rightarrow &\Hom_{K_\BK^0}(\wedge^{d_\BK} \fs_\BK^\BC, \Pi_\BK\otimes F_\BK^\vee)
          \end{eqnarray}
      is a linear isomorphism.
If $\BK\cong\BC$, then the natural map
         \begin{eqnarray}\label{idh2}
\Hom_{K_\BK^0}(\wedge^{d_\BK} \fs_\BK^\BC, \tau_\BK\otimes \delta_\BK^\vee)
       &\rightarrow &\Hom_{K_\BK^0}(\wedge^{d_\BK} \fs_\BK^\BC, \Pi_\BK\otimes F_\BK^\vee)
          \end{eqnarray}
      is a linear isomorphism.
      \end{prop}
     \begin{proof}
     Realize the representation $\Pi_\BK$ via the cohomological induction as in \cite[Section 3]{Su17} (the real case) and  \cite[Section 6.2]{Su17} (the complex case). By using the explicit calculation of the relative Lie algebra cohomologies of these  cohomological induced representations (see \cite[Theorem 9.6.6]{WallachRealReductiveGroups1} or \cite[Theorem 5.5]{VoganZuckermanUnitaryRepresentationNonzeroCohomology}), we know that
     \[
       \dim \Hom_{K_\BK^0}(\wedge^{d_\BK} \fs_\BK^\BC, \Pi_\BK\otimes F_\BK^\vee)
        =\begin{cases}
        2, & \hbox{if $\BK\cong\BR$;} \\
        \dim \wedge^{d_\BK-r_\BK} \, \ft'_\BK/(\ft'_\BK\cap \fk_\BK^\BC\oplus \BC), & \hbox{if $\BK\cong\BC$.}
        \end{cases}
     \]
 By Lemmas \ref{tauk0} and \ref{dt2222},  the dimensions of the domains of the maps \eqref{idh1} and \eqref{idh2} are respectively at least the above numbers. Hence the proposition follows by noting that the maps \eqref{idh1} and \eqref{idh2} are both injective.
     \end{proof}

\subsection{Proof of  Theorem \ref{thm: archi period relations} in the real case}\label{section: realProof of Critical place and Non-vanishing Modular Symbol}

 In this subsection we assume that $\BK\cong\BR$.
We have the following period relation for $\lambda_{F_\BK,j}$.

\begin{lemma}\label{lems0}
The composition functionals
\begin{equation}\label{ff1}
  \delta_\BK^{+, \vee}\xrightarrow{v\mapsto v\otimes 1} \delta_\BK^{+, \vee}\otimes \xi_{\BK,j}^\vee\subset F_\BK^\vee\otimes \xi_{\BK,j}^\vee\xrightarrow{\mathrm i^{-jn}\cdot \lambda_{F_\BK,j}}\BC
\end{equation}
and
\begin{equation}\label{ff2}
  \delta_\BK^{-, \vee}\xrightarrow{v\mapsto v\otimes 1} \delta_\BK^{-, \vee}\otimes \xi_{\BK,j}^\vee\subset F_\BK^\vee\otimes \xi_{\BK,j}^\vee\xrightarrow{(-1)^j\cdot\mathrm i^{-jn}\cdot \lambda_{F_\BK,j}}\BC
\end{equation}
 are independent of the critical place $\frac{1}{2}+j$ of $\Pi_\BK$.
\end{lemma}
\begin{proof}
Recall that $\Ad_{J_{-1}}(\ft'_\BK)=\ft_\BK'$ and the Borel subalgebra $\Ad_{J_{-1}}(\fb_\BK')$ is opposite to $\fb_\BK'$.
Note that the representation $\delta_\BK^+\subset F_\BK$ of $K_\BK^0$ is generated by $F_\BK^{\fu_\BK'}\subset F_\BK$. Hence  its contragredient representation $\delta_\BK^{+,\vee}\subset F_\BK^\vee$ is  generated by
$( F_\BK^\vee)^{\Ad_{J_{-1}}(\fu_\BK')}= ( F_\BK^\vee)^{\Ad_{J_{-\mathrm i}}(\fu_\BK^\BC)}=\BC\cdot ( J_{-\mathrm i}. v_{F_{\BK}}^\vee)$.
Consequently,
\begin{equation}\label{gdp1}
 (\delta_\BK^{+,\vee})^{(\Ad_{J_{-1}}(\fu_\BK'))\cap \fk_\BK^\BC}=\BC\cdot ( J_{-\mathrm i}. v_{F_{\BK}}^\vee).
\end{equation}
The linear functional \eqref{ff1} sends $J_{-\mathrm i}. v_{F_{\BK}}^\vee$ to
${\mathrm i}^{-jn}\cdot \lambda_{F_\BK,j}(J_{-\mathrm i}. v_{F_{\BK}}^\vee\otimes 1)$, which is equal to ${\mathrm i}^{-jn}\cdot \xi_{\BK,j}^{-1}(J_{-\mathrm i})\cdot \lambda_{F_\BK,j}( v_{F_{\BK}}^\vee\otimes 1)=1$.
This implies that \eqref{ff1} is independent of the critical place $\frac{1}{2}+j$ of $\Pi_\BK$, since
$(\Ad_{J_{-1}}(\fu_\BK'))\cap \fk_\BK^\BC+\fc_\BK^\BC=\fk_\BK^\BC $ (as \eqref{transvb}) 
and the linear functional \eqref{ff1} is $C_\BK^0$-invariant. (See the proof of \cite[Lemma A.6]{SunCohomologicalDistinguishedRepn}.)

Put $J^-:=\mtrtwo{1_n^-}{0}{0}{1_n}\in  H_\BC$,
where $1_n^-:= \mtrtwo{1_{n-1}}{0}{0}{-1}$. One may easily deduce from \eqref{gdp1} that
$(\delta_\BK^{-,\vee})^{((\Ad_{J^-}\circ\Ad_{J_{-1}})(\fu_\BK'))\cap \fk_\BK^\BC}=\BC\cdot ( (J^-\cdot J_{-\mathrm i}). v_{F_{\BK}}^\vee)$.
The linear functional \eqref{ff2} sends $(J^-\cdot J_{-\mathrm i}). v_{F_{\BK}}^\vee$ to
\[
(-1)^j\cdot {\mathrm i}^{-jn}\cdot \lambda_{F_\BK,j}((J^-\cdot J_{-\mathrm i}). v_{F_{\BK}}^\vee\otimes 1)=(-1)^j\cdot {\mathrm i}^{-jn}\cdot \xi_{\BK,j}^{-1}(J^-\cdot J_{-\mathrm i})\cdot \lambda_{F_\BK,j}( v_{F_{\BK}}^\vee\otimes 1)=1.
\]
As before, this implies that the linear functional \eqref{ff2} is independent of the critical place $\frac{1}{2}+j$.

\end{proof}

We have a decomposition
$\Pi_\BK=\Pi_\BK^+\oplus \Pi_\BK^-$,
where $\Pi_\BK^\pm$ is the irreducible $G_\BK^0$-subrepresentation of $\Pi_\BK$ containing $\tau_\BK^\pm$. By uniqueness of Shalika functionals, we have that
\[
  \dim \Hom_{S_\BK}(\Pi_\BK^+, \eta_\BK\otimes \psi_\BK)+\dim \Hom_{S_\BK}(\Pi_\BK^-, \eta_\BK\otimes \psi_\BK)=\dim \Hom_{S_\BK}(\Pi_\BK, \eta_\BK\otimes \psi_\BK)=1.
\]
Hence we obtain that either $\lambda_{\BK}|_{\Pi_\BK^+}\neq 0$  or $\lambda_{\BK}|_{\Pi_\BK^-}\neq  0$.
Define
\begin{equation}\label{epsilonk}
\epsilon_\BK:=\epsilon_{\Pi_\BK, \psi_\BK}:=
                  \begin{cases}
                    1, & \hbox{if  $\lambda_{\BK}|_{\Pi_\BK^+}\neq 0$;} \\
                   -1, & \hbox{if  $\lambda_{\BK}|_{\Pi_\BK^-}\neq 0$.} 
                  \end{cases}
\end{equation}
We remark that $\epsilon_\BK$ depends not only on the representation $\Pi_\BK$, but also on the additive character $\psi_\BK$.

\begin{lemma}\label{lems1}
For all $s\in\BC$, we have that 
\[
Z^\circ(\,\cdot \,,s, \sgn_\BK)|_{\Pi_\BK^+}=\epsilon_\BK Z^\circ(\,\cdot \,,s, \sgn_\BK^0)|_{\Pi_\BK^+}\ {\rm and}\ 
Z^\circ(\,\cdot \,,s, \sgn_\BK)|_{\Pi_\BK^-}=-\epsilon_\BK Z^\circ(\,\cdot \,,s, \sgn_\BK^0)|_{\Pi_\BK^-}.
\]
\end{lemma}
\begin{proof}
If $\epsilon_\BK=1$ and $v\in \Pi_\BK^+$, then
$\left \langle\lambda_\BK, \mtrtwo{g}{0}{0}{1_n}.v\right \rangle=0$ for all $g\in \GL_n(\BK)\setminus \GL_n(\BK)^0$.
Hence for all $i\in \BZ$, and all $s\in \BC$ with  sufficiently large real part, we have that
\begin{eqnarray*}
   Z^\circ(v,s, \sgn_\BK^i)&=&\frac{1}{L(s, \Pi_\BK\otimes \sgn_\BK^i)}\int_{\GL_n(\BK)} \left \langle\lambda_\BK, \mtrtwo{g}{0}{0}{1_n}.v\right \rangle\cdot \abs{\det g}_\BK^{s-\frac{1}{2}}\cdot \sgn^i_\BK(\det g)\, dg\\
&=& \frac{1}{L(s, \Pi_\BK)}\int_{\GL_n(\BK)^0} \left \langle\lambda_\BK, \mtrtwo{g}{0}{0}{1_n}.v\right \rangle\cdot \abs{\det g}_\BK^{s-\frac{1}{2}}\,dg.\\
\end{eqnarray*}
This proves the equality for this case. The remaining cases are proved by the same argument. 
\end{proof}

\begin{lemma}\label{lems2}
The functional $Z^\circ(\,\cdot \,,s, \sgn_\BK^0)|_{\tau_\BK}$ 
is independent of $s\in \BC$.
\end{lemma}

\begin{proof}
 Applying the Cartan-Helgason theorem (see \cite[Chapter V, Theorem 4.1]{He00}), we know that
$\dim \Hom_{C_\BK}(\tau_\BK\otimes \xi_{\BK,0}, \BC)=1$.
The lemma follows from Theorem \ref{thm: uniform coh test vector}.
\end{proof}

\begin{lemma}\label{lems3}
The composition functionals
\begin{equation}\label{ff11}
  \tau_\BK^+\xrightarrow{v\mapsto v\otimes 1} \tau_\BK^+\otimes \xi_{\BK,j}\subset \Pi_\BK\otimes \xi_{\BK,j}\xrightarrow
  {\epsilon_\BK^j\cdot Z^\circ(\,\cdot \,,\frac{1}{2}+j, \sgn_\BK^j)}\BC
\end{equation}
and
\begin{equation}\label{ff22}
  \tau_\BK^-\xrightarrow{v\mapsto v\otimes 1} \tau_\BK^-\otimes \xi_{\BK,j}\subset \Pi_\BK\otimes \xi_{\BK,j}\xrightarrow
  {(-\epsilon_\BK)^j\cdot Z^\circ(\,\cdot \,,\frac{1}{2}+j, \sgn_\BK^j)}\BC
\end{equation}
 are independent of the critical place $\frac{1}{2}+j$ of $\Pi_\BK$.
\end{lemma}

\begin{proof}
This directly follows from Lemmas \ref{lems1} and \ref{lems2}.
\end{proof}

Because $(\tau_\BK^+\otimes \delta_\BK^{+,\vee})\oplus (\tau_\BK^-\otimes \delta_\BK^{-, \vee})\subset\Pi_\BK\otimes F_\BK^\vee
=(\Pi_\BK\otimes \xi_{\BK,j})\otimes (F_\BK^\vee\otimes \xi_{\BK,j}^\vee)$, 
by combining Lemma \ref{lems0} and Lemma \ref{lems3}, the restriction of the functional 
\[
(\Pi_\BK\otimes \xi_{\BK,j})\otimes (F_\BK^\vee\otimes \xi_{\BK,j}^\vee)
\xrightarrow{\epsilon_\BK^j \cdot \mathrm i^{-jn} \cdot Z^\circ(\cdot,\frac{1}{2}+j,\sgn_{\BK}^j)\otimes \lambda_{F_\BK,j}} \BC
\]
to the subspace $(\tau_\BK^+\otimes \delta_\BK^{+,\vee})\oplus (\tau_\BK^-\otimes \delta_\BK^{-, \vee})$
is independent of the critical place $\frac{1}{2}+j$. 
Therefore the first assertion  of Proposition \ref{propdcoh}, together with the formulation of the archimedean modular symbol map $\CP_{\BK,j}$ in \eqref{cpbkj2}, implies that $\epsilon_\BK^j \cdot \mathrm i^{-jn}\cdot \CP_{\BK,j}$ is  independent of the critical place $\frac{1}{2}+j$. This finishes the proof of
 Theorem \ref{thm: archi period relations} in the real case.

\subsection{Proof of  Theorem \ref{thm: archi period relations} in the complex case}\label{section: comProof of Critical place and Non-vanishing Modular Symbol}

We take $\BK\cong\BC$. In this case, $\sgn_\BK$ is the trivial character of $\BK^\times$ and take $\epsilon_\BK=1$. 
The following is similar to Lemma \ref{lems0}.

\begin{lemma}\label{clems0}
The composition functional
\begin{equation}\label{cff1}
  \delta_\BK^\vee\xrightarrow{v\mapsto v\otimes 1} \delta_\BK^\vee\otimes \xi_{\BK,j}^\vee\subset F_\BK^\vee\otimes \xi_{\BK,j}^\vee\xrightarrow{\mathrm i^{-2jn}\cdot \lambda_{F_\BK,j}}\BC
\end{equation}
is independent of the critical place $\frac{1}{2}+j$ of $\Pi_\BK$.
\end{lemma}
\begin{proof}
Recall that $\Ad_{J_{-1, \BK}}(\ft'_\BK)=\ft_\BK'$ and the Borel subalgebra $\Ad_{J_{-1, \BK}}(\fb_\BK')$ is opposite to $\fb_\BK'$.
As in the proof of Lemma \ref{lems0}, we know that
$(\delta_\BK^\vee)^{(\Ad_{J_{-1,\BK}}(\fu_\BK'))\cap \fk_\BK^\BC}=\BC\cdot ( J_{-\mathrm i, \BK}. v_{F_{\BK}}^\vee)$.
The linear functional \eqref{cff1} belongs to
  $\Hom_{C_\BK}(\delta_\BK^\vee, \xi_{\BK,j})$ and sends $J_{-\mathrm i, \BK}. v_{F_{\BK}}^\vee$ to
${\mathrm i}^{-2jn}\cdot \lambda_{F_\BK,j}(J_{-\mathrm i, \BK}. v_{F_{\BK}}^\vee\otimes 1)$ which is eual to ${\mathrm i}^{-2jn}\cdot \xi_{\BK,j}^{-1}(J_{-\mathrm i,\BK})\cdot \lambda_{F_\BK,j}( v_{F_{\BK}}^\vee\otimes 1)=1$.
This implies that \eqref{cff1} is independent of the critical place $\frac{1}{2}+j$ of $\Pi_\BK$, because the character $\xi_{\BK,j}|_{C_\BK}$ is independent of $j$ and
  $(\Ad_{J_{-1, \BK}}(\fu_\BK'))\cap \fk_\BK^\BC+\fc_\BK^\BC=\fk_\BK^\BC.$
\end{proof}

\begin{lemma}\label{clems2}
The composition functionals
\begin{equation}\label{ff11}
  \tau_\BK\xrightarrow{v\mapsto v\otimes 1} \tau_\BK\otimes \xi_{\BK,j}\subset \Pi_\BK\otimes \xi_{\BK,j}\xrightarrow{Z^\circ(\,\cdot \,,\frac{1}{2}+j, \sgn_\BK^j)}\BC
\end{equation}
is independent of the critical place $\frac{1}{2}+j$ of $\Pi_\BK$.
\end{lemma}

\begin{proof}
 As an instance of  H. Schlichtkrull's
generalization of the Cartan-Helgason Theorem (\cite[Theorem 7.2]{Sch}, see also \cite[Theorem 2.1]{KnappBranchingRuleCptSymmetricSpace}), we know that
\begin{equation}\label{chc}
  \dim \Hom_{C_\BK}(\tau_\BK\otimes \xi_{\BK,0}, \BC)=1.
\end{equation}
Then it follows from Theorem \ref{thm: uniform coh test vector} that the functional
$Z^\circ(\,\cdot \,,s, \sgn_\BK)|_{\tau_\BK}$ 
is independent of $s\in \BC$. This implies the lemma.
\end{proof}

Because $\tau_\BK\otimes \delta_\BK^\vee \subset \Pi_\BK\otimes F_\BK^\vee
=(\Pi_\BK\otimes \xi_{\BK,j})\otimes (F_\BK^\vee\otimes \xi_{\BK,j}^\vee)$, 
by combining Lemma \ref{clems0} and Lemma \ref{clems2}, the restriction of the functional 
\[
(\Pi_\BK\otimes \xi_{\BK,j})\otimes (F_\BK^\vee\otimes \xi_{\BK,j}^\vee)
\xrightarrow{ \mathrm i^{-2jn} \cdot Z^\circ(\cdot,\frac{1}{2}+j,\sgn_{\BK}^j)\otimes \lambda_{F_\BK,j}}\BC
\]
to the subspace $\tau_\BK\otimes \delta_\BK^\vee$ is independent of the critical place $\frac{1}{2}+j$. This implies
 Theorem \ref{thm: archi period relations} in the complex case.

\subsection{Proof of Theorem \ref{thm: Non-vanishing hypothesis} in the complex case}\label{section: ncProof of Critical place and Non-vanishing Modular Symbol}

 In the real case, Theorem \ref{thm: Non-vanishing hypothesis} has already been proved in \cite[Theorem 5.5]{SunCohomologicalDistinguishedRepn}. We consider the case that $\BK\cong\BC$.

Recall that $\tau_{\BK,0}$ has its extreme weight
$(2n-1, 2n-3,\cdots, 3-2n, 1-2n)$ as an irreducible representation of $K_\BK$. 
The Cartan-Helgason theorem (see \cite[Chapter V, Theorem 4.1]{He00})  implies that
\begin{equation} \label{d1}
\dim (\tau_{\BK,0})^{C_\BK}=1.
\end{equation}

\begin{lemma}\label{cprv00}
There exists $\phi_0\in \Hom_{K_\BK}(\wedge^{d_\BK} \fs_\BK^\BC , \tau_{\BK,0})$ such that the composition of
\[
 \wedge^{d_\BK}\fq_\BK^\BC\hookrightarrow \wedge^{d_\BK} \fs_\BK^\BC \xrightarrow{\phi_0} \tau_{\BK,0}
\]
is nonzero. Moreover, the image of this composition map equals $(\tau_{\BK,0})^{C_\BK}$.
\end{lemma}
\begin{proof}
The proof is a modification of  that of \cite[Lemma 2.10]{Su17}. Fix a $K_\BK$-invariant positive definite Hermitian form $\langle\,,\rangle$ on $\fs_\BK^\BC$. This induces a $K_\BK$-invariant positive definite Hermitian form $\langle\,,\rangle_\wedge$ on $\wedge^{d_\BK} \fs_\BK^\BC$.

For every subspace $\fm$ of $\fg_\BK^\BC$, we define
$\fm_{\mathrm n}:=\fm/(\fm\cap (\fk_\BK^\BC\oplus \BC))\subset \fs_\BK^\BC$.
In particular, $(\fh_\BK^\BC)_{\mathrm n}=\fq_\BK^\BC$.
Write
$(\ft_\BK')_{\mathrm n}=\fa\oplus \fa'$,
where $\fa:=(\ft_\BK')_{\mathrm n}\cap (\fh_\BK^\BC)_{\mathrm n}$, and $\fa'$ is its orthogonal complement in $(\ft_\BK')_{\mathrm n}$. Then we have an orthogonal decomposition
\begin{equation}\label{dimeq0}
   \fs_\BK^\BC=\left(\fa\oplus (\fu'_\BK)_{\mathrm n}\right)\oplus \left(\fa'\oplus(\Ad_{J_{-1, \BK}}(\fu_\BK'))_{\mathrm n}\right).
\end{equation}
The second equality in \eqref{transvb} implies that
$\Ad_{J_{-1, \BK}}(\fb_\BK') \cap \fh_\BK\subset \ft_\BK'$.
This further implies that
\begin{equation}\label{dimeq1}
  \fq_\BK^\BC \cap \left(\fa'\oplus(\Ad_{J_{-1, \BK}}(\fu_\BK'))_{\mathrm n}\right)=\{0\}.
\end{equation}
Note that
\begin{equation}\label{dimeq2}
  \dim \fq_\BK^\BC=\dim \left(\fa\oplus (\fu'_\BK)_{\mathrm n}\right).
\end{equation}

Using \eqref{dimeq1}, \eqref{dimeq2} and the orthogonal decomposition \eqref{dimeq0}, we conclude that the  paring
\[
  \langle\,,\,\rangle\ \colon\ \fq_\BK^\BC\times \left(\fa\oplus (\fu'_\BK)_{\mathrm n}\right)\rightarrow \BC
\]
is non-degenerate. Therefore the one-dimensional subspaces $\wedge^{d_\BK} \fq_\BK^\BC$ and $\wedge^{d_\BK}(\fa\oplus (\fu'_\BK)_{\mathrm n})$ of $\wedge^{d_\BK}\fs_\BK^\BC$ are not orthogonal to each other. The space $\wedge^{d_\BK}(\fa\oplus (\fu'_\BK)_{\mathrm n})$ generates an irreducible $K_\BK$-subrepresentation of $\wedge^{d_\BK}\fs_\BK^\BC$, which is isomorphic to $\tau_{\BK,0}$.
We write $\tau_{\BK,0}'$ for this representation. Let $\phi_0': \fs_\BK^\BC \rightarrow \tau_{\BK,0}'$ denote the orthogonal projection. Then the
 composition of
\[
 \wedge^{d_\BK}\fq_\BK^\BC\hookrightarrow \wedge^{d_\BK} \fs_\BK^\BC \xrightarrow{\phi_0'} \tau_{\BK,0}'
\]
is nonzero. This proves the first assertion of the lemma. The second assertion then follows by \eqref{d1} and by considering the $C_\BK$-actions.
\end{proof}

\begin{lemma}\label{lcprv}
The composition of
\begin{equation}\label{cprv}
(\tau_{\BK,0})^{C_\BK}\hookrightarrow \tau_{\BK,0}\xrightarrow{\phi_1} \tau_\BK\otimes \delta_\BK^\vee=(\tau_\BK\otimes \xi_{\BK,j})\otimes (\delta_\BK^\vee\otimes \xi_{\BK,j}^\vee)\xrightarrow{Z^\circ(\,\cdot\,, \frac{1}{2}+j, \sgn_\BK^j)\otimes \lambda_{F_\BK,j}}\BC
\end{equation}
is a linear isomorphism, where $\phi_1$ is a nonzero element of $\Hom_{K_\BK}(\tau_{\BK,0}, \tau_\BK\otimes\delta_\BK^\vee)$.
\end{lemma}

\begin{proof}
Recall that the space $(\tau_{\BK,0})^{C_\BK}$ is one dimensional. Fix a generator $v_{\BK,0}$ of this space.
We only need to show that the composition of \eqref{cprv} does not vanish on $v_{\BK,0}$.

It is easy to see that this composition map equals the composition of
\begin{equation}\label{cprv2}
(\tau_{\BK,0})^{C_\BK}\hookrightarrow  \tau_{\BK,0} \xrightarrow{v\mapsto v\otimes \lambda_{F_\BK,j}} \tau_{\BK,0}\otimes \delta_{\BK}\otimes \xi_{\BK,j}\xrightarrow{\phi_1\otimes 1} \tau_\BK\otimes \xi_{\BK,j}\xrightarrow{Z^\circ(\,\cdot\,, \frac{1}{2}+j, \sgn_\BK^j)}\BC.
\end{equation}
Since $\Hom_{\BC}(\delta_\BK^\vee\otimes \xi_{\BK,j}^\vee, \BC)= \delta_{\BK}\otimes \xi_{\BK,j}$, we may view $\lambda_{F_\BK,j}$ 
as an element of $ \delta_{\BK}\otimes \xi_{\BK,j}$. Similarly, we view $\phi_1$ as an element of $\Hom_{K_\BK}(\tau_{\BK,0}\otimes \delta_\BK, \tau_\BK)$ because of the identification
\[
  \Hom_{K_\BK}(\tau_{\BK,0}\otimes \delta_\BK, \tau_\BK)=\Hom_{K_\BK}(\tau_{\BK,0}, \tau_\BK\otimes\delta_\BK^\vee).
\]
From the proof of Lemma \ref{clems0}, we know that the functional
$\lambda_{F_\BK,j}: \delta_\BK^\vee\otimes \xi_{\BK,j}^\vee\rightarrow \BC$ 
is nonzero. Hence $\lambda_{F_\BK,j}$ is nonzero as an element of $ \delta_{\BK}\otimes \xi_{\BK,j}$.
Note that the map
$\phi_1: \tau_{\BK,0}\otimes \delta_{\BK}\rightarrow \tau_\BK$
sends nonzero decomposable vectors to nonzero vectors (see \cite[Section 2.1]{Ya} and \cite[Lemma A.5]{SunCohomologicalDistinguishedRepn}). Consequently, we obtain that $(\phi_1\otimes 1)(v_{\BK,0}\otimes \lambda_{F_\BK,j})\neq 0$.
Therefore, the composition of the first three arrows in \eqref{cprv2} sends $v_{\BK,0}$ to a nonzero element of $(\tau_\BK\otimes \xi_{\BK,j})^{C_\BK}$.

By \eqref{chc} and Theorem \ref{thm: uniform coh test vector},  we know that the functional
\[
  Z^\circ(\,\cdot\,, \frac{1}{2}+j, \sgn_\BK^j): \tau_\BK\otimes \xi_{\BK,j}\rightarrow \BC
\]
is nonzero. By the $C_\BK$-invariant of this functional, it is injective  when restricted to the one dimensional space  $(\tau_\BK\otimes \xi_{\BK,j})^{C_\BK}$. Therefore the composition of \eqref{cprv2} sends $v_{\BK,0}$ to a nonzero number. This finishes the proof of the lemma.
\end{proof}

Finally, we are ready to prove Theorem \ref{thm: Non-vanishing hypothesis} in the complex case. Let $\phi_0$ be as in Lemma \ref{cprv00}, and let $\phi_1$ be as in Lemma \ref{lcprv}. Then the composition of
\[
\wedge^{d_\BK}\fq_\BK^\BC\hookrightarrow \wedge^{d_\BK} \fs_\BK^\BC \xrightarrow{\phi_0} \tau_{\BK,0}\xrightarrow{\phi_1} \tau_\BK\otimes \delta_\BK^\vee=(\tau_\BK\otimes \xi_{\BK,j})\otimes (\delta_\BK^\vee\otimes \xi_{\BK,j}^\vee)\xrightarrow{Z^\circ(\,\cdot\,, \frac{1}{2}+j, \sgn_\BK^j)\otimes \lambda_{F_\BK,j}}\BC
\]
is nonzero. Hence $\CP_{\BK,j}$ sends 
$\phi_1\circ \phi_0\in \Hom_{K_\BK}(\wedge^{d_\BK} \fs_\BK^\BC, \tau_\BK\otimes \delta_\BK^\vee)= \RH^{d_\BK}_\ct(\BR^\times_+\backslash G_\BK^0; \Pi_\BK\otimes F_\BK^\vee)$
to a nonzero element of $\Hom_{C_\BK}(\wedge^{d_\BK} \fq_\BK^\BC, \BC)= \RH^{d_\BK}_\ct(\BR^\times_+\backslash H_\BK^0; \BC)$. This finishes the proof of Theorem \ref{thm: Non-vanishing hypothesis} in the complex case.

\section{Proof of Theorem \ref{thm: uniform coh test vector}}\label{section: Identification}

Assume that $\BK$ is archimedean. Recall that $\psi_\BK:\BK\rightarrow \BC^\times$ is a non-trivial unitary  character, and $\chi_\BK: \BK^\times \rightarrow \BC^\times$ is a character. In order to prove Theorem \ref{thm: uniform coh test vector}, it is enough to prove 
the coincidence of the twisted linear functionals defined via the archimedean Friedberg-Jacquet
integrals and those constructed in  \cite{JiangLinTianExplicitCohomologicalVectorReal} and \cite{LinTianExplicitCohomologicalComplex}.
These constructions will be briefly recalled in Section \ref{ssec-FJI-LF} and then Theorem \ref{thm: uniform coh test vector} follows from Proposition \ref{thm: FJ-L} and  Lemma \ref{lem: coh repn symp type induced parameter}.

Let $\kappa:\BK^\times \rightarrow \BC^\times$ be a character.
Let $\sigma_i$ ($i=1, 2, \cdots, n$) be a generic irreducible Casselman-Wallach representation of $\GL_2(\BK)$ of central character $\kappa$. We further assume that the induced representation
 \begin{equation}\label{eq: parabolically induction pi data}
     \Xi:=    \sigma_1\dot \times\sigma_2\dot \times\cdots \dot \times\sigma_n
        \end{equation}
    of $G_\BK:=\GL_{2n}(\BK)$ is irreducible.

\subsection{Friedberg-Jacquet integrals and twisted linear functionals}\label{ssec-FJI-LF}
 By the uniqueness of local Whittaker functionals, one must have that
$\dim  \Hom_\BK(\sigma_i, \psi_\BK)=1$ for $i=1,2,\cdots, n$.
Here $\BK$ is identified with a subgroup of $\GL_2(\BK)$ via the embedding
$
x\mapsto \left[
          \begin{smallmatrix}
             1 & x \\
             0 & 1 \\
          \end{smallmatrix}
         \right].
 $
Fix a nonzero element
$\lambda_i\in \Hom_\BK(\sigma_i, \psi_\BK)$, which is a Whittaker functional and also a $(\kappa,\psi_\BK)$-Shalika functional for $\GL_2(\BK)$.

    \begin{lemma}
    The space of $(\kappa, \psi_\BK)$-Shalika functionals on the irreducible induced representation $\Xi$ of $G_\BK$ as given in 
    \eqref{eq: parabolically induction pi data} is one dimensional, that is,
    $\dim \Hom_{S_\BK}(\Xi, \kappa\otimes\psi_\BK)=1$.
    \end{lemma}
\begin{proof}
In view of the uniqueness of Shalika functionals (Lemma \ref{uniquesh}), this is precisely \cite[Theorem 2.1]{JiangLinTianExplicitCohomologicalVectorReal}.
\end{proof}

Fix a nonzero $(\kappa, \psi_\BK)$-Shalika functional
$\lambda_\Xi\in \Hom_{S_\BK}(\Xi, \kappa\otimes\psi_\BK)$.
For each $t\in \BC$, we define, as in \eqref{eq: xichi}, a character of $H_\BK$:
    \begin{equation}\label{eq222: xichi}
         \xi_{\kappa, \chi_\BK, t}:= \xi_{n,\kappa, \chi_\BK, t}:=\left( (\chi_\BK\circ\det)\cdot \abs{\det }_\BK^{t}\right)\otimes\left(((\chi_\BK^{-1}\cdot \kappa^{-1})\circ\det) \cdot \abs{\det }_\BK^{-t} \right).
    \end{equation}
Unless otherwise mentioned, we assume in the rest of this paper that $s\in \BC$ has sufficiently large real part.  Using $\lambda_\Xi$, the local Friedberg-Jacquet integral defined in \eqref{eq: F-J integral} yields a nonzero element
    \begin{equation}\label{eq: space of linear models}
   Z(\cdot,s,\chi_\BK)=Z_\Xi(\cdot,s,\chi_\BK)\in \Hom_{H_\BK}(\Xi,  \xi_{\kappa, \chi_\BK, s-\frac{1}{2}}^{-1}).
    \end{equation}
Using $\lambda_i$ and specializing  the linear functional \eqref{eq: space of linear models} to the case of $n=1$, we get the archimedean Hecke integral
 \begin{equation}\label{eqh: space of linear models}
   Z_i(\cdot,s,\chi_\BK)\in \Hom_{\BK^\times \times \BK^\times }(\sigma_i,  \xi_{1, \kappa, \chi_\BK, s-\frac{1}{2}}^{-1}).
    \end{equation}

Recall the construction of the linear functionals from \cite[Section 3.2]{JiangLinTianExplicitCohomologicalVectorReal} and \cite[Section 2.2]{LinTianExplicitCohomologicalComplex}.
Take a  permutation matrix
\[
  w_n:=[\epsilon_1\,\, \epsilon_3\, \, \cdots\,\,  \epsilon_{2n-1}\, \, \epsilon_2\,\,  \epsilon_4\,\, \cdots\,\,\epsilon_{2n}]\in G_\BK,
\]
where $\epsilon_1, \epsilon_2, \cdots, \epsilon_{2n}$ is the standard basis of $\BK^{2n}$, and $\BK^{2n}$ is viewed as a space of column vectors.
%
Define a continuous linear functional
\[
  \Lambda_{s, \chi_\BK, \Xi}' : \Xi\rightarrow \BC, \quad f\mapsto \langle \otimes_{i=1}^n Z_i(\cdot,s,\chi_\BK), f(w_n)\rangle.
\]
 Let  $B_n:=B_n(\BK)\subset \GL_n(\BK)$ denote the Borel subgroup of all the upper triangular matrices. As checked in \cite[Section 3.2]{JiangLinTianExplicitCohomologicalVectorReal} and \cite[Section 2.2]{LinTianExplicitCohomologicalComplex}, for all $(b_1,b_2)\in B_n\times B_n\subset G_\BK$, and all $f\in \Xi$, one has that
\begin{equation}\label{eql}
\langle \Lambda_{s, \chi_\BK, \Xi}', (b_1,b_2).f\rangle=\delta_{B_n}(b_1)\cdot \delta_{B_n}(b_2)\cdot \xi^{-1}_{\kappa,\chi_\BK,s-\frac{1}{2}}(b_1,b_2)\cdot \langle \Lambda_{s, \chi_\BK, \Xi}', f\rangle,
\end{equation}
where $\delta_{B_n}$ denotes the modulus character of $B_n$ that is explicitly given by
\[
  (a_{i,j})_{1\leq i,j\leq n}\mapsto \prod_{i=1}^n |a_{i,i}|_\BK^{n+1-2i}.
\]

\begin{lemma}\label{lemintg}
Let $H_0$ be a real reductive group with a maximal compact subgroup $K_0$. Let $P_0$ and $P_0^-$ be two parabolic subgroups of $H_0$ which are opposite to each other.  Write $\delta_{P_0}$ for the modulus character of $P_0$, and write $U_0^-$ for the unipotent radical of $P_0^-$. Then for every $f\in \Ind_{P_0}^{H_0} \delta_{P_0}^{\frac{1}{2}}$, the integral
$\int_{U_0^-} f(x)\,dx$ is absolutely convergent and equals $\int_{K_0} f(k)\,dk$, where $dk$ is the normalized Haar measure on $K_0$, and $dx$ is a Haar measure on $U_0^-$ which is independent of $f$. Moreover,
\[
  \Ind_{P_0}^{H_0} \delta_{P_0}^{\frac{1}{2}}\rightarrow \BC, \quad f\mapsto \int_{U_0^-} f(x)\,dx
\]
is a nonzero $H_0$-invariant continuous linear functional.

\end{lemma}

\begin{proof}
This is well known and easily follows by identifying   $\Ind_{P_0}^{H_0} \delta_{P_0}^{\frac{1}{2}}$ with the space of smooth densities on $P_0\backslash H_0$.
\end{proof}

For each positive integer $r$, write $N_r^-:=N_r^-(\BK)\subset \GL_r(\BK)$ for the subgroup of all the unipotent lower triangular matrices. It follows from \eqref{eql} that for all $f\in \Xi$,
\[
 (h\mapsto \xi_{\kappa,\chi_\BK,s-\frac{1}{2}}(h)\cdot \langle \Lambda_{s, \chi_\BK, \Xi}', h.f\rangle )\in \Ind_{B_n\times B_n}^{H_\BK} (\delta_{B_n}^{\frac{1}{2}}\otimes \delta_{B_n}^{\frac{1}{2}}).
\]
 Also note that the character $\xi_{\kappa,\chi_\BK,s-\frac{1}{2}}$ has trivial restriction to $N_n^-\times N_n^-$.
 In view of Lemma \ref{lemintg}, we now define
\begin{equation}\label{eq:LF-Lambda}
\langle \Lambda_{s, \chi_\BK, \Xi}, f\rangle:=
\int_{N_n^-\times N_n^-}  \langle \Lambda_{s, \chi_\BK, \Xi}', a.f\rangle \,da,
\end{equation}
where  $da$ is a Haar measure on $N_n^-\times N_n^-$. Then
 \begin{equation}\label{eqlf: space of linear models}
   \Lambda_{s, \chi_\BK, \Xi} \in \Hom_{H_\BK}(\Xi,  \xi_{\kappa, \chi_\BK, s-\frac{1}{2}}^{-1}).
    \end{equation}

\begin{prop}\label{thm: FJ-L}
Let the notation and the assumptions be as above. Then there is a constant $c_0\in\BC^\times $ such that
       $Z_\Xi(\,\cdot\,,s,\chi_\BK) = c_0\cdot\Lambda_{s,\chi_\BK,\Xi}$
      for all $s\in \BC$ with sufficiently large real part.
    \end{prop}

The proof of Proposition  \ref{thm: FJ-L} will be given in Sections \ref{subsection: Linear Models and Parabolic Induction}$-$\ref{subsection: Friedberg-Jacquet Integral and Parabolic Induction}.

    \begin{lemma}\label{lem: coh repn symp type induced parameter}
     Assume that the $\sigma_i$'s are as in Corollary \ref{cor: coh repn symp type induced parameter} so that $\kappa=\eta_\BK$ and $\Xi\cong\Pi_\BK$ is an irreducible essentially tempered cohomological representation of $G_\BK$   of symplectic type. Let $\tau\subset \Xi$ denote the minimal $K_\BK$-type. If
     $\Hom_{C_\BK}( \tau\otimes \xi_{\kappa, \chi_\BK,0}, \BC)\neq \{0\}$,
        then  there exists a vector $v_0\in \tau$ such that
$\langle \Lambda_{s,\chi_\BK,\Pi_\BK}, v_0\rangle=L(s, \Pi_\BK\otimes\chi_\BK)$
for all $s\in \BC$ with sufficiently large real part.
   \end{lemma}
   \begin{proof}
   This is proved in
 \cite[Theorem 5.1]{JiangLinTianExplicitCohomologicalVectorReal} for the real case and \cite[Theorem 4.3]{LinTianExplicitCohomologicalComplex} for the complex case.
 \end{proof}

Theorem \ref{thm: uniform coh test vector} then directly follows from Proposition \ref{thm: FJ-L} and  Lemma \ref{lem: coh repn symp type induced parameter}.

\subsection{Twisted linear functionals and parabolic induction}\label{subsection: Linear Models and Parabolic Induction}
Proposition \ref{thm: FJ-L} is obvious when $n=1$. In the rest of this section, we assume that $n\geq 2$.
Let $n_1,n_2$ be two positive integers such that $n_1+n_2=n$.
Let $\Xi_1$ and $\Xi_2$ be respectively irreducible  Casselman-Wallach representation of $\GL_{2n_1}(\BK)$ and $\GL_{2n_2}(\BK)$, with the same central character $\kappa$. Then we have the normalized parabolically induced representation $\Xi_1\dot \times \Xi_2$
 of $G_\BK$. Given  linear functionals
\[
\Lambda_i\in \Hom_{\GL_{n_i}(\BK)\times \GL_{n_i}(\BK)}(\Xi_i, \xi_{n_i, \kappa, \chi_\BK, s-\frac{1}{2}}^{-1}),\quad i=1,2,
\]
we will define a linear functional
$\Lambda_1\dot \times \Lambda_2\in \Hom_{H_\BK}(\Xi_1\dot \times \Xi_2, \xi_{ \kappa, \chi_\BK, s-\frac{1}{2}}^{-1})$
in what follows.

  Recall that $P_{n_1,n_2}$ denote the standard parabolic subgroup of $\GL_n(\BK)$ associated to the partition $n=n_1+n_2$. Its unipotent radical is denoted by
$N_{n_1,n_2}$, and the unipotent radical of the usual opposite parabolic subgroup  is denoted by $N_{n_1, n_2}^-$.
Define a permutation matrix
\begin{equation}\label{eq: wl1l2}
w_{n_1,n_2}
:=
\left[\begin{smallmatrix}
                                                                                   1_{n_1} &0  & 0 & 0 \\
                                                                                   0 & 0 & 1_{n_1} &0  \\
                                                                                   0 & 1_{n_2} & 0 & 0 \\
                                                                                   0 & 0 & 0 & 1_{n_2}
                                                                                 \end{smallmatrix}\right]\in G_\BK.
\end{equation}
 Define a continuous linear functional
\[
  \Lambda': \Xi_1\dot \times \Xi_2\rightarrow \BC, \quad f\mapsto \langle \Lambda_1\otimes \Lambda_2, f(w_{n_1, n_2})\rangle.
\]
It is routine to check, using a similar computation as in \cite[Subsection 3.2]{JiangLinTianExplicitCohomologicalVectorReal}, that for all $p_1, p_2\in P_{n_1,n_2}$, and all $f\in \Xi_0$,
    \begin{equation}\label{eq22: Section: Fun Eq 01}
       \left \langle  \Lambda', \mtrtwo{p_1}{0}{0}{p_2}.f  \right \rangle = \delta_{P_{n_1,n_2}}(p_1)\cdot \delta_{P_{n_1,n_2}}(p_2)\cdot  \xi_{\kappa, \chi_\BK, s-\frac{1}{2}}^{-1}(p_1,p_2)\cdot\langle \Lambda',f\rangle,
\end{equation}
where $\delta_{P_{n_1,n_2}}$ is the modulus character of $P_{n_1, n_2}$ that is explicitly given by
\[
\mtrtwo{a}{b}{0}{c}\mapsto |\det(a)|_\BK^{n_2} \cdot |\det(c)|_\BK^{-n_1}, \qquad a\in \GL_{n_1}(\BK), \, c\in \GL_{n_2}(\BK), \, b\in \Mat_{n_1\times n_2}(\BK).
\]
Here and as usual, $\Mat_{n_1\times n_2}$ indicates the space of $n_1\times n_2$ matrices.
Similar to \eqref{eq:LF-Lambda}, we define
\begin{equation}\label{eq2:LF-Lambda}
(\Lambda_1\dot \times \Lambda_2)(f):=
\int_{N^-_{n_1, n_2}\times N^-_{n_1, n_2}} \Lambda'(a.f)\, da.
\end{equation}
Then
  \begin{equation}\label{eqlf: space of linear models}
   \Lambda_1\dot \times \Lambda_2 \in \Hom_{H_\BK}(\Xi_1\dot \times \Xi_2,  \xi_{\kappa, \chi_\BK, s-\frac{1}{2}}^{-1}).
    \end{equation}

\begin{prop}\label{parli}
Suppose that $\Xi_1=\sigma_1\dot \times \sigma_2\dot \times \cdots \dot \times \sigma_{n_1}$ and $\Xi_2=\sigma_{n_1+1}\dot \times \sigma_{n_1+2}\dot \times \cdots \dot \times \sigma_{n}$. Then there   is a positive real number $c_1$ such that for all $s\in \BC$ with sufficiently large real part, the diagram
\[
\begin{CD}
          \Xi_1\dot \times \Xi_2@>f\mapsto \tilde f >> \Xi \\
            @V \Lambda_{s, \chi_\BK, \Xi_1} \dot \times  \Lambda_{s, \chi_\BK, \Xi_2}V   V           @VV\Lambda_{s, \chi_\BK, \Xi}  V\\
            \xi_{\kappa, \chi_\BK, s-\frac{1}{2}}^{-1} @>\textrm{multiplication by $c_1$}>> \xi_{\kappa, \chi_\BK, s-\frac{1}{2}}^{-1}\\
  \end{CD}
  \]
commutes, where the top horizontal arrow is the $G_\BK$-intertwining topological linear isomorphism given by
$\tilde f(g):=(f(g))(1_{2n_1}, 1_{2n_2})\in \sigma_1\widehat \otimes \sigma_2\widehat \otimes \cdots \widehat \otimes  \sigma_n$,
for $g\in G_\BK$. 

\end{prop}
\begin{proof}
Let $f\in \Xi_1\dot \times \Xi_2$. Then
\begin{eqnarray*}
 &&(\Lambda_{s, \chi_\BK, \Xi_1} \dot \times  \Lambda_{s, \chi_\BK, \Xi_2})(f)\\
 &=&
\int_{N^-_{n_1, n_2}\times N^-_{n_1, n_2}} \langle \Lambda_{s, \chi_\BK, \Xi_1} \otimes  \Lambda_{s, \chi_\BK, \Xi_2}, f(w_{n_1, n_2} a )\rangle\, da\\
&=&  \int_{N^-_{n_1, n_2}\times N^-_{n_1, n_2}} \int_{N_{n_2}^-\times N_{n_2}^-} \int_{N_{n_1}^-\times N_{n_1}^-} \left \langle \Lambda'_{s, \chi_\BK, \Xi_1} \otimes  \Lambda'_{s, \chi_\BK, \Xi_2}, \mtrtwo{b}{0}{0}{c}.(f(w_{n_1, n_2} a ))\right \rangle\,db\, dc\, da\\
&=&  \int_{N^-_{n_1, n_2}\times N^-_{n_1, n_2}} \int_{N_{n_2}^-\times N_{n_2}^-} \int_{N_{n_1}^-\times N_{n_1}^-} \left \langle\otimes_{i=1}^n Z_i(\cdot,s,\chi_\BK), (f(w_{n_1, n_2} a ))\left(\mtrtwo{w_{n_1}b}{0}{0}{w_{n_2}c}\right)\right \rangle\,db\, dc\, da\\
&=&  \int_{N^-_{n_1, n_2}\times N^-_{n_1, n_2}} \int_{N_{n_2}^-\times N_{n_2}^-} \int_{N_{n_1}^-\times N_{n_1}^-} \left \langle\otimes_{i=1}^n Z_i(\cdot,s,\chi_\BK), \tilde f\left( \mtrtwo{w_{n_1}b}{0}{0}{w_{n_2}c} w_{n_1, n_2} a\right)\right \rangle\,db\, dc\, da\\
&=&  \int_{N^-_{n_1, n_2}\times N^-_{n_1, n_2}} \int_{N_{n_2}^-\times N_{n_2}^-} \int_{N_{n_1}^-\times N_{n_1}^-} \left \langle\otimes_{i=1}^n Z_i(\cdot,s,\chi_\BK), \tilde f\left( w_n w_{n_1, n_2}^{-1}\mtrtwo{b}{0}{0}{c} w_{n_1, n_2} a\right)\right \rangle\,db\, dc\, da.\\
\end{eqnarray*}

It is easily verified that the map
\[
\begin{array}{rcl}
  (N^-_{n_1, n_2}\times N^-_{n_1, n_2})\times (N_{n_2}^-\times N_{n_2}^-)\times (N_{n_1}^-\times N_{n_1}^-)&\rightarrow & N_{n}^-\times N_{n}^-,\\
  (a, c, b)&\mapsto & w_{n_1, n_2}^{-1}\mtrtwo{b}{0}{0}{c} w_{n_1, n_2} a
\end{array}
\]
 is a well-defined  diffeomorphism and preserves the Haar measures. Therefore, there is a positive number $c_1$, which only depends on the choices of the Haar measures, such that
 \[
  c_1 \cdot (\Lambda_{s, \chi_\BK, \Xi_1} \dot \times  \Lambda_{s, \chi_\BK, \Xi_2})(f)=\int_{N_{n}^-\times N_{n}^-} \left \langle\otimes_{i=1}^n Z_i(\cdot,s,\chi_\BK), \tilde f( w_n x) \right \rangle\,dx=\Lambda_{s, \chi_\BK, \Xi}(f).
 \]

\end{proof}

\subsection{Twisted Shalika functionals and parabolic induction}


  Recall that $n_1,n_2$ are two positive integers such that $n_1+n_2=n$, and $\Xi_1$ and $\Xi_2$ are respectively irreducible  Casselman-Wallach representation of $\GL_{2n_1}(\BK)$ and $\GL_{2n_2}(\BK)$, with the same central character $\kappa$. Then we have the normalized parabolically induced representation $\Xi_1\dot \times \Xi_2$
 of $G_\BK$. Suppose that we are given   $(\kappa,\psi_\BK)$-Shalika functional $\lambda_{\Xi_i}$ on $\Xi_i$ ($i=1,2$), namely,
$\lambda_{\Xi_i}\in \Hom_{S_{n_i}(\BK)}(\Xi_i, \kappa\otimes \psi_\BK)$.


The following lemma is proved in \cite[Lemma 3.4]{AizenbudGourevitchJacquetUniquenessShalika}, which in particular implies the convergence of the Friedberg-Jacquet integrals in the archimedean case.
\begin{lemma} \label{lemagj}
Let $M_i$ ($i=1,2$) be a sufficiently large positive real number. Then for every real polynomial function $\varphi_i$ on the real vector space $\Mat_{n_i\times n_i}(\BK)$, there exists a continuous semi-norm $|\,\cdot\,|_{\varphi_i,M_i}$ on $\Xi_i$  such that
\[
  \left |\varphi_i(g)\cdot \left \langle \lambda_{\Xi_1},  \mtrtwo{g}{0}{0}{1_{n_i}}.v \right \rangle \right |\leq |\det(g)|_\BK^{-M_i}\cdot  |v|_{\varphi_i,M_i}
\]
 for all $g \in \GL_{n_i}(\BK)$ and $v\in \Xi_i$.
\end{lemma}

Lemma \ref{lemagj} for $\varphi_i=1$ easily implies the following lemma.
\begin{lemma}\label{lemagj2}
Let $M_0$ be a sufficiently large positive real number. Then  there exists a continuous semi-norm $|\,\cdot\,|_{M_0}$ on $\Xi_1\widehat \otimes \Xi_2$  such that
\[
  \left |  \left \langle \lambda_{\Xi_1}\otimes \lambda_{\Xi_2},
                                                                            \left[
                                                                                 \begin{smallmatrix}
                                                                                   1_{n_1} & 0 & 0 & 0 \\
                                                                                    0&  g_1 & 0 & 0 \\
                                                                                    0& 0  & g_2 & 0 \\
                                                                                     0& 0 & 0 & 1_{n_2} \\
                                                                                 \end{smallmatrix}
                                                                               \right]. v \\
                                                                               \right\rangle \right| \leq |\det(g_1)|_\BK^{M_0}\cdot |\det(g_2)|_\BK^{-M_0}\cdot|v|_{M_0}
\]
 for all $g_1 \in \GL_{n_1}(\BK)$, $g_2 \in \GL_{n_2}(\BK)$ and $v\in \Xi_1\widehat \otimes \Xi_2$.
\end{lemma}

Define a group
$$
N_{n_1, n_2}':=\left\{ \begin{bmatrix}
1_{n_1}&0&0&0\\
0&1_{n_2}&x&0\\
0&0&1_{n_1}&0\\
0&0&0&1_{n_2}
\end{bmatrix}\mid x\in \Mat_{n_2\times n_1}(\BK)\right\}\subset G_\BK.
$$

Write $\CS(N_{n_1, n_2}')$ for the space of complex valued Schwartz functions on $N_{n_1,n_2}'$, which is a topological vector space as usual.

\begin{lemma}\label{lemshw}
 There is a well-defined continuous linear map given by
\[
\begin{array}{rcl}
  \Xi_1\dot \times \Xi_2&\rightarrow& \CS(N_{n_1, n_2}'),\\
  f&\mapsto&(x \mapsto \langle \lambda_{\Xi_1}\otimes \lambda_{\Xi_2}, f(w_{n_1,n_2} x )\rangle).
  \end{array}
\]
\end{lemma}
\begin{proof}
    By the Iwasawa decomposition, for every $y\in\Mat_{n_2,n_1}(\BK)$, we have a unique decomposition
      \[
         \mtrtwo{1_{n_1}}{0}{y}{1_{n_2}}=\mathrm u(y)\mtrtwo{\mathrm a_1(y)}{0}{0}{\mathrm a_2(y)}\mathrm c(y)
         \]
     where $\mathrm u(y)\in N_{n_1, n_2}$,  $\mathrm a_i(y)$ is an upper triangular  matrix in $\GL_{n_i}(\BK)$ with positive real diagonal entries ($i=1,2$), and $\mathrm c(y)\in K_n$. Here $K_n$ denotes the standard maximal compact subgroup of $\GL_n(\BK)$. Note that
     $\det( \mathrm a_1(y))= \det( \mathrm a_2(y))^{-1}$.
     It is well-known and easily verified (see \cite[Lemma 2.5]{JiangLinTianExplicitCohomologicalVectorReal} for example) that 
     $\det( \mathrm a_2(y))\geq 1$ for all $y\in\Mat_{n_2\times n_1}(\BK)$,       
      and for every real valued polynomial function $\varphi$ on $\Mat_{n_2\times n_1}(\BK)$, there exists $M_\varphi>0$ such that
     \begin{equation}\label{eqnay}
     |\varphi(y)|\leq (\det(\mathrm a_2(y)))^{M_\varphi}\qquad \textrm{for all } y\in\Mat_{n_2,n_1}(\BK).
     \end{equation}

      For every $f\in\Xi_1\dot \times \Xi_2$ and $y\in\Mat_{n_2\times n_1}(\BK)$, write
      \[
        \phi_f(y):=\left \langle \lambda_{\Xi_1}\otimes \lambda_{\Xi_2}, f\left(\left[
                                                                                 \begin{smallmatrix}
                                                                                   1_{n_1} & 0 & 0 & 0 \\
                                                                                    0& 1_{n_1} & 0 & 0 \\
                                                                                    0& y  & 1_{n_2} & 0 \\
                                                                                     0& 0 & 0 & 1_{n_2} \\
                                                                                 \end{smallmatrix}
                                                                               \right]\right)\right \rangle.
      \]
      The lemma is equivalent to the assertion that
    \[
      \begin{array}{rcl}
  \Xi_1\dot \times \Xi_2&\rightarrow& \CS(\Mat_{n_2\times n_1}(\BK)),\\
  f&\mapsto&\phi_f
  \end{array}
\]
  is a well-defined continuous linear map.

 Let $\varphi$ be a real valued polynomial function  on $\Mat_{n_2\times n_1}(\BK)$ and let $M_\varphi>0$ be as in \eqref{eqnay}. Let $M_0$ and $|\,\cdot\,|_{M_0}$  be as in Lemma \ref{lemagj2} such that $2 M_0\geq M_\varphi-n_1-n_2$. Write
 \[
   |f|_{M_0,K_\BK} :=\max_{k\in K_\BK}|f(k)|_{M_0}.
 \]
 This defines a continuous seminorm on $\Xi_1\dot \times \Xi_2$.
 Then
 \begin{eqnarray}\label{abc}
   &&|  \varphi(y)\cdot\phi_f(y)|\\
  \nonumber &\leq & | (\det(\mathrm a_2(y)))^{M_\varphi}\cdot\phi_f(y)|\\
   \nonumber  &= &\left | (\det(\mathrm a_2(y)))^{M_\varphi-n_1-n_2} \cdot
                                       \left \langle \lambda_{\Xi_1}\otimes \lambda_{\Xi_2},
                                                                            \left[
                                                                                 \begin{smallmatrix}
                                                                                   1_{n_1} & 0 & 0 & 0 \\
                                                                                    0& \mathrm a_1(y) & 0 & 0 \\
                                                                                    0& 0  & \mathrm a_2(y) & 0 \\
                                                                                     0& 0 & 0 & 1_{n_2} \\
                                                                                 \end{smallmatrix}
                                                                               \right]. \left(f\left(\left[
                                                                                 \begin{smallmatrix}
                                                                                   1_{n_1} & 0 & 0  \\
                                                                                            0& \mathrm c(y)  & 0 \\
                                                                                     0& 0 &1_{n_2} \\
                                                                                 \end{smallmatrix}
                                                                               \right]\right)\right)\right\rangle \right| \\
                                                                      \nonumber            &\leq  & (\det(\mathrm a_2(y)))^{M_\varphi-2M_0-n_1-n_2} \cdot
                                        |f|_{M_0,K_\BK}  \\
                                         \nonumber &\leq  &
                                        |f|_{M_0,K_\BK}.  \\
                                        \nonumber   \end{eqnarray}

Let $X$ be a constant coefficient differential operator on  the real vector space $\Mat_{n_2\times n_1}(\BK)$. 
By definition, there exists an element $X'$ in the universal enveloping algebra of $\fg_\BK^\BC$ such that
$X\phi_f=\phi_{X'.f}$ for all $f\in \Xi_1\dot \times \Xi_2$.
By applying \eqref{abc} to $X'.f$, we have that
\[
  |  \varphi(y)\cdot (X\phi_f)(y)|\leq  |X'.f|_{M_0,K_\BK}\qquad \textrm{for all }y\in  \Mat_{n_2\times n_1}(\BK).
\]
This proves the lemma.
\end{proof}

For every $f\in \Xi_1\dot \times \Xi_2$, define
\begin{equation}\label{intx}
 \langle \lambda', f\rangle:=\int_{N_{n_1, n_2}'} \langle \lambda_{\Xi_1}\otimes \lambda_{\Xi_2}, f(w_{n_1,n_2} x )\rangle\, dx,
\end{equation}
where $dx$ is a Haar measure on $N_{n_1, n_2}'$. By Lemma \ref{lemshw},
 the integral \eqref{intx} is absolutely convergent and yields a continuous linear functional on $\Xi_1\dot \times \Xi_2$.
It is easily checked that for all
$g=\mtrtwo{h}{0}{0}{h} \mtrtwo{1_n}{a}{0}{1_n}\in S_\BK$, with $h\in P_{n_1,n_2}$,
 we have that \begin{equation}\label{Eq: Def of Shalika Functional}
       \langle \lambda', g.f\rangle=\delta_{P_{n_1,n_2}}(h)\cdot  (\kappa\otimes\psi_\BK)(g)\cdot \langle \lambda', f\rangle.
    \end{equation}
 As before, in view of  Lemma \ref{lemintg}, we define
\begin{eqnarray*}
\langle  \lambda_{\Xi_1}\dot \times \lambda_{\Xi_2}, f\rangle &:=&\int_{N_{n_1, n_2}^- } \left \langle \lambda', \mtrtwo{h}{0}{0}{h}.f \right \rangle \,dh.\\
\end{eqnarray*}
Then
$\lambda_{\Xi_1}\dot \times \lambda_{\Xi_2}\in \Hom_{S_\BK}(\Xi_1\dot \times \Xi_2, \kappa\otimes \psi_\BK)$.
It is clear that
\begin{eqnarray*}
\langle  \lambda_{\Xi_1}\dot \times \lambda_{\Xi_2}, f\rangle = \int_{\Mat_{n_1\times n_2}(\BK)}  \int_{\Mat_{n_2\times n_1}(\BK)}\left \langle \lambda_{\Xi_1}\otimes \lambda_{\Xi_2}, (w_{n_1,n_2}.f)\left(\left[
                                                                                 \begin{smallmatrix}
                                                                                   1_{n_1} & 0 & 0 & 0 \\
                                                                                    0& 1_{n_1} & 0 & 0 \\
                                                                                    y & x  & 1_{n_2} & 0 \\
                                                                                     0& y & 0 & 1_{n_2} \\
                                                                                 \end{smallmatrix}
                                                                               \right]\right)\right \rangle   \,d x\, dy,\\
\end{eqnarray*}
where $dx$ and $dy$ are suitable Haar measures. This easily implies the following lemma.
\begin{lemma}
If $\lambda_{\Xi_1}\neq 0$ and $\lambda_{\Xi_2}\neq 0$, then $\lambda_{\Xi_1}\dot \times \lambda_{\Xi_2}\neq 0$.
\end{lemma}

 \subsection{Friedberg-Jacquet integrals and parabolic induction}\label{subsection: Friedberg-Jacquet Integral and Parabolic Induction}

 Assume that $\lambda_{\Xi_1}\neq 0$ and $\lambda_{\Xi_2}\neq 0$ so that $\lambda_{\Xi_1}\dot \times \lambda_{\Xi_2}\neq 0$.
  We further assume that $\Xi_1\dot \times \Xi_2$ is irreducible.
   Fixing a Haar measure on $\GL_{n_i}(\BK)$ and using $\lambda_{\Xi_i}$ ($i=1,2$), we get the Friedberg-Jacquet integral
 \[
 Z_{\Xi_i}(\, \cdot\,,s,\chi_\BK)\in \Hom_{\GL_{n_i}(\BK)\times \GL_{n_i}(\BK)}(\Xi_i, \xi_{n_i, \kappa, \chi_\BK, s-\frac{1}{2}}^{-1})
 \]
 which is given by
 \[
        Z_{\Xi_i}(v,s,\chi_\BK) := \int_{\GL_{n_i}(\BK)} \left \langle\lambda_{\Xi_i}, \mtrtwo{g}{0}{0}{1_n}.v\right \rangle\cdot \abs{\det g}_\BK^{s-\frac{1}{2}}\cdot \chi_\BK(\det g)\,dg, \quad v\in \Xi_i.
  \]
 Likewise, fixing a Haar measure on $\GL_n(\BK)$ and using $\lambda_{\Xi_1}\dot \times \lambda_{\Xi_2}$  we get the Friedberg-Jacquet integral
 $Z(\, \cdot\,,s,\chi_\BK)\in \Hom_{H_\BK}(\Xi_1\dot\times \Xi_2, \xi_{\kappa, \chi_\BK, s-\frac{1}{2}}^{-1})$.

 \begin{prop}\label{parafj}
Let the notation and the assumptions be as above. Then there   is a positive real number $c_2$ such that for all $s\in \BC$ with sufficiently larger real part, the diagram
\[
\begin{CD}
          \Xi_1\dot \times \Xi_2@> = >> \Xi_1\dot \times \Xi_2 \\
            @V Z_{\Xi_1}(\, \cdot\,,s,\chi_\BK) \dot \times Z_{\Xi_i}(\, \cdot\,,s,\chi_\BK)V   V           @VV Z(\, \cdot\,,s,\chi_\BK)  V\\
            \xi_{\kappa, \chi_\BK, s-\frac{1}{2}}^{-1} @>\textrm{multiplication by $c_2$}>> \xi_{\kappa, \chi_\BK, s-\frac{1}{2}}^{-1}\\
  \end{CD}
  \]
commutes.

\end{prop}

By induction, Proposition \ref{parli}, Proposition \ref{parafj} and the uniqueness of twisted Shalika functionals obviously imply Proposition \ref{thm: FJ-L}, and thus we complete the proof of
Theorem \ref{thm: uniform coh test vector}. We are going to prove Proposition \ref{parafj} in the rest of this subsection.

It is clear that for all $f\in \Xi_1\dot \times \Xi_2$,
\begin{equation}\label{intfj2}
  Z(f,s,\chi_\BK):=\int_{\GL_n(\BK)}\int_{N_{n_1,n_2}^-} \left \langle \lambda', \mtrtwo{hg}{0}{0}{h}.f \right \rangle\cdot  |\det g|_\BK^{s-\frac{1}{2}} \cdot (\chi_\BK\circ\det)(g)\, dh\, dg.
\end{equation}
\begin{lemma}\label{leminffj2}
When  the real part of $s\in \BC$ is sufficiently large, the integral  \eqref{intfj2} is absolutely convergent for all $f\in \Xi_1\dot \times \Xi_2$.
\end{lemma}
\begin{proof}
   By Lemma \ref{lemintg}, for all $f\in \Xi_1\dot \times \Xi_2$,
   \[
       \int_{N_{n_1,n_2}^-} \left|\left \langle \lambda', \mtrtwo{h}{0}{0}{h}.f \right \rangle \right |\, dh
       = \int_{K_n}\left| \left \langle \lambda', \mtrtwo{k}{0}{0}{k}.f \right \rangle \cdot (\kappa\circ\det)^{-1}(k)\right|\,dk,
    \]
    where $K_n$ denotes the standard maximal compact subgroup  of $\GL_n(\BK)$, and $dk$ is a suitable Haar measure on $K_n$.
    Thus, to prove the Lemma, it suffices to prove that the following integral is absolutely convergent:
    \begin{equation}\label{eq: temp 91}
        \int_{\GL_n(\BK)}\int_{K_n}  \left \langle \lambda', \mtrtwo{kg}{0}{0}{k}.f \right \rangle  \cdot (\kappa\circ\det)^{-1}(k)\cdot  |\det g|_\BK^{s-\frac{1}{2}} \cdot (\chi_\BK\circ\det)(g)\, dk\, dg.
    \end{equation}

Using the well-known Dixmier-Malliavin Lemma (\cite{D-M}), we assume without loss of generality that
 \[
   f=\int_{N_{n,n}} \phi(x)\cdot (x.f_1)\, dx,
 \]
 where $\phi\in \CS(N_{n,n})$ and $f_1\in \Xi_1\dot \times \Xi_2$.
 It is easy to check that
    \begin{equation}\label{eq: temp 93}
        \left \langle \lambda', \mtrtwo{kg}{0}{0}{k}.f \right \rangle = \left \langle \lambda', \mtrtwo{kg}{0}{0}{k}.f_1 \right \rangle\cdot \hat{\phi}(g),
    \end{equation}
    where $\hat{\phi}$ is the Fourier transform given by
    \[
      \hat{\phi}(g):= \int_{N_{n,n}} \phi(x)\psi_\BK(\tr(gx)) dx,
    \]
    and $N_{n,n}$ is obviously identified with $\Mat_{n\times n}(\BK)$. Now the lemma easily follows, because the representation $\Xi_1\dot \times \Xi_2$ is of moderate growth.
\end{proof}

\begin{lemma}\label{lfi}
There is a positive real number $c_2$ such that
\[
\int_{N_{n_1,n_2}} \left \langle \lambda',\mtrtwo{y}{0}{0}{1_n}.f \right \rangle\, dy=c_2\cdot \langle\lambda_{\Xi_1}\otimes  \lambda_{\Xi_2}, f(w_{n_1,n_2}) \rangle
\]
for all $f\in \Xi_1\dot \times \Xi_2$.
\end{lemma}
\begin{proof}
Let  $f\in \Xi_1\dot \times \Xi_2$. For every $x\in \Mat_{n_2\times n_1}(\BK)$, write
\[
  \phi_f(x):=\left \langle \lambda_{\Xi_1}\otimes \lambda_{\Xi_2}, f\left(w_{n_1,n_2} \begin{bmatrix}
1_{n_1}&0&0&0\\
0&1_{n_2}&x&0\\
0&0&1_{n_1}&0\\
0&0&0&1_{n_2}
\end{bmatrix} \right)\right \rangle.
\]
By Lemma \ref{lemshw}, $\phi_f$ is a Schwartz function on $\Mat_{n_2\times n_1}(\BK)$.

Note that for all $x\in \Mat_{n_2\times n_1}(\BK)$ and  $y\in \Mat_{n_1\times n_2}(\BK)$, we have that
\begin{eqnarray*}
 && w_{n_1,n_2} \begin{bmatrix}
1_{n_1}&0&0&0\\
0&1_{n_2}&x&0\\
0&0&1_{n_1}&0\\
0&0&0&1_{n_2}
\end{bmatrix} \begin{bmatrix}
1_{n_1}&y&0&0\\
0&1_{n_2}&0&0\\
0&0&1_{n_1}&0\\
0&0&0&1_{n_2}
\end{bmatrix}\\
&=&\begin{bmatrix}
1_{n_1}&0&y&0\\
0&1_{n_1}&0&0\\
0&0&1_{n_2}&0\\
0&0&0&1_{n_2}
\end{bmatrix}
\begin{bmatrix}
1_{n_1}&-yx&0&0\\
0&1_{n_1}&0&0\\
0&0&1_{n_2}&0\\
0&0&0&1_{n_2}
\end{bmatrix}
w_{n_1,n_2}\begin{bmatrix}
1_{n_1}&0&0&0\\
0&1_{n_2}&x&0\\
0&0&1_{n_1}&0\\
0&0&0&1_{n_2}
\end{bmatrix} .
\end{eqnarray*}
Therefore
\begin{eqnarray*}
  &&\left \langle \lambda_{\Xi_1}\otimes \lambda_{\Xi_2}, f\left(w_{n_1,n_2} \begin{bmatrix}
1_{n_1}&0&0&0\\
0&1_{n_2}&x&0\\
0&0&1_{n_1}&0\\
0&0&0&1_{n_2}
\end{bmatrix} \begin{bmatrix}
1_{n_1}&y&0&0\\
0&1_{n_2}&0&0\\
0&0&1_{n_1}&0\\
0&0&0&1_{n_2}
\end{bmatrix}
\right)\right \rangle \\
&=&
 \psi_\BK(-\tr(yx))\left \langle \lambda_{\Xi_1}\otimes \lambda_{\Xi_2}, f\left(w_{n_1,n_2} \begin{bmatrix}
1_{n_1}&0&0&0\\
0&1_{n_2}&x&0\\
0&0&1_{n_1}&0\\
0&0&0&1_{n_2}
\end{bmatrix} \right)\right \rangle\\
&=& \psi_\BK(-\tr(yx))\phi_f(x).
\end{eqnarray*}

Using Fourier inversion formula,  this implies that
\begin{eqnarray*}\label{intfj44}
\int_{N_{n_1,n_2}} \left \langle \lambda',\mtrtwo{y}{0}{0}{1_n}.f \right \rangle\, dy
&=& \int_{\Mat_{n_1\times n_2}(\BK)}\int_{\Mat_{n_2\times n_1}(\BK)}\psi_\BK(-\tr(yx))\phi_f(x)\,dx\,dy\\
&=& c_2 \cdot \phi_f(0)\\
&=&c_2\cdot \langle\lambda_{\Xi_1}\otimes  \lambda_{\Xi_2}, f(w_{n_1,n_2}) \rangle,
\end{eqnarray*}
where $c_2$ is a positive real number which depends only on the choices of the involving Haar measures.
\end{proof}

Now we are ready to prove Proposition  \ref{parafj}.
We have that
\begin{eqnarray*}
  &&Z(f,s,\chi_\BK)\\
  &=&\int_{\GL_n(\BK)} \langle \lambda_{\Xi_1}\dot \times \lambda_{\Xi_2}, g.f\rangle \cdot  |\det g|_\BK^{s-\frac{1}{2}} \cdot (\chi_\BK\circ\det)(g)\,dg\\
   &=&\int_{\GL_n(\BK)}\int_{N_{n_1,n_2}^-} \left \langle \lambda', \mtrtwo{hg}{0}{0}{h}.f \right \rangle\cdot  |\det g|_\BK^{s-\frac{1}{2}} \cdot (\chi_\BK\circ\det)(g)\,dh\,dg \\
  &=&\int_{N_{n_1,n_2}^-} \int_{\GL_n(\BK)}\left \langle \lambda', \mtrtwo{hg}{0}{0}{h}.f \right \rangle\cdot  |\det g|_\BK^{s-\frac{1}{2}} \cdot (\chi_\BK\circ\det)(g)\,dg\,dh \quad (\textrm{by Lemma \ref{leminffj2}})\\
      &=&\int_{N_{n_1,n_2}^-} \int_{\GL_n(\BK)} \left \langle \lambda', \mtrtwo{ g}{0}{0}{h}.f \right \rangle\cdot  |\det g|_\BK^{s-\frac{1}{2}} \cdot (\chi_\BK\circ\det)(g)\,dg\, dh.\\
\end{eqnarray*}
     By decomposing the Haar measure on $\GL_n(\BK)$, the above integral equals
\begin{eqnarray*}
 &&\int_{N_{n_1,n_2}^-\times N_{n_1,n_2}^-} \int_{\GL_{n_1}(\BK)\times \GL_{n_2}(\BK)} \int_{N_{n_1,n_2}}  \left \langle \lambda', \left(\mtrtwo{ y}{0}{0}{1_n}\mtrtwo{ a}{0}{0}{1_n}z\right).f \right \rangle \cdot \delta_{P_{n_1,n_2}}(a)^{-1}\\
      &&\qquad \cdot |\det a|^{s-\frac{1}{2}} \cdot (\chi_\BK\circ\det)(a)\,dy\,da\, dz\\
       &=&c_2\cdot\int_{N_{n_1,n_2}^-\times N_{n_1,n_2}^-} \int_{\GL_{n_1}(\BK)\times \GL_{n_2}(\BK)} \left  \langle\lambda_{\Xi_1}\otimes  \lambda_{\Xi_2}, f\left(w_{n_1,n_2}\mtrtwo{ a}{0}{0}{1_n}z\right) \right \rangle  \\
      &&\qquad \qquad \qquad \qquad \cdot \delta_{P_{n_1,n_2}}(a)^{-1}\cdot |\det a|_\BK^{s-\frac{1}{2}} \cdot (\chi_\BK\circ\det)(a)\,da\, dz \quad (\textrm{by Lemma \ref{lfi}})\\
      &=&c_2\cdot\int_{N_{n_1,n_2}^-\times N_{n_1,n_2}^-} \int_{\GL_{n_2}(\BK)}\int_{\GL_{n_1}(\BK)} \left  \langle\lambda_{\Xi_1}\otimes  \lambda_{\Xi_2},  \begin{bmatrix}
a_1&0&0&0\\
0&1_{n_1}&0&0\\
0&0&a_2&0\\
0&0&0&1_{n_2}
\end{bmatrix}.(f(w_{n_1,n_2}z)) \right \rangle  \\
      &&\qquad \qquad \qquad \qquad \cdot |\det a_1|_\BK^{s-\frac{1}{2}}\cdot |\det a_2|_\BK^{s-\frac{1}{2}} \cdot (\chi_\BK\circ\det)(a_1)\cdot  (\chi_\BK\circ\det)(a_2)\,da_1\,da_2\, dz \\
      &=&c_2\cdot\int_{N_{n_1,n_2}^-\times N_{n_1,n_2}^-}
       \langle Z_{\Xi_1}(\, \cdot\,,s,\chi_\BK) \otimes Z_{\Xi_2}(\, \cdot\,,s,\chi_\BK), f(w_{n_1,n_2}z)\rangle \, dz\\
      &=&c_2 \cdot \langle Z_{\Xi_1}(\, \cdot\,,s,\chi_\BK) \dot \times Z_{\Xi_2}(\, \cdot\,,s,\chi_\BK), f\rangle.
\end{eqnarray*}
This finishes the proof of Proposition  \ref{parafj}.


\bibliographystyle{elsarticle-num}
    
\end{document}